\pdfoutput=1
\documentclass[11pt]{article}
\usepackage{authblk}
\usepackage[toc,page]{appendix}

\usepackage{hyperref}
\hypersetup{
    colorlinks=true,
    linkcolor=blue,
    citecolor=red,
    urlcolor=blue,
    pdfborder={0 0 0}
}

\usepackage{color}
\usepackage{helvet}         
\usepackage{courier}        
\usepackage{type1cm}        
\usepackage{framed}
\usepackage{tikz}
\usepackage{makeidx}         
\usepackage{graphicx}        
\usepackage{multicol}        
\usepackage[bottom]{footmisc}
\usepackage{amsmath}
\usepackage{amssymb}
\usepackage{bbold}
\usepackage{amsthm}
\usepackage{subcaption}
\usepackage{sidecap}
\usepackage{floatrow}
\usepackage{pdflscape}
\usepackage{bm}
\usepackage{comment}
\usepackage[font=small]{caption}
\usepackage[top=2cm, bottom=2cm, left=2cm, right=2cm]{geometry}
\newtheorem{theorem}{Theorem}

\newtheorem{lemma}[theorem]{Lemma}

\newtheorem{proposition}[theorem]{Proposition}

\newtheorem{remark}[theorem]{Remark}
\newtheorem{definition}[theorem]{Definition}
\numberwithin{theorem}{section}
\numberwithin{figure}{section}
\numberwithin{equation}{section}

\usepackage{enumitem}

\DeclareMathOperator{\CR}{CR}

\DeclareMathOperator{\dist}{dist}
\DeclareMathOperator{\SLE}{SLE}

\DeclareMathOperator{\diam}{diam}

\DeclareMathOperator{\Leb}{Leb}

\newcommand{\eps}{\epsilon}

\newcommand{\T}{\mathbb{T}}

\newcommand{\LA}{\mathcal{A}}

\newcommand{\LE}{\mathcal{E}}
\newcommand{\LG}{\mathcal{G}}
\newcommand{\LL}{\mathcal{L}}

\newcommand{\LN}{\mathcal{N}}
\newcommand{\LQ}{\mathcal{Q}}

\newcommand{\LR}{\mathcal{R}}
\newcommand{\LT}{\mathcal{T}}

\newcommand{\LV}{\mathcal{V}}
\newcommand{\LX}{\mathcal{X}}
\newcommand{\LZ}{\mathcal{Z}}

\newcommand{\R}{\mathbb{R}}
\newcommand{\C}{\mathbb{C}}
\newcommand{\D}{\mathbb{D}}
\newcommand{\UC}{\mathbb{S}}
\newcommand{\N}{\mathbb{N}}
\newcommand{\Z}{\mathbb{Z}}
\newcommand{\E}{\mathbb{E}}
\newcommand{\PP}{\mathbb{P}}

\newcommand{\A}{\mathbb{A}}
\newcommand{\one}{\mathbb{1}}

\newcommand{\cond}{\,|\,}
\newcommand{\Bigcond}{\;\Big|\;}

\renewcommand{\Im}{\mathrm{Im}}
\renewcommand{\Re}{\mathrm{Re}}
\newcommand{\ii}{\mathrm{i}}

\def\BdryIn{\partial_{\scalebox{0.6}{\textnormal{in}}}\Omega}
\def\BdryOut{\partial_{\scalebox{0.6}{\textnormal{out}}}\Omega}
\def\OmegaFill{\Omega_{\scalebox{0.6}{\textnormal{fill}}}}
\def\OmegaFillCl{\Omega_{\scalebox{0.6}{\textnormal{fill}}}}
\def\OmegaUniv{\Omega_{\scalebox{0.6}{\textnormal{univ}}}}
\def\LVUniv{\LV_{\scalebox{0.6}{\textnormal{univ}}}}
\def\LEUniv{\LE_{\scalebox{0.6}{\textnormal{univ}}}}

\def\AnnBdryIn{\partial_{\scalebox{0.6}{\textnormal{in}}}\A}
\def\AnnBdryOut{\partial_{\scalebox{0.6}{\textnormal{out}}}\A}

\def\BdryOutUniv{\partial_{\scalebox{0.6}{\textnormal{out}}}\OmegaUniv}

\def\Pois{\mathrm{H}^{\textnormal{ND}}}
\def\Gren{\mathrm{G}}

\def\HarmUniv{h^{\scalebox{0.6}{\textnormal{univ}}}}

\def\phiUniv{\phi^{\scalebox{0.6}{\textnormal{univ}}}}

\def\hUniv{h^{\scalebox{0.6}{\textnormal{univ}}}}

\newcommand{\BdryInNbhdDiscr}[1]{\LV^\delta_{\scalebox{0.6}{\textnormal{in}}}{(#1)}}
\newcommand{\BdryInNbhd}[1]{\Omega_{\scalebox{0.6}{\textnormal{in}}}{(#1)}}

\newcommand{\BdryInNbhdDiscrGr}[1]{\Omega^\delta_{\scalebox{0.6}{\textnormal{in}}}{(#1)}}
\newcommand{\BdryInNbhdDiscrDom}[1]{\tilde\Omega^\delta_{\scalebox{0.6}{\textnormal{in}}}{(#1)}}

\def\BdryOutV{\partial_{\scalebox{0.6}{\textnormal{out}}}\LV}
\def\BdryInV{\partial_{\scalebox{0.6}{\textnormal{in}}}\LV}

\def\BdryOutEUniv{\partial_{\scalebox{0.6}{\textnormal{out}}}\LEUniv}
\def\BdryOutVUniv{\partial_{\scalebox{0.6}{\textnormal{out}}}\LVUniv}

\def\root{\varrho}

\global\long\def\bs{\boldsymbol}
\global\long\def\ud{\mathrm{d}}

\newcommand{\cev}[1]{\reflectbox{\ensuremath{\vec{\reflectbox{\ensuremath{#1}}}}}{}}
\global\long\def\field{\Phi}
\global\long\def\fline{\vartheta}
\global\long\def\gff{\mathrm{h}}
\global\long\def\slecurv{\varpi}
\global\long\def\simpleCurv{\varrho}

\global\long\def\harm{\,\mathfrak{h}}
\global\long\def\ee{\mathrm{e}}
\def\exitT{\tau}
\def\exitTOut{\tau_{\scalebox{0.6}{\textnormal{out}}}}
\def\exitTIn{\tau_{\scalebox{0.6}{\textnormal{in}}}}
\def\exitTT{\tau}

\def\Prob{\mathsf{P}^\delta}
\def\ProbE{\mathsf{E}^\delta}
\def\ProbHat{\hat{\mathsf{P}}^\delta}
\def\ProbEHat{\hat{\mathsf{E}}^\delta}

\def\ProbBM{\mathsf{P}}
\def\ProbEBM{\mathsf{E}}

\global\long\def\TVnorm#1{\big|\big| \, #1 \, \big|\big|_{\rm TV}}
\global\long\def\TVnormInline#1{|| \, #1 \, ||_{\rm TV}}

\def\ph{\varphi}

\begin{document}

\title{Uniform spanning trees and random matrix statistics}

\author{Nathana\"el Berestycki\thanks{University of Vienna. nathanael.berestycki@univie.ac.at}, \, 
Marcin Lis\thanks{Vienna University of Technology. marcin.lis@tuwien.ac.at}, \, 
Mingchang Liu\thanks{Capital Normal University, Beijing. liumc\_prob@163.com}, \, 
and Eveliina Peltola\thanks{Aalto University and University of Bonn. eveliina.peltola@hcm.uni-bonn.de}}

\date{}

\maketitle{
}

\vspace*{-5mm}

\begin{abstract}
We consider a uniform spanning tree in a $\delta$-square grid approximation of a planar domain $\Omega$. For given integer $n\ge 2$, we condition the tree on the following $n$-arm event: we pick $n$ branches, emanating from $n$ points microscopically close to a given interior point, and condition them to connect to the boundary $\partial \Omega$ without intersecting. What can be said about the geometry of these branches? 

We derive an exact formula for the characteristic function of the total winding of the branches, involving the Brownian loop measure in the scaling limit $\delta \to 0$. A surprising consequence of this formula is that in the scaling limit, the behaviour of this function depends on the total number of branches $n$ only through its parity.

We also describe the scaling limit of the branches. If $\Omega = \D$ is the unit disc, then they hit the boundary of $\D$ (i.e., the unit circle) at random positions which coincide exactly with the eigenvalues of a random matrix of size $n$ drawn from the Circular Orthogonal Ensemble (COE, also called C$\beta$E with $\beta =1$). Furthermore, the branches converge to Loewner evolution driven by the circular Dyson Brownian motion with parameter $\beta = 4$ (i.e., $n$-sided radial $\SLE_\kappa$ with $\kappa=2$). We thus verify a prediction made by Cardy in this setting. 

Along the way, we develop a flow-line (imaginary geometry) coupling of $n$-sided radial $\SLE_\kappa$ with the Gaussian free field, which may be of independent interest. Surprisingly, we find that the variance of the corresponding field near the singularity also does not depend on the number $n\ge 2$ of curves. In contrast, the variance of the the \emph{winding} of the curves behaves as $\kappa/n^2$, which agrees with the predictions from the physics literature made by Wieland and Wilson numerically, and by Duplantier and Binder using Coulomb gas methods --- but disagrees with a result of Kenyon. 
\end{abstract}


%

\vspace*{2mm}

\begin{center}
\includegraphics[width=.45\textwidth]{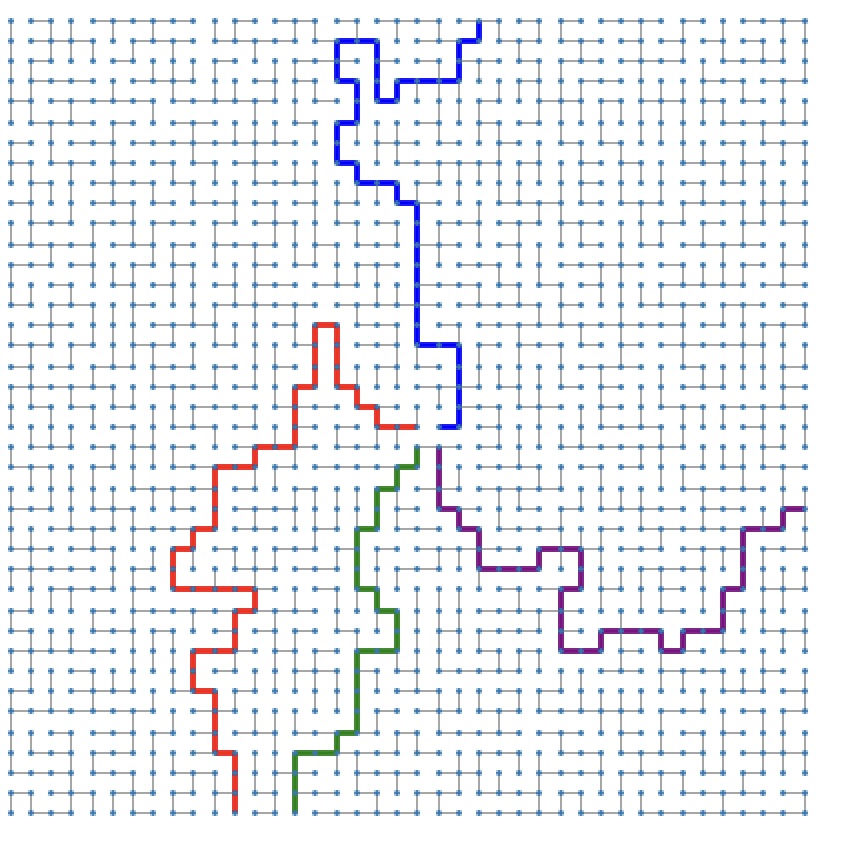}
\\ 
UST 
on a $40\times40$ square, conditioned so that four given neighbouring branches do not intersect.
\end{center}

\newpage
\tableofcontents

\section{Introduction and main results}
\label{sec::main}

It is well known that the scaling limit of a \textbf{uniform spanning tree} (UST) in two dimensions 
almost surely contains exceptional \emph{triple points} (points which, if removed, produce three connected components) but no points of higher multiplicity~\cite{Benjamini:Large_scale_degrees_and_the_number_of_spanning_clusters_for_UST}.  
In this article, we consider the question of conditioning a uniform spanning tree on the event that a given interior point, say $z \in \Omega$, has a given arbitrary (fixed) multiplicity, $n\ge 2$.  
This is a degenerate conditioning if $n\ge 3$, which needs to be carefully defined.  
The main goal of this work is to answer the following question: what can be said about the geometry of these curves?  
For ease, we will consider USTs on finite subgraphs of the square lattice $\delta \Z^2$ with mesh size $\delta>0$, for which the conditioning can be defined relatively simply (see below for precise definitions). 

\medskip 
We find two types of results.  
On the one hand, we obtain an exact formula, valid at the lattice level (for fixed $\delta>0$), for the characteristic function of the \textbf{total winding} (i.e., algebraic total number of crossings by the $n$ branches of a fixed reference line); see Theorem~\ref{thm::total_winding}.  
This formula can be analysed in the scaling limit where the mesh size 
$\delta$ tends to zero, in terms of a certain Brownian loop measure. 
A surprising consequence of our formula is that the total winding has a limit law which is independent of $n$. 

\medskip 
On the other hand, we fully describe the scaling limit of the law of these $n$ curves (with $n$ fixed, while $\delta \to 0$), 
and reveal a connection to objects appearing naturally in the context of \textbf{random matrix theory} 
(specifically, the Circular $\beta$-Ensemble, C$\beta$E with $\beta =1$ and $\beta =4$), 
which correspond to predictions originally made by Cardy~\cite{Cardy:SLE_and_Dyson_circular_ensembles};
see Theorem~\ref{thm::conv_curves}. 
The former describes the density of the hitting points of the curves to the outer boundary, while the latter describes the the Loewner evolution of the curves themselves (as a  variant of the Schramm--Loewner evolution process, SLE, with $\kappa=2$).
In addition, we describe the topological winding of the curves, generalising Schramm's result for a single radial SLE;
see Proposition~\ref{coro::truncated_winding}. 
We also derive a coupling of these curves with a variant of the Gaussian free field,
which fit into the imaginary geometry framework \`a la Miller~\&~Sheffield; see Theorem~\ref{coro::truncated_winding}.

\subsection{Statement of the problem}
\label{subsec:setup}

Let us first define more precisely the conditioning that we have in mind.  
Suppose that $\Omega \subsetneq \C$ is a doubly connected domain, i.e., 
$\Omega$ is conformally equivalent to an annulus, so $\partial\Omega$ has two connected components.  
Denote by $\BdryOut$ the boundary of the unbounded connected component of $\C\setminus\Omega$ and define $\BdryIn := \partial\Omega\setminus\BdryOut$.  
We will suppose that $\BdryIn$ is analytic and $\BdryOut$ is at least a Jordan curve; the inside of this Jordan curve (i.e., the unique bounded component of $\C \setminus \BdryOut)$ is a simply connected domain which we denote by $\OmegaFill$.  
We may assume without loss of generality that $0$ is separated from $\infty$ by $\BdryIn$.

Suppose $\{\Omega^\delta = (\LV^\delta, \LE^\delta)\}_{\delta>0}$ is a sequence of finite graphs such that 
$\Omega^\delta\subset\delta\Z^2$ and $\Omega^\delta$ is doubly connected for every $\delta>0$.  
Define $\BdryOut^\delta$ and $\BdryIn^\delta$ similarly as above.  
Using the usual planar self-duality of the square lattice, each vertex $v\in \LV^\delta$ can be identified with a square $f(v)$ of sidelength $\delta$ centred at $v$.
To the graph $\Omega^\delta$, we can then associate an open set $\tilde \Omega^\delta \subset \R^2 \approx \C$ given by $\tilde \Omega^\delta := \textnormal{Int} (\smash{\underset{v\in \LV^\delta}{\cup}} \, \overline{ f(v)})$.  

\medskip 
\textbf{Setup.}
In the scaling limit results, we assume that $\Omega^\delta$ converges to $\Omega$ as $\delta \to 0$ in the following sense:
\begin{enumerate}[leftmargin=*]
\item \label{setup1}
There exists a conformal isomorphism $\phi = \phi_r$ from $\Omega = \Omega_r$ with some $r = r_\Omega \in(0,1)$ 
onto an annulus $\A_r = \{ z \in \C \,|\, r<|z|<1 \}$ 
and a sequence $(\phi^\delta)_{\delta>0}$ of conformal isomorphisms from $\tilde \Omega^\delta$ 
onto $\A_{r^\delta}$ such that $r^\delta\to r$ as $\delta\to 0$ and $(\phi^\delta)^{-1}$ converges to $\phi^{-1}$ locally uniformly on $\A_r$ as $\delta\to 0$. 
\item \label{setup2}
Using the conventional metric $\dist_X(\cdot,\,\cdot)$ of unparametrised curves defined in Equation~\eqref{eqn::curve_metric} (which is the metric of uniform convergence up to reparametrisation),  
we have $\dist_X (\BdryOut^\delta,\BdryOut)\to 0$ as $\delta\to 0$, 
and there exists a constant $C = C_\Omega>0$ such that $\dist_X (\BdryIn^\delta,\BdryIn)\le C\delta$ for every $\delta>0$.
\end{enumerate}

\begin{figure}
\includegraphics[scale=.5]{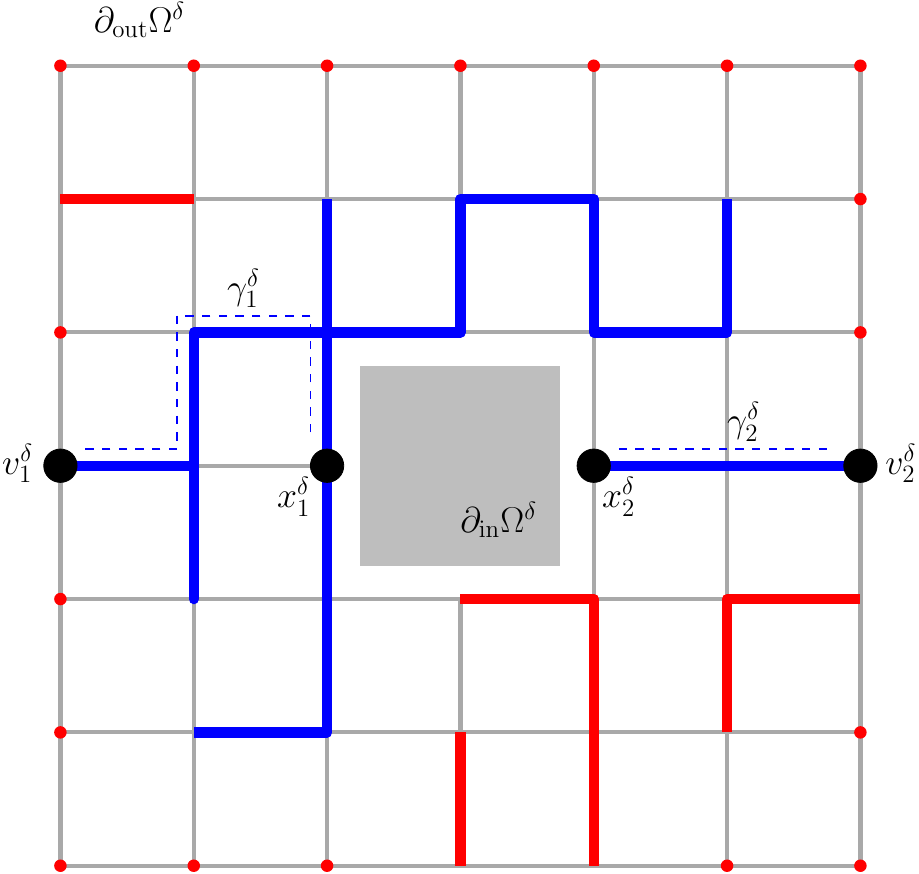} \hspace{1cm}
\includegraphics[scale=.5]{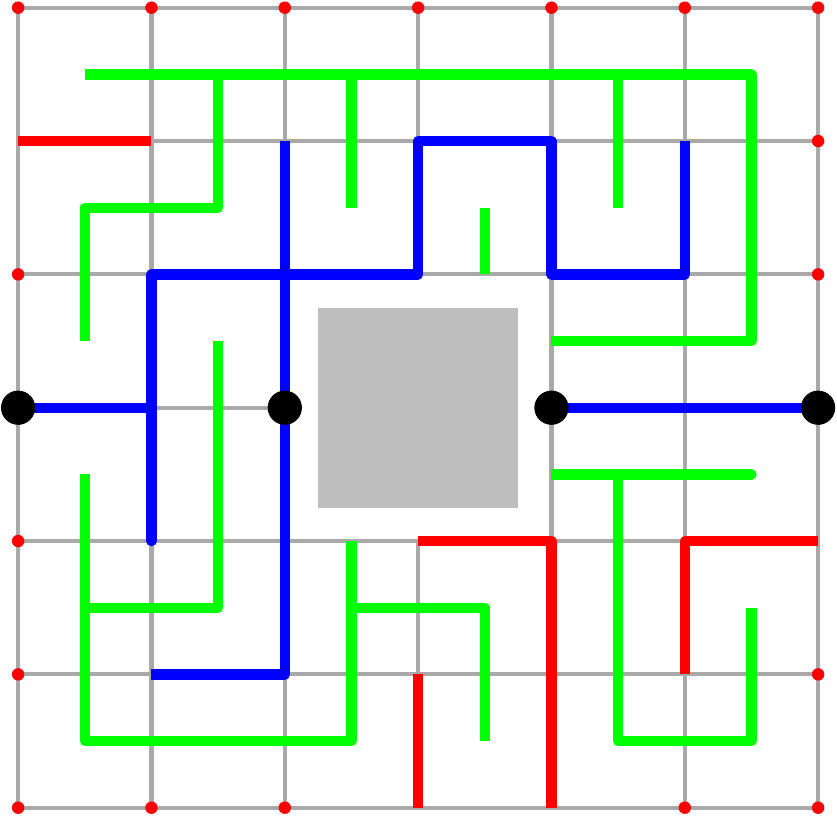}
\caption{Uniform spanning tree on the graph $\Omega^\delta$ (gray), with wired (resp.~free) boundary conditions on $\BdryOut^\delta$ (resp.~$\BdryIn^\delta$), conditioned on the event $E^\delta_{\bs{x}}$, case $n=2$.  
The components containing the special points $x_1^\delta$ and $x_2^\delta$ are in blue. Other components are in red. 
All components of the tree are oriented towards the wired boundary $\BdryOut^\delta$. 
The two disjoint oriented paths $\gamma_1^\delta$ and $\gamma_2^\delta$, from $x_1^\delta$ and $x_2^\delta$ respectively to $\BdryOut^\delta$, have been highlighted in dashed blue. 
In the right panel, we superimpose the corresponding dual spanning tree on the planar dual lattice (in green). Its components (three in this example) are naturally oriented towards the inner boundary.
}
\label{F:UST}
\end{figure}

\medskip 
We fix $n$ distinct vertices $x_1^\delta,\ldots,x_n^\delta$ on the inner boundary $\BdryIn^\delta$, by convention labelled in counterclockwise order, 
and assume that each $x^\delta_j$ converges to $x_j\in\BdryIn$, where $x_1,\ldots,x_n$ are $n$ distinct points on $\BdryIn$.
We will consider the uniform spanning tree $\LT^\delta$ on $\Omega^\delta$ with 
wired\footnote{For simplicity, we will not distinguish the graph $\Omega^\delta$ and the graph obtained from $\Omega^\delta$ by wiring the outer boundary $\BdryOut^\delta$ to a root vertex $\root^\delta$. We assume these boundary conditions throughout.} 
boundary condition on $\BdryOut^\delta$ and free boundary condition on $\BdryIn^\delta$, 
i.e., we view the outer boundary $\BdryOut^\delta$ as a single vertex $\root^\delta$ declared to be the root of the UST. 
Observe that $\LT^\delta$ can be viewed as an \emph{arborescence}, i.e., an acyclic collection of oriented edges in $\Omega^\delta$ such that every vertex $x \in \Omega^\delta$ distinct from the root has a unique outgoing edge.
For $1\le j\le n$, define $\gamma^\delta_j = \vec\gamma^\delta_j$ to be the \textbf{UST branch} in $\LT^\delta$ connecting $x_j^\delta$ to the root vertex $\root^\delta$, obtained by following the outgoing edges emanating from $x_j^\delta$ inductively, and which necessarily terminates at the root $\root^\delta$.  
We primarily view $\gamma_j^\delta$ as an unordered collection of vertices, $\gamma_j^\delta = \{x_j^\delta, \ldots, \root^\delta\}$, though sometimes it will be useful to identify it with the ordered collection $\vec\gamma^\delta_j$ of edges it successively traverses, and sometimes as a continuous curve up to reparametrisation ending at $\BdryOut^\delta$. 
With slight abuse of notation, we usually denote the path as $\gamma_j^\delta$; we highlight the orientation as $\vec\gamma^\delta_j$ and its reverse as $\cev{\gamma}^\delta_j$ only when necessary.   

\medskip 
Let us now define the event we are interested in conditioning on --- see Figure~\ref{F:UST} for an illustration.
For $\bs{x}^\delta = (x_1^\delta,\ldots,x_n^\delta)$ as above, we define 
\begin{align}\label{eqn::disjointE}
E^\delta_{\bs{x}} 
= E(\Omega^\delta;x_1^\delta,\ldots,x_n^\delta) 
:= \big\{\LT^\delta \colon \gamma^\delta_i  \cap \gamma^\delta_j   = \{ \root^\delta\} \textnormal{ for all }1\le i<j\le n \big\},
\end{align}
which is the event that the branches emanating from $x_1^\delta, \ldots, x_n^\delta$ are pairwise disjoint; i.e., $\gamma^\delta_i$ and $\gamma^\delta_j$ for $i \neq j$ do not have any common vertex except for the root vertex $\root^\delta$. 
Events of this type feature prominently in the theory of SLEs, since they determine the crossing and/or arm exponents, whose computation was one of the major motivation and outstanding success of the theory~\cite{LSW:Brownian_intersection_exponents1, LSW:Brownian_intersection_exponents2, LSW:Brownian_intersection_exponents3}.  
See also~\cite{Kenyon:Long-range_properties_of_spanning_trees} and the physics 
literature~\cite{Duplantier-Binder:Harmonic_measure_and_winding_of_conformally_invariant_curves, Wieland-Wilson:Winding_angle_variance_of_Fortuin-Kasteleyn_contours, Duplantier:Conformal_random_geometry} 
(to which we will come back later in Remark~\ref{rem:rel_to_lit}) 
more specifically for computations precisely in this case of the uniform spanning tree. 

\medskip 
In the present work, we are interested in letting the mesh size $\delta$ tend to zero and in letting the inner boundary $\BdryIn$ shrink to a point, meaning that $r \to 0$. 
This can be done simultaneously or in different orders, and either way we do not expect this choice to impact the answers (we will partly justify this with our results).  
For technical reasons, it will turn out to be more convenient to first let the mesh size $\delta \to 0$ (keeping $\Omega$ and in particular $\BdryIn$ fixed) and subsequently letting $\diam(\BdryIn) \to 0$;  
though some of our results (in particular concerning the winding of the curves conditional on the event $E^\delta_{\bs{x}}$) will also be valid independently of the order in which these limits are taken.

\subsection{Total winding}

A question which is particularly natural concerns the winding of the UST branches near their common endpoint. 
To address this, we begin with some generalities about different notions of winding for curves, and an intuitive heuristics that guides the subsequent results.  
If $\gamma \colon [a,b] \to \C$ is a smooth simple curve with $\gamma'(t) \neq 0$ for $t\in [a , b]$, the (intrinsic) winding of $\gamma$ is, by definition, the quantity $\int_a^b \arg \gamma'(t) \, \ud t$.  
This quantity extends\footnote{Note that the intrinsic winding is ill-defined for fractal paths, for which we use the topological winding instead. 
They can be related in the piecewise smooth case; see~\cite[Section~2]{BLR:Dimers_and_imaginary_geometry}.} 
without any difficulty to piecewise smooth paths such as the lattice paths of interest, $\bs{\gamma}^\delta = (\gamma_1^\delta, \ldots, \gamma_n^\delta)$. 
In this case, the winding of such a path also coincides with the sum of turning angles.

\paragraph{Heuristics.} 
Fix $k\ge 1$, and suppose $r \ll \ee^{-k}$. We consider the $n$ branches between scales $\ee^{- k }$ and $\ee^{ - k -1}$, and more specifically, the winding of these branches at this scale, 
say in the sense of the algebraic number of full turns around the annulus. 
For topological reasons, as the curves do not intersect, each branch accumulates the same winding between these two scales.  
Also, scale-invariance suggests that the law of the winding accumulated between scales $\ee^{-k}$ and $\ee^{-k-1}$ does not depend on $k$, and it seems reasonable to guess that these increments are also roughly independent as $k$ varies. 
It is therefore natural to expect that the winding up to scale $\ee^{-k}$ evolves like a Brownian motion as a function of the logarithm $k$ of the scale, but with a variance that needs to be determined 
--- and should in particular depend in an interesting way on the number $n\ge 2$ of branches.  
Intuitively, the larger $n$ is, the more rigid is the system, and thus the smaller we should expect the variance to be.  
We will see that \emph{the variance behaves quadratically in $n$}:
\begin{align}\label{eq:varianceheuristics}
\textnormal{variance of winding} \sim \frac{2}{n^2} ,
\end{align}
up to any given scale, for each curve.
(Here, the constant $2 = \kappa$ is the variance of the corresponding $\SLE_\kappa$ driving function, also discussed below. In general, the variance of the winding turns out to be $\sim \kappa/n^2$.)

\paragraph{An exact law.} 
To state our exact result concerning total winding, it is convenient to consider a slightly more restricted event where we impose the locations at which 
the branches $\bs{\gamma}^\delta = (\gamma_1^\delta, \ldots, \gamma_n^\delta)$ hit the outer boundary. 
Let $v_1^\delta,\ldots,v_n^\delta$ be $n$ distinct vertices on $\BdryOut^\delta$, by convention labelled in counterclockwise order.
We use ``$\cdot \longleftrightarrow \cdot$'' to indicate that there is a connection in the tree $\LT^\delta$.
Thus, bearing in mind that $\LT^\delta$ is an arborescence oriented towards $\BdryOut^\delta$, 
the event $\{x_i^\delta \longleftrightarrow v_j^\delta\}$ means that, following successively the unique oriented branch $\vec \gamma_i^\delta$ started at $x_i^\delta$, one eventually arrives at the vertex $v_j^\delta$ at time $\exitT_j^\delta$.

Using the convention that $v_{j}^\delta = v_{j-n}^\delta$ for $j>n$, we now define
\begin{align}
\nonumber
E^\delta_{\bs{x}, \bs{v}, k} 
:= \; &
\big\{ x_1^\delta  \longleftrightarrow v_{k+1}^\delta, \, \ldots, \, x_n^\delta \longleftrightarrow v_{n+k}^\delta \big\}
\\
= \; & \big\{\LT^\delta \; | \; \vec\gamma^\delta_1(\exitT_1^\delta) = v_{k+1}^\delta, \, \ldots, \, 
\vec\gamma_n(\exitT_n^\delta) = v_{k+n}^\delta\big\} 
, \qquad \textnormal{for } 0\le k\le n-1 ,
\label{eqn::defE}
\\
\label{eqn::defE2}
E^\delta_{\bs{x}, \bs{v}} 
= \; & E(\Omega^\delta;x_1^\delta,\ldots,x_n^\delta;v_1^\delta,\ldots,v_n^\delta) 
:= \bigcup_{k=0}^{n-1}E^\delta_{\bs{x}, \bs{v}, k}.
\end{align}
Due to planarity, the event $E^\delta_{\bs{x}, \bs{v}}$ is exactly the event that $E^\delta_{\bs{x}}$ occurs and $\{\{x_1^\delta, \ldots, x_n^\delta\} \longleftrightarrow \{v_1^\delta, \ldots, v_n^\delta\}\}$: 
indeed, if the latter occurs, then necessarily these connections must respect the counterclockwise order.

\medskip
Let $K(\gamma_j^\delta)$ be the (signed) number of crossings by $\gamma_j^\delta$ across a given line (``zipper'')  connecting the two boundary components of the annulus (see Figure~\ref{F:dimers} for an illustration).  
Our result, Theorem~\ref{thm::total_winding} below, gives an exact identity for the law of the {\textbf total winding}\footnote{This crossing number differs from the actual intrinsic winding, normalised by $2\pi$, by at most $\pm 1$; see Equation~\eqref{eq:crossings}.} 
$K(\bs{\gamma}^\delta) := \sum_{j=1}^n 2\pi K(\gamma_j^\delta)$.  

\medskip
Our identity takes a different form according to whether $n$ is odd or even, but involves either way the (discrete) \textbf{random walk loop soup} on $\Omega^\delta$ (see Section~\ref{subsec:Fomin} and~\cite{Lawler-Ferreras:RW_loop_soup} for definitions, 
bearing in mind that random walks on $\Omega^\delta$ are only reflected at $\BdryIn^\delta$ but absorbed at $\BdryOut^\delta$).
Denote by $\LL_*^\delta$ the set of (discrete) uncontractible loops on $\Omega^\delta$ (i.e., loops not homotopic to a point) which do not hit $\BdryOut^\delta$.  
To each such loop $\ell\in\LL_*^\delta$, one can associate a quantity $\ph(\ell)$ equal to $2\pi$ times the signed number of turns of $\ell$ around the origin --- in other words, $\ph(\ell)$ is the \textbf{topological winding} of $\ell$ around the origin.
Denote by $\PP_{\bs{x}, \bs{v}}^\delta$ the conditional probability measure of $\LT^\delta$ given the event $E^\delta_{\bs{x}, \bs{v}}$ defined in~\eqref{eqn::defE2}. 

\medskip
In Section~\ref{sec::annulus}, we will also show that as $\delta \to 0$, the restriction of the discrete loop measure to $\LL_*^\delta$ converges to the analogous quantity for a \textbf{Brownian loop measure} $\mu^{\textnormal{ND}}_{\Omega}$ on $\Omega$, with Neumann boundary condition on $\BdryIn$ and Dirichlet boundary condition on $\BdryOut$ 
(see Section~\ref{subsec:Fomin} and~\cite{LSW:Conformal_restriction_the_chordal_case, Lawler-Werner:The_Brownian_loop_soup} for definitions).  
This does not directly follow from known results such as~\cite{Lawler-Ferreras:RW_loop_soup} owing to the presence of the reflecting boundary. 
Denote by $\LL_*$ the set of (continuous) loops on $\Omega$ and for each such loop $\ell\in\LL_*$, define $\ph(\ell)$ similarly as before.  
Using this result combined with an exact asymptotic computation of determinants involving random walk excursion kernels, 
we are able to deduce the following result on the law of the total winding of curves. 

\medskip
Recall from the scaling limit Setup (Item~\ref{setup1}) that we denote by $\phi$ a fixed conformal map from $\Omega$ onto the annulus $\A_r$. Choosing the branch of the argument function $\arg$ to take values in $[-\pi,\pi)$, set
\begin{align*}
c := \sum_{j=1}^{n}\arg \phi(x_j)-\sum_{j=1}^{n}\arg \phi(e_j) .
\end{align*}

\begin{theorem}\label{thm::total_winding}
Let $x_1^\delta,\ldots,x_n^\delta \in \BdryIn^\delta$ and $v_1^\delta,\ldots,v_n^\delta \in \BdryOut^\delta$ be distinct vertices labelled in counterclockwise order along the respective boundaries. 
Assume that as $\delta \to 0$, each $x^\delta_j$ converges to $x_j\in\BdryIn$ and each $v^\delta_j$ converges to $v_j\in\BdryIn$,
where $x_1,\ldots,x_n \in \BdryIn$ and $v_1,\ldots,v_n \in\BdryOut$ are distinct and in counterclockwise order.
For any odd $n\ge 1$, we have the following two cases.
\begin{enumerate}
\item If $\beta \in [0,\frac{1}{2})\cup(\frac{1}{2},1)$, 
and $N_\beta\in\Z$ is such that $|\beta-N_\beta| = \underset{k\in\Z}\min \, |\beta-k|$, then we have
\begin{align}\label{eqn::scaling_wind_odd_limit}
\lim_{\delta\to 0} \E_{\bs{x}, \bs{v}}^\delta \Big[\ee^{\ii \beta \sum_{j=1}^{n}2\pi K(\gamma_j^\delta)}\Big]
=  \big( \ee^{\ii(\beta-N_\beta)c} + o_r(1)\big) \; r^{|\beta-N_\beta|}  \; \exp \Big(\mu_\Omega^{\textnormal{ND}}\big[(1-\ee^{\ii\beta\ph(\ell)}) \, \one{\{\ell\in\LL_*\}}\big] \Big).
\end{align}

\item If $\beta=\frac{1}{2}$, then we have
\begin{align}\label{eqn::scaling_wind_odd_limit1}
\lim_{\delta\to 0} \E_{\bs{x}, \bs{v}}^\delta \Big[\ee^{\ii \beta \sum_{j=1}^{n}2\pi K(\gamma_j^\delta)}\Big]
= \big( 2\cos(c/2)+o_r(1) \big) \; r^{{1}/{2}}  \; \exp \Big(\mu_\Omega^{\textnormal{ND}}\big[(1-\ee^{\ii\beta\ph(\ell)}) \, \one{\{\ell\in\LL_*\}}\big] \Big).
\end{align}
\end{enumerate}
\noindent 
For any even $n\ge 2$, denoting $N_\beta\in\Z$ such that $|\beta-N_\beta-\frac{1}{2}| = \underset{k\in\Z}\min \,|\beta-k-\frac{1}{2}|$, we have
\begin{align}\label{eqn::scaling_wind_even_limit}
\lim_{\delta\to 0} \E_{\bs{x}, \bs{v}}^\delta \Big[\ee^{\ii \beta \sum_{j=1}^{n}2\pi K(\gamma_j^\delta)}\Big]
= \big(\ee^{\ii(\beta-N_\beta+\frac{1}{2})c}+o_r(1)\big) \; \exp \Big(\mu_\Omega^{\textnormal{ND}}\big[ \ee^{\frac{\ii}{2}\ph(\ell)} (1-\ee^{\ii\beta\ph(\ell)}) \, \one{\{\ell\in\LL_*\}}\big] \Big).
\end{align}
The term $o_r(1)$ is uniform over the locations of $x_1,\ldots,x_n \in \BdryIn$ and $v_1,\ldots,v_n \in \BdryOut$.
\end{theorem}

A remarkable aspect Theorem~\ref{thm::total_winding} is that, setting aside the dependence of the winding on the initial and end position of the $n$ curves (captured by the term $c$), 
the behaviour of the right-hand side as the inner boundary shrinks, meaning that $r \to 0$, depends on $n$ \emph{only through its parity}. 
The difference between the odd and even cases comes from certain determinantal identities (whose source can be traced to Fomin's formula for random walks~\cite{Fomin:LERW_and_total_positivity}, although the edge weights are complex numbers instead of being positive in our case), and a careful analysis of the signs of the permutations which may occur in the expansion of the determinant.
The proof of Theorem~\ref{thm::total_winding} comprises most of Section~\ref{sec::annulus}.

\medskip 
In the limit as $\delta \to 0$ and then $r \to 0$ (in this order),
a more precise rigorous derivation of the asymptotic variance of the total winding $K(\bs{\gamma}^\delta)$ as
claimed in~\eqref{eq:varianceheuristics} follows from the identification of the scaling limit of the UST branches $\gamma_j^\delta$ as $\delta \to 0$, which we now proceed to describe.

\subsection{Convergence to random matrix statistics}
\label{subsec:curves_intro}

We now consider a collection $\bs{\eta} = (\eta_1,\ldots,\eta_n)$ of $n$ random curves in the unit disc $\D = \{ z \in \C \,|\, |z|<1 \}$ starting at some boundary points on the unit circle $\UC = \{ z \in \C \,|\, |z|=1 \}$.
As we will later see, these curves will correspond to the scaling limit of $\bs{\gamma}^\delta = (\gamma^\delta_1, \ldots, \gamma^\delta_n)$ (conditioned on the event $E^\delta_{\bs{x}}$).
First, let $\T^n = (\R/2\pi\Z)^n$ be the $n$-torus with periodic boundary conditions, 
and let $\T_0^n$ denote the subset of elements admitting representatives 
$\bs\theta = (\theta_1, \ldots, \theta_n)$ satisfying $\theta_1<\theta_2< \cdots< \theta_n <\theta_1+2\pi$.
Define
\begin{align*}
\LX_n := \big\{ \ee^{\ii \bs{\theta}} = (\ee^{\ii\theta_1},\ldots,\ee^{\ii\theta_n}) \in \UC^n \cond \bs{\theta} \in \T_0^n \big\}.
\end{align*}
Let the density of the starting points of the curves $\bs{\eta}$ on $\UC^n$ be given by
\begin{align}\label{eqn::density}
\begin{split}
\rho(\ee^{\ii \bs{\theta}}) 
:= \; & \frac{1}{\LZ_n}\one{\{\ee^{\ii \bs{\theta}}\in\LX_n\}} 
\prod_{1\le j<k\le n}\big|\ee^{\ii\theta_j}-\ee^{\ii\theta_k}\big|\prod_{j=1}^{n}\ud\theta_j ,
\\
\textnormal{where} \quad 
\LZ_n := \; & \int_{\UC^n}\one{\{\ee^{\ii \bs{\theta}}\in\LX_n\}} 
\prod_{1\le j<k\le n}\big|\ee^{\ii\theta_j}-\ee^{\ii\theta_k}\big|\prod_{j=1}^{n}\ud\theta_j .
\end{split}
\end{align}
Note that~\eqref{eqn::density} coincides with the eigenvalue distribution of an $(n\times n)$-random matrix from the classical \textbf{Circular Orthogonal Ensemble} (COE, also known as C$\beta$E for $\beta =1$)~\cite{Dyson:Statistical_theory_of_the_energy_levels_of_complex_systems1, Dyson:Algebraic_structure_of_symmetry_groups_and_ensembles_in_quantum_mechanics, Forrester:Log_gases_and_random_matrices}.

\medskip 
Given these starting points, let the conditional law of $\bs{\eta}$ be that of $n$-sided radial $\SLE_2$ started from $\ee^{\ii \bs{\theta}}$. 
In other words, given $\ee^{\ii \bs{\theta}}$, the curves $\bs{\eta}$ can be generated by a radial Loewner evolution
driven by \textbf{Dyson Brownian motion} associated to the Circular $\beta$-Ensemble (C$\beta$E) 
with parameter
$\beta = 4$: 
\begin{align}\label{eqn:DBM}
\ud\Theta_i(t) = \sqrt 2 \, \ud W_i(t) + 2 \sum_{\substack{1\leq j\leq n\\j\neq i}} \cot\bigg(\frac{\Theta_i(t) - \Theta_j(t)}{2}\bigg)\ud t, \qquad 1 \leq i \leq n ,
\end{align}
where $(W_i(t) \colon t\ge 0)$ for $1\le i\le n$ are $n$ independent Brownian motions on $\R$ with $W_i(0) = \theta_i$ for each $i$.

\medskip 
As mentioned, $\bs{\eta}$ is also known as $n$-sided radial SLE (see Section~\ref{sec::scaling}).
The relationship of Dyson Brownian motion and SLE theory was first observed by John Cardy~\cite{Cardy:SLE_and_Dyson_circular_ensembles},
who used conformal field theory techniques to show how radial $\SLE_\kappa$ curves can be generated from Dyson Brownian motion with a judicious choice of the parameter $\beta$ (namely $\beta = 8/\kappa$). 
Later, $n$-sided radial SLE was rigorously constructed in~\cite{Healey-Lawler:N_sided_radial_SLE}. 
The key difficulty is that the configurational measure approach to build SLE variants by tilting fails due to the divergence of the Brownian loop measure term at the common endpoint of the curves; 
which could also be observed from our formulas in Theorem~\ref{thm::total_winding}. 
(See also~\cite{AHP:Large_deviations_of_DBM_and_multiradial_SLE} for a large deviations principle as $\kappa \to 0$, 
and references for connections to integrable systems.)

\medskip 
There are several natural ways to parametrise the curves $\bs{\eta}$, corresponding to the speed at
which the different curves are discovered (e.g., one can discover all the curves at once, or one at a time).
The above description corresponds to what is called the \textbf{common parametrisation}~\cite{Healey-Lawler:N_sided_radial_SLE}, 
where the capacity (i.e., the negative log-conformal radius) seen from the origin in the multi-slit disc\footnote{Here and throughout, abusing notation and with no danger of confusion, we identify the curves $(\eta_1[0, t], \ldots, \eta_n[0,t])$ with the closed hull $\bs{\eta}[0, t] := \cup_j \eta_j[0, t]$ of $\overline{\D}$ that they generate.} 
$\D\setminus \bs{\eta}[0, t]$ equals $nt$. 
In this vein, for $1\le i \le n$, let $\cev{\gamma}_i^\delta$ denote the time-reversal of the branch $\vec{\gamma}_i^\delta$, viewed as a (linearly interpolated and hence continuous) curve from $\BdryOut^\delta$ to $\BdryIn^\delta$.  
We reparametrise $(\cev{\gamma}_1^\delta, \ldots, \cev{\gamma}_n^\delta)$ to obtain curves $\bs{\eta}^\delta = (\eta_1^\delta, \ldots, \eta_n^\delta)$ such that the common parametrisation condition holds:
\begin{align*}
-\log \CR (0; \D\setminus \bs{\eta}^\delta[0, t] ) = nt.
\end{align*}

\medskip 
To state our scaling limit result, which roughly speaking says that the discrete branches $\bs{\eta}^\delta$ converge to the above SLE variant, let us describe in more detail the framework we use for this convergence. 
Let $(\Omega_r)_{0\le r \le 1}$ denote a family of doubly connected domains as in the above Setup (Item~\ref{setup1}).
For the sake of explaining the idea, we will assume (although that is not really necessary)
that $\BdryOut_r$ does not depend on $r$, and $\textnormal{diam} ( \BdryIn_r) \to 0$ as $r \to 0$, 
so in fact $\BdryIn_r$ converges with respect to the Hausdorff distance to $0 \in \OmegaFill$. 
In this way, the conformal isomorphism $\phi_r \colon \Omega_r \to \A_r$ can simply be taken to be the restriction to $\Omega_r$ of the fixed conformal isomorphism 
$\phi \colon \Omega \to \D$ with $\phi(0)=0$ and $\phi'(0)>0$. 

\begin{theorem}\label{thm::conv_curves}
Conditionally on the event $E^\delta_{\bs{x}}$ defined in~\eqref{eqn::disjointE}, we have 
\begin{align*}
(\phi(\eta_1^\delta),\ldots,\phi(\eta_n^\delta)) \longrightarrow (\eta_1,\ldots,\eta_n)
\end{align*}
in law, uniformly on compact intervals of time as $ \delta \to 0$ and then $r \to 0$.
\end{theorem}

This result establishes the convergence of the UST branches on the event $E^\delta_{\bs{x}}$ 
to a multiple radial Loewner evolution driven by the circular Dyson Brownian motion with $\beta = 4$ as in~\eqref{eqn:DBM}, 
thus verifying Cardy's prediction~\cite{Cardy:SLE_and_Dyson_circular_ensembles}. 
The general relation $\beta = 8/\kappa$ also explains why we should not be surprised to obtain C$\beta$E with $\beta = 1$ in the hitting distribution~\eqref{eqn::density} --- indeed, in that case the relevant value of $\kappa$ is that of the associated UST Peano curves, meaning $\kappa = 8$~\cite{Schramm:Scaling_limits_of_LERW_and_UST, LSW:Conformal_invariance_of_planar_LERW_and_UST}. 

\medskip 
An interesting feature of Theorem~\ref{thm::conv_curves} is that 
in particular, the $n$ positions where the curves $\bs{\eta}^\delta$ hit the outer boundary $\BdryOut^\delta$ of the domain 
converge in law as $\delta \to 0$ and then $r \to 0$ to the COE distribution obtained as a pullback of~\eqref{eqn::density} on the configuration space $\LX_n$ by $\phi$.  
This particular fact was conjectured by Arista and O'Connell~\cite{Arista-OConnell:Loop-erased_walks_and_random_matrices} and already proved by Kenyon in the cases $n =2, 3$ in~\cite{Kenyon:Long-range_properties_of_spanning_trees}. 

\medskip 
One might suspect that a first step to describe the geometry of the branches $\bs{\eta}^\delta$ will be to obtain sharp asymptotics on the partition function $\PP^\delta[E^\delta_{\bs{x}}]$. 
This is only partly true, in that we only need to determine the asymptotics of the partition function when $n$ is odd. (In fact, in the proof of Proposition~\ref{prop::hittingpoints} in Section~\ref{sec::det} we obtain a much more precise result, including the determination of the leading constant.)

\begin{theorem}\label{thm::partfun}
For the event $E^\delta_{\bs{x}}$ defined in~\eqref{eqn::disjointE} with $\bs{x}^\delta = (x_1^\delta,\ldots,x_n^\delta)$ on the inner boundary 
and the domain $\Omega_r$ conformally equivalent to the annulus $\A_r$, 
when $n\ge 3$ is odd, we have
\begin{align*}
\PP^\delta[E^\delta_{\bs{x}}] = r^{\frac{n^2-1}{4} \,+\, o(1)} ,
\qquad \textnormal{where $o(1) \to 0$ as $\delta \to 0$ followed by $r \to 0$.}
\end{align*}
\end{theorem}

\medskip 
Theorem~\ref{thm::conv_curves} also gives us access to the winding of the curves $\bs{\eta}$ 
in the common parametrisation. 
This appears, in fact, relatively easy thanks to the explicit form of the SDE system~\eqref{eqn:DBM} and the close connection between topological winding and Loewner driving function. 
Let $\ph_i(t) := \arg \eta_i(t) - \arg \eta_i(0)$ denote the \textbf{topological winding} around the origin of $\eta_i[0,t]$
(which measures the change in the argument in some branch choice; see~\cite{BLR:Dimers_and_imaginary_geometry} for the relation between topological and intrinsic winding).
The next result is a multi-curve version of Schramm's result~\cite[Theorem~7.2]{Schramm:Scaling_limits_of_LERW_and_UST}.
 
\begin{proposition} \label{coro::truncated_winding}
Let $X\sim\LN(0,1)$. The following convergence holds in law as $t\to\infty$\textnormal{:}
\begin{align*}
\Big(\frac{\ph_1(t)}{\sqrt{2t/n}},\ldots, \frac{\ph_n(t)}{\sqrt{2t/n}}\Big) \longrightarrow (X,\ldots,X) .
\end{align*}
\end{proposition}

Taking into account the time-parametrisation, note that 
Proposition~\ref{coro::truncated_winding} can be viewed as a rigorous version of the winding variance~\eqref{eq:varianceheuristics}.  
The result is stated for convenience in the case of $\kappa=2$, which is of interest to us in the context of the uniform spanning tree. 
However, we have written the proof in fact for general $n$-sided radial $\SLE_\kappa$ in Section~\ref{subsec:truncated_winding}: 
in this case, the same result holds, but the term $\sqrt{2/n}$ should be replaced by $\sqrt{\kappa/n}$, 
yielding the variance $\sim \kappa/n^2$ as predicted in~\cite{Wieland-Wilson:Winding_angle_variance_of_Fortuin-Kasteleyn_contours}. 

\subsection{Flow-line coupling for $n$-sided radial SLE with GFF}

To the pair of uniform spanning trees defined by $\LT^\delta$ and its dual one can associate, via a form of Temperley's bijection~\cite{KPW:Trees_and_matchings, Dubedat:SLE_and_free_field, BLR:Note_on_dimers_and_T-graphs}, 
a dimer configuration living on the medial graph (Figure~\ref{F:dimers}).  
The dimer model has a natural height function, which converges to the Gaussian free field (GFF)~\cite{Kenyon:Conformal_invariance_of_domino_tiling, Kenyon:Dominos_and_the_Gaussian_free_field}. 
As we will not rely on the dimer interpretation, 
we do not give precise definitions here (which could, however, be inferred from Figure~\ref{F:dimers}; see also~\cite{BLR:Dimers_on_Riemann_surfaces_1} for a closely related combinatorial setup and bijection, and~\cite{BLR:Dimers_on_Riemann_surfaces_2} for a scaling limit result).

\medskip 
Given the close relations between dimers and imaginary geometry, it is natural to wonder whether the $n$-sided curves $\bs{\eta} = (\eta_1, \ldots, \eta_n)$ can be coupled with a GFF-like object as ``flow lines''. 
In the result below, we derive such a coupling (not only for $\kappa =2$ but for arbitrary $\kappa \in (0,4)$). 
The theory of flow lines for the GFF was extensively developed by Miller and Sheffield~\cite{Miller-Sheffield:Imaginary_geometry1, Miller-Sheffield:Imaginary_geometry4}, who coined the term ``imaginary geometry'',
and earlier by Dub\'edat~\cite{Dubedat:SLE_and_free_field} who focused on the role of partition functions.
The series of works~\cite{BLR:Dimers_and_imaginary_geometry, BLR:Dimers_on_Riemann_surfaces_2, BLR:Dimers_on_Riemann_surfaces_1} 
addresses the case of present interest. 

\begin{figure}
\includegraphics[scale=.5]{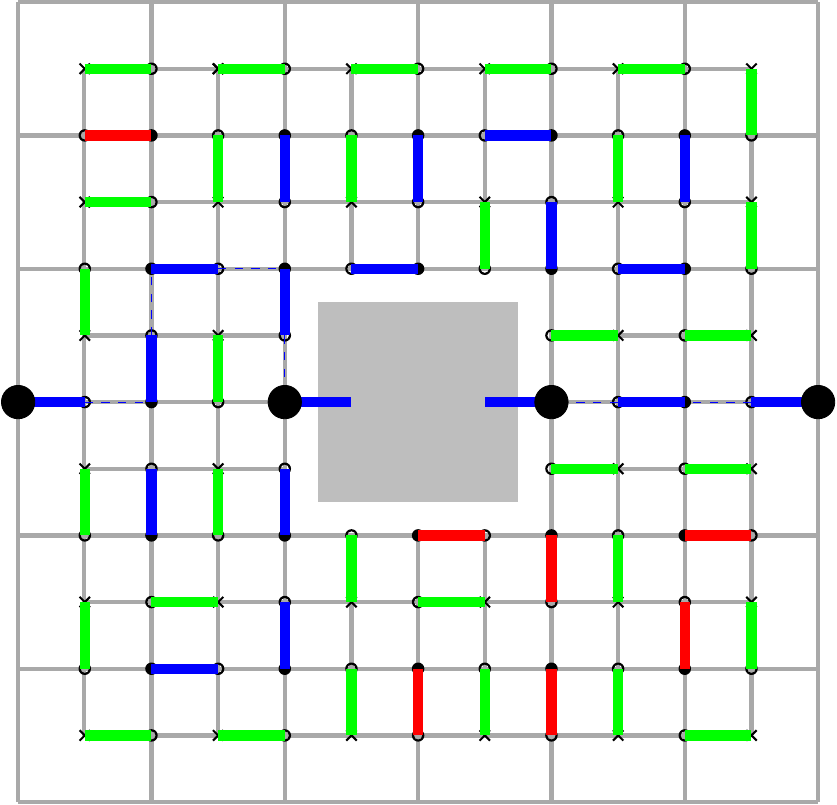} \hspace{1cm} 
\includegraphics[scale=.5]{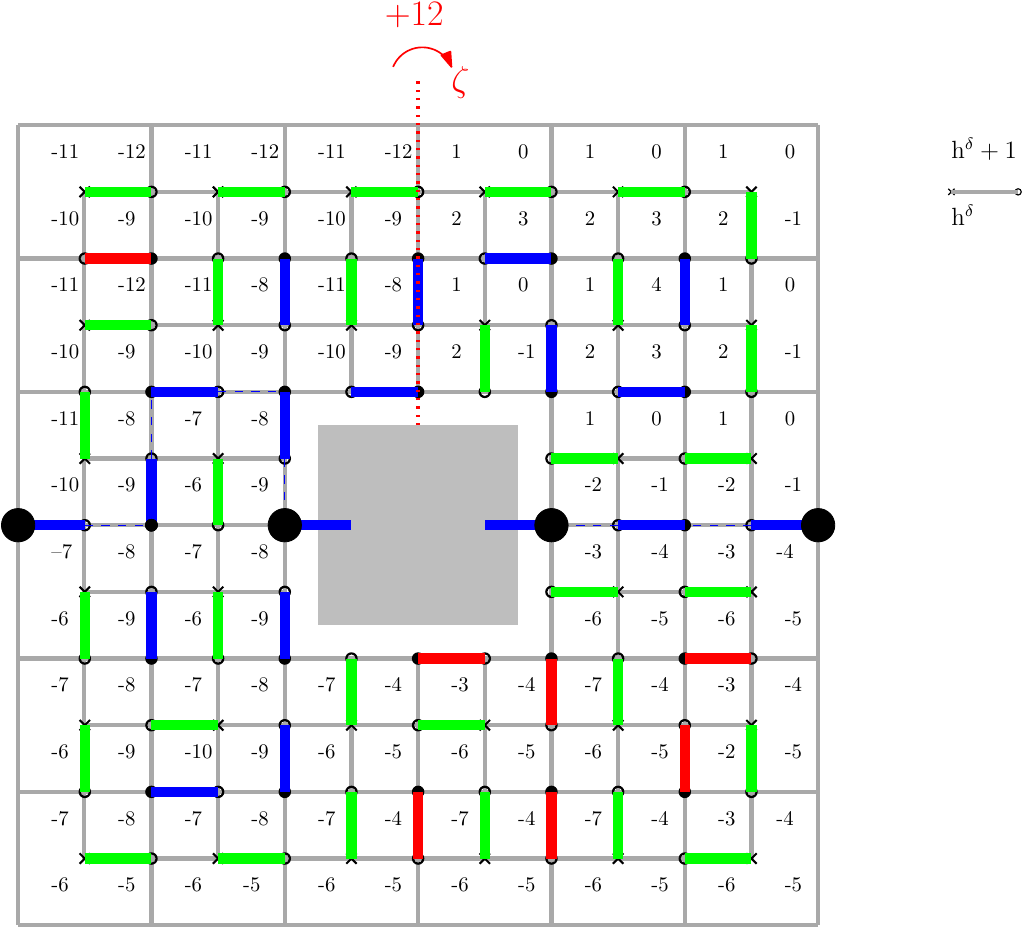} 
\caption{The dimer configuration associated to $\LT^\delta$ and its dual. 
The underlying tree and color coding are identical to those in Figure~\ref{F:UST}.  
In the right panel, we show the associated height function. The conventions for the height are indicated at the top and right of this picture (black vertices are either crosses or black dots; all other vertices are white). 
By convention, the top right corner has height zero and the height increases by $\pm 1$ across edges without dimers (and by $\pm 3$ otherwise). The height function is only well defined after introducing a 
branch cut $\zeta$ (in red), which entails an additional jump of $n+1$ (times 4 with this normalisation, so $+12$ in this example).
}
\label{F:dimers}
\end{figure}

\begin{theorem}\label{thm::coupling}
Fix $\kappa \in (0, 4)$ and $\ee^{\ii\bs{\theta}}\in\LX_n$.
Let $\chi_\kappa:=\frac{2}{\sqrt\kappa}-\frac{\sqrt\kappa}{2}$ and $\lambda_\kappa:=\frac{\pi}{\sqrt\kappa}$.
Consider the harmonic function $\harm \colon \D \to \R$ defined for $w \in \D$ as
\begin{align}\label{eqn::bound}
\harm(w) = \harm_{\bs{\theta}}(w) 
:= -\frac{1}{\sqrt \kappa} \sum_{j=1}^n\arg \bigg(\frac{\ee^{\ii\theta_j}-w}{1-\overline{w}\ee^{\ii\theta_j}}\bigg) 
= -\frac{2}{\sqrt\kappa} \sum_{j=1}^n\arg(\ee^{\ii\theta_j}-w) + \frac{1}{\sqrt\kappa}\sum_{j=1}^n\theta_j.
\end{align}
Let $\gff$ be a Gaussian free field with Dirichlet boundary conditions in $\D$, 
and $\bs{\fline} = (\fline_1,\ldots,\fline_n)$ flow lines of 
\begin{align}\label{eq:field}
\field(\cdot) := \gff(\cdot) + \Big(\chi_\kappa+\frac{n}{\sqrt\kappa}\Big)\arg(\cdot)+\harm_{\bs{\theta}}(\cdot)
\end{align}
starting from $\ee^{\ii\bs{\theta}}$ and with respective angles $\alpha_j = \frac{2\lambda_\kappa (j-1)}{\chi_\kappa}$ for $1 \leq j \leq n$.
After parameterising the curves by their common parametrisation, the law of $\bs{\fline}$ 
agrees with that of the $n$-sided radial $\SLE_\kappa$ curves $\bs{\eta}$. 

\medskip 
In other words, there exists a coupling between $\field$ and $\bs{\eta}$ 
such that for all stopping times $\tau$,  
conditionally given $\bs{\eta}[0, \tau] = (\eta_1[0, \tau], \ldots, \eta_n[0, \tau])$, we have
\begin{align}\label{eq:IG}
\field|_{\D_{\tau}} = \tilde \field \circ g_{\tau} - \chi_\kappa \arg g'_{\tau},  
\end{align}
where $\tilde \field$ has the same law as~\eqref{eq:field}, and $g_{\tau} \colon \D_{\tau} \to \D$ is the unique conformal isomorphism uniformising the complement of $\D_{\tau} := \D \setminus \bs{\eta}[0, \tau]$ 
normalised such that $g_{\tau}(0) = 0$ and $g'_{\tau} (0) >0$. 
\end{theorem}

\begin{remark}\textnormal{
We now make a few remarks about the theorem above. 
\begin{itemize}[leftmargin=2em]
\item In common with the literature on imaginary geometry~\textnormal{\cite{Miller-Sheffield:Imaginary_geometry1, Miller-Sheffield:Imaginary_geometry4}} 
and with~\textnormal{\cite{BLR:Dimers_and_imaginary_geometry}}, 
our convention here is to normalise $\gff$ so that $\E [\gff(z) \gff(w)] = - \log |z-w| + O(1)$ as $|z-w| \to 0$. 
\item In Equation~\eqref{eqn::bound}, the function $\arg(\ee^{\ii\theta_j}-\cdot)$ \textnormal{(}and thus the function $\harm$\textnormal{)} 
is a well-defined \textnormal{(}single-valued\textnormal{)} function on $\D$, 
since we can use a branch cut for $\arg(\ee^{\ii\theta_j}-\cdot)$ which remains outside of $\D$. 
\item However, $z\mapsto \arg z$ cannot be defined as a single-valued function on $\D$.
We may choose, for instance, $\arg z \in [\pi/2,5\pi/2)$ with discontinuity in the positive imaginary axis; note the analogy with Figure~\ref{F:dimers}. 
\item Theorem~\ref{thm::coupling} says that the curves $\bs{\eta}$ are flow lines of $\gff + (n+1)\chi_\kappa  \arg + \harm$.  
Using the methods in~\textnormal{\cite{BLR:Dimers_and_imaginary_geometry}}, 
one can in fact show as a corollary of Theorem~\ref{thm::coupling} that the limiting height function 
\textnormal{(}with respect to the winding flow\textnormal{)} 
of the corresponding dimer model is $\gff/\chi_\kappa  + (n+1) \arg + \harm / \chi_\kappa$. 
\end{itemize}
}
\end{remark}

\subsection{A paradox and its resolution}

One of the most interesting aspects of Theorem~\ref{thm::coupling} is that the field $\gff$ in~\eqref{eq:field} is just a standard Gaussian free field: it does not have any unusual singularity at the origin 
(even though the Loewner evolution of the curves themselves is highly singular near the origin). 
In particular, if $\gff_\eps$ denotes the circle average of $\gff$ at distance $\eps>0$, then \emph{independently} of $n$, its variance reads
\begin{align*}
\textnormal{var} ( (1/\chi_\kappa)\gff_\eps (0)) = 2\log (1/\eps) .
\end{align*}
However, in Temperley's bijection, the winding of curves corresponds to the height function, so the fact that the leading-order coefficient in this variance is independent of $n$  appears to be in contradiction with Proposition~\ref{coro::truncated_winding}.  Indeed, both Proposition~\ref{coro::truncated_winding} and its more precise version, Theorem~\ref{thm::total_winding}, indicate that the winding of curves has a variance of order $(2/n^2) \log (1/\eps)$, rather than $2 \log (1/\eps)$. In other words, in both Proposition~\ref{coro::truncated_winding} and Theorem~\ref{thm::total_winding}, the law of the winding of each curve is affected by the number of curves --- which makes sense, because of the rigidity discussed above~\eqref{eq:varianceheuristics} (the more curves, the more rigid the system --- hence, the less the winding of each curve).  
But the variance of the associated \emph{height function} near the origin does not depend on $n$, 
even though (via Temperley's bijection) it is supposed to measure the winding of the curves! 
The resolution of this apparent paradox will be explained below.

\begin{remark}\label{rem:rel_to_lit}
\textnormal{
As we were preparing to write our results, we learnt from Rick Kenyon that he had already considered in his work~\textnormal{\cite{Kenyon:Long-range_properties_of_spanning_trees}} the problem of conditioning a uniform spanning tree on the event $E^\delta_{\bs{x}}$ of consideration here, at least in the cases $n =2$ and $n=3$. 
In~\textnormal{\cite[Top of page~2]{Kenyon:Long-range_properties_of_spanning_trees}}, he writes:
}

\smallskip

\indent{``A final computation is the asymptotic ``winding number'' of the branches of the tree. For the UST on $\Z^2$, there is almost surely a unique branch from the origin to $\infty$.
In the annulus of inner
radius $r$ and outer radius $R$ concentric about the origin, we show that the total turning of this
branch about the origin has variance tending to $(1/2 )\log(R/r) +  O(1)$, as $R/r \to \infty$ while $r \to \infty$.
Surprisingly, this same variance holds even if the origin is conditioned on having two or three
branches to $\infty$.''}

\smallskip

\textnormal{
The factor ``$1/2$'' above is probably a typo --- as we indicated above, the winding when there is no conditioning should be given by $2\log (1/\eps)$, instead of $(1/2) \log (1/\eps)$.  
More intriguingly, the setting described in the above paragraph is in particular covered by our Theorem~\ref{thm::total_winding}, so that \textbf{our result contradicts the conclusion in}~\textnormal{\cite[Theorem~6.1]{Kenyon:Long-range_properties_of_spanning_trees}}.
Namely, according to Theorem~\ref{thm::total_winding}, the variance should be divided by $4$ and $9$, respectively, instead of remaining the same. 
The error in the argument of~\textnormal{\cite{Kenyon:Long-range_properties_of_spanning_trees}} 
appears to us to be similar to the paradox mentioned in the previous paragraph: indeed, Kenyon first computes the asymptotic dimer-dimer correlation, and then deduces asymptotics for the corresponding height function.  
In particular, he then deduces that the variance of the height function at the inner boundary does not depend on $n$, from which his conclusion seemingly follows. 
}
\end{remark}

\paragraph{Resolution.} 
Our resolution of the paradox comes from the fact that the dimers live on a non-simply connected graph, and the height function $\gff^\delta$ is only well defined as a function (so we can speak about its variance, say) 
\emph{if we take into account the jump across the branch cut} 
(see again Figure~\ref{F:dimers}). 

\medskip 
First, let us briefly recall the relation between winding and height function in the simply connected (say, Temperleyan) case.  Consider a 
path going from one face to another, staying parallel --- immediately adjacent --- to a branch from the associated uniform spanning tree. Such a path does not encounter any dimer, as per definition of Temperley's bijection. 
Hence the height function, as we travel along this path, oscillates between two values differing by $\pm 1$, on each straight portion of the path, and has an extra jump at each corner. Overall, the height difference accumulated along this path between faces with the right parity is thus equal to the winding of the path (see~\cite{KPW:Trees_and_matchings, BLR:Note_on_dimers_and_T-graphs, BLR:Dimers_on_Riemann_surfaces_1}). 

\medskip 
Now, consider what happens in the non simply connected case.
Then, in addition to the winding of the path, one also has to take into account all the times when the path jumps across the branch cut. Thus, each time it crosses the branch cut, the path also accumulates a value given by 
$\mp (n+1) 2\pi$ (if the height is measured with respect to the winding reference flow, i.e., in units of $2\pi$
and the system has $n$ radial branches as in Figure~\ref{F:dimers}). 
Thus overall, as we go from inside to outside, if the number of turns across the annulus is $k\in \Z$, 
then the height difference will be equal to (up to unimportant $O(1)$ errors): 
\begin{align*}
\gff^\delta(f') - \gff^\delta(f) = 2k\pi - k (n+1) 2\pi = - 2nk \pi = - n \big( 2 \pi K(\gamma_1^\delta) \big) ,
\end{align*}
at two squares $f,f'$. 
As the winding has a variance proportional to $2/n^2$, it is no wonder that the height difference should have a variance independent of $n$!

\begin{remark}
\textnormal{
In fact, in the the physics literature, Wieland and Wilson~\cite{Wieland-Wilson:Winding_angle_variance_of_Fortuin-Kasteleyn_contours} 
had investigated numerically the variance of the winding of various random fractal curves, including loop-erased random walks, 
and including the setup of the present article. 
In particular, they expected (and numerically verified) that when $n$ strands whose scaling limit is expected to be an $\SLE_\kappa$ curve meet at an interior point, the variance of the winding angle scales as $(\kappa/n^2) \log (R/r)$.  
Our work confirms these numerical results; see the comments after Proposition~\ref{coro::truncated_winding}. 
This is also consistent with the winding spectrum conjectured by Duplantier and Binder using Coulomb gas methods\footnote{In the Coulomb gas formalism of conformal field theory, for each $\kappa \in (0,8]$, the common interior endpoint of the $n$ branches should correspond to a primary field of conformal weight $2h_{0,n/2}(\kappa) = \frac{4n^2-(\kappa-4)^2}{8\kappa}$ (equalling $\frac{n^2-1}{4}$ when $\kappa=2$). Interestingly enough, the conformal weight $2h_{0,1/2}(\kappa) = \frac{(6-\kappa)(\kappa-2)}{8\kappa}$ for $n=1$ vanishes at $\kappa=2$. See also Lemma~\ref{lem::flowline_nsided}.}
(see~\cite[Section~10.6]{Duplantier:Conformal_random_geometry} 
and~\cite{Duplantier-Binder:Harmonic_measure_and_winding_of_conformally_invariant_curves, HPW:Multiradial_SLE_with_spiral}). 
}
\end{remark}

\subsection*{Acknowledgements}

Part of this project was performed while the first, second, and last author were participating in a program hosted by the Institute for Pure and Applied Mathematics (IPAM) in Los Angeles, California, US, in Spring 2024,  
supported by the National Science Foundation (Grant No.~DMS-1925919).

N.B. and M.Lis both acknowledge the support from the Austrian Science Fund (FWF) grants 10.55776/\linebreak F1002 on ``Discrete random structures: enumeration and scaling limits". N.B.'s research is furthermore supported by FWF grant  10.55776/PAT1878824 on ``Random Conformal Fields''. M.Lis' research is supported by FWF grant P36298 on ``Spins, Loops and Fields''.

E.P.~is supported by the European Research Council (ERC) under the European Union's Horizon 2020 research and innovation programme (101042460): 
ERC Starting grant ``Interplay of structures in conformal and universal random geometry'' (ISCoURaGe); 
the Academy of Finland grant number 340461 ``Conformal invariance in planar random geometry'';
the Academy of Finland Centre of Excellence Programme grant number 346315 ``Finnish centre of excellence in Randomness and STructures (FiRST)''; 
and by the Deutsche Forschungsgemeinschaft (DFG, German Research Foundation) under Germany's Excellence Strategy EXC-2047/1-390685813.

\medskip 
We would like to thank Rick Kenyon for pointing out his article~\cite{Kenyon:Long-range_properties_of_spanning_trees}
and for insightful discussions.

\section{Winding identities: proof of Theorem~\ref{thm::total_winding}}
\label{sec::annulus}

In this section, we will derive the exact winding identity stated in Theorem~\ref{thm::total_winding}. 
Some formulas obtained in this section will also be used in the proof of the convergence of the discrete branches to $\SLE_2$ (Theorem~\ref{thm::conv_curves}). 
We use the same setup and notation as in Section~\ref{subsec:setup}.
The first Section~\ref{subsec:Fomin} collects well-known results:
a determinantal formula of Fomin~\cite{Fomin:LERW_and_total_positivity} for random walks,
the algorithm due to Pemantle and Wilson~\cite{Pemantle:Choosing_spanning_tree_for_integer_lattice_uniformly, Wilson:Generating_random_spanning_trees_more_quickly_than_cover_time}
relating UST branches with loop-erased random walks,
and definitions for random walk loop soups~\cite{Lawler-Ferreras:RW_loop_soup}. 
In Section~\ref{subsec:winding}, using Fomin's theorem we express the total winding of the branches in terms of the random walk and its loop soup 
(Proposition~\ref{prop::total_winding}). 
Section~\ref{subsec:total_winding} concerns the scaling limit to Brownian loop soup, 
and after several key lemmas, we finish the proof of Theorem~\ref{thm::total_winding}.

\subsection{Complex weights, Fomin's formula, and Random walk loop soup}
\label{subsec:Fomin}

Fix a simple path on $\zeta$ the dual graph of $\Omega^\delta = (\LV^\delta, \LE^\delta)$ oriented from the inner boundary to the outer boundary 
(this is the ``\textbf{branch cut}'' or ``\textbf{zipper}'', depicted in red in Figure~\ref{F:dimers}). 
For each edge $e\in \LE^\delta$ crossing $\zeta$, we distinguish two possible directions: 
we let $\vec{\LE}{}^\delta_\zeta$ (resp.~$\cev{\LE}{}^\delta_\zeta$)
denote the set of oriented edges crossing $\zeta$ from the left side of $\zeta$ to the right side of $\zeta$
(resp.~from the right side to left side).

For a discrete path $P = (y_0, \ldots, y_l)$ on $\Omega^\delta$, let $w_0(P) = \big( \prod_{i=0}^{l-1}\deg(y_i) \big)^{-1}$ 
denote the probability that a simple random walk starting from $y_0$ and run for time $l$ is equal to $P$. 
Fix $\beta\in\R$, and define
\begin{align}\label{eqn::weight_beta}
w_\beta(P) 
:= w_0(P) \, 
\bigg(\prod_{j=0}^{l-1} \ee^{\ii \beta} \, \one{\{(y_{j}, y_{j+1} ) \in \vec{\LE}{}^\delta_\zeta \}} \bigg) 
\bigg(\prod_{j=0}^{l-1} \ee^{-\ii \beta} \, \one{\{(y_{j}, y_{j+1} ) \in \cev{\LE}{}^\delta_\zeta \}} \bigg).
\end{align}
In words, under $w_\beta$, each path $P$ receives not only its random walk weight, but also an extra factor $\ee^{ \ii \beta}$ each time it crosses $\zeta_\delta$ from left to right, and a factor $\ee^{ - \ii \beta}$ for each crossing in the opposite direction. 

\medskip 
We will recall Fomin's determinant formula in this setup. 
It is more well known in the case of positive weights, but it is equally valid for complex weights as it is simply a consequence of planar topology.
Let $\BdryOutV^\delta$ (resp~$\BdryInV^\delta$) denote the set of vertices of $\BdryOut^\delta$ (resp~$\BdryIn^\delta$).
We define the ``\textbf{hitting probability}''
\begin{align}\label{eq:h_beta}
h^\delta_\beta(x, v) := \sum_{P \colon x \to v} w_\beta(P), 
\end{align}
where the sum is taken over all paths starting from the vertex $x \in \LV^\delta$ and ending at $v \in \BdryOutV^\delta$. 
Note that if $\beta=0$, then $h^\delta_\beta(x, v)$ really corresponds to the probability that a random walk starting from $x$ first touches the outer boundary $\BdryOut^\delta$ at $v$. 
For $\beta \in \R$, the function $h^\delta_\beta(\cdot, v)$ also has a natural interpretation as a discrete harmonic function on the \textbf{universal cover} of $\Omega^\delta$.
Namely, let $\{\Omega^\delta_m\}_{m \in \Z} = \{(\LV^\delta_m,\LE^\delta_m)\}_{m \in \Z}$ be copies of $\Omega^\delta = \Omega^\delta_0$, 
and denote by $\OmegaUniv^\delta$ the graph obtained by gluing $\Omega^\delta_m$ and $\Omega^\delta_{m+1}$ along $\zeta$, for every $m$.  
In particular, there are bijections $\LV^\delta \leftrightarrow \LV^\delta_m$ and $\LE^\delta \leftrightarrow \LE^\delta_m$ given by $v \leftrightarrow v_m$ and $e \leftrightarrow e_m$. 
Then, $h^\delta_\beta(\cdot, v)$ can be extended to all of $\OmegaUniv^\delta$ and is then a discrete harmonic function there,
equalling $\ee^{\ii m \beta}$ at $v_m$, for every $m$.  
In particular, if $\beta=0$, then $h^\delta_\beta(\cdot, v)$ yields back the above probability. 
If $\beta=\pi$, then $h^\delta_\beta(\cdot, v)$ also has a natural interpretation in terms a discrete harmonic function defined on a double cover of $\Omega^\delta$.

\medskip 
For a path $P$, we let $\mathrm{LE}(P)$ denote the \textbf{chronological loop-erasure} of $P$. 
Performing this operation to a random walk sample with reflecting boundary condition on the inner boundary $\BdryInV^\delta$ and absorbing boundary condition on the outer boundary $\BdryOutV^\delta$ yields a sample of the loop-erased random walk (LERW).

\begin{theorem}[Wilson's algorithm~\cite{Wilson:Generating_random_spanning_trees_more_quickly_than_cover_time}]
\label{thm: Wilsons algorithm}
Let $y_0 , \ldots, y_N$ be any enumeration of the set of vertices $\LV^\delta$. 
Define $\LG_0$ as the subgraph consisting of
only the vertex $y_0$, and recursively 
for $k=1,\ldots,N$ define $\LG_k$ as the union of $\LG_{k-1}$ and a LERW from $y_k$ to $\LG_{k-1}$, independently for each $k$. 
Then, the last step of the algorithm, $\LG_N = \LT^\delta$,
gives a sample of a UST on $\Omega^\delta = (\LV^\delta, \LE^\delta)$.
\end{theorem}

Equation~\eqref{eq:h_beta} applies to random walks starting from $\BdryIn^\delta$ exiting at the outer boundary $\BdryOut^\delta$. 
Via Wilsons algorithm, the event $E^\delta_{\bs{x}, \bs{v}}$ defined in~\eqref{eqn::defE2} concerns exactly this kinds of UST branches.
The following determinantal formula is our key tool to calculate the total winding of these branches.

\begin{theorem}[Fomin's formula~\cite{Fomin:LERW_and_total_positivity}]
\label{T:Fomin}
Fix distinct vertices $v_1^\delta,\ldots,v_n^\delta \in \BdryOutV^\delta$ 
and distinct vertices $x_1^\delta,\ldots,x_n^\delta \in \BdryInV^\delta$ in counterclockwise order. 
For $\beta\in\R$, we have 
\begin{align}\label{eq:Fomin}
\det \big( h^\delta_\beta(x^\delta_i, v_j^\delta)\big)_{1\le i,j \le n} 
= \sum_{\sigma \in \mathfrak{S}_n} \mathrm{sgn}(\sigma)\sum_{P_1, \ldots, P_n} w_\beta(P_1) \cdots w_\beta(P_n),
\end{align}
where the sum is over all choices of paths $P_1, \ldots, P_n$ on $\Omega^\delta$ such that 
$P_j \colon x^\delta_j \to v_{\sigma(j)}^\delta$ for each $1\le j\le n$ 
and $\mathrm{LE}(P_i) \cap P_j = \emptyset$ for all $1\le i < j\le n$. 
\end{theorem}

Intuitively, the random walk is recovered from its loop-erasure by adding a soup of discrete loops on top of it. 
The random walk loop measure $\nu$, introduced in~\cite{Lawler-Ferreras:RW_loop_soup}, is supported on discrete rooted loops on $\Z^2$.
For each loop $\ell$, we assign the weight $\nu(\ell) := \frac{(1/4)^{|\ell|}}{|\ell|}$, 
and we let $\bar \ell$ denote the equivalence class of $\ell$ in the equivalence relation $\ell \sim \ell'$ if $\ell'$ can be obtained from $\ell$ by rerooting. 
We denote by 
\begin{align*}
\Lambda[\bar\ell] := \sum_{\ell\in\bar\ell} \nu(\ell) 
\end{align*}
the measure obtained from $\nu$ by forgetting the root. 
The \textbf{random walk loop soup with intensity} $\vartheta>0$ is the Poisson point process with intensity measure $\vartheta\Lambda$. 
In~\cite{Lawler-Ferreras:RW_loop_soup}, the authors construct a coupling between Brownian loop soup and random walk loop soup with the same intensity. 
To do so, they first scale the loops so that the are on $\delta \Z^2$ instead of $\Z^2$, and define the \emph{lifetime} of the discrete loop $\ell$ to be given by $t_\ell := \frac{\delta^2}{2} |\ell|$. 
(For future reference, note that we have $t_\ell\ge\frac{\delta}{2} \diam(\ell)$,
where $\diam(\ell)$ is the diameter of the loop $\ell$ with respect to the Euclidean distance.) 
The authors of~\cite{Lawler-Ferreras:RW_loop_soup} then construct a coupling such that, with large probability, 
all the rescaled random walk loops above a certain mesoscopic size correspond uniquely to a unique Brownian loop, and are very well (uniformly) approximated by it.

\medskip 
In the present article, we will consider the random walk loop soup with Neumann boundary condition on $\BdryIn^\delta$ and Dirichlet boundary condition on $\BdryOut^\delta$. 
The presence of a Neumann boundary condition introduces a number of technicalities which we briefly explain. 
For each loop $\ell=(v_0,\ldots,v_{|\ell|})$ of duration $|\ell|$ with root $v_0 = v_{|\ell|} \in \LV^\delta\setminus\BdryOutV^\delta$, we assign the weight
\begin{align*}
\nu^{\textnormal{ND}}_{\Omega^\delta}(\ell) 
:= \frac{1}{|\ell|} \Big( \prod_{i=0}^{|\ell|}\deg(v_i) \Big)^{-1} .
\end{align*}
We similarly define $\Lambda_{\Omega^\delta}^{\textnormal{ND}}$ to be the measure obtained from $\nu_{\Omega^\delta}^{\textnormal{ND}}$ by forgetting the root. 
In this case, we still denote by $t_\ell := \frac{\delta^2}{2} |\ell|$ the lifetime of $\ell$, 
and we parametrise $\ell$ by setting $\ell(s) := v_{[2s/\delta^2]}$ for $0\le s\le t_\ell$.

Note that (with a small abuse of notation), if $F$ is a functional on rooted loops that is invariant under rerooting (so that $F$ can also be seen as a functional on \emph{unrooted} loops), we have
\begin{align*} 
\Lambda^{\textnormal{ND}}_{\Omega^\delta}[F] 
= \; & \sum_{\bar \ell} \Lambda^{\textnormal{ND}}_{\Omega^\delta} (\bar \ell) \, F(\bar \ell) 
\; = \; 
\sum_{\bar\ell}\sum_{\ell\in\bar\ell}\nu(\ell) \,F(\ell) 
\; = \; \sum_{\ell}\nu(\ell)F(\ell) \\
= \; & \sum_{z\in\LV^\delta\setminus\BdryOutV^\delta} \sum_{m=1}^\infty \;
\sum_{\substack{\ell \textnormal{ s.t. } |\ell|=m \textnormal{ and }\\ \textnormal{ the root of $\ell$ is }z}} 
\frac{w_0(\ell) }{m} \, F(\ell) .
\end{align*}
After parameterising the loop by $t \in [0,\infty)$ in such a way at time $t$ its position equals that of the $[2t/\delta^2]$-th step of the original loop,
we obtain 
\begin{align}\label{eqn::dec_discrete_loop}
\Lambda^{\textnormal{ND}}_{\Omega^\delta}[F] 
= \sum_{z\in\LV^\delta\setminus\BdryOutV^\delta} 
\sum_{t\in \frac{1}{2}\delta^2\N}
\frac{\delta^2}{2 t} 
\sum_{\substack{\ell \textnormal{ s.t. } |\ell|=t \textnormal{ and }\\ \textnormal{ the root of $\ell$ is }z}} w_0(\ell) \, F(\ell) .
\end{align}

According to the right side of~\eqref{eqn::dec_discrete_loop}, we will occasionally, with a slight of notation abuse, view $\Lambda^{\textnormal{ND}}_{\Omega^\delta}$ as a measure on rooted loops. 
The \textbf{random walk loop soup} with Neumann boundary condition on $\BdryIn^\delta$ and Dirichlet boundary condition on $\BdryOut^\delta$ is the Poisson point process with intensity measure $\Lambda^{\textnormal{ND}}_{\Omega^\delta}$. 

\medskip 
As part of our argument, we will need to prove the convergence of the measure on the left-hand side of~\eqref{eqn::dec_discrete_loop} --- more precisely, 
the convergence of the conditional law $\Prob[\, \cdot \cond \LR(0) = \LR(t^\delta) = z^\delta]$ 
of the random walk $\LR \sim \Prob$ to return to its root at time $t^\delta$, to a Brownian bridge (with mixed boundary conditions) of duration $t$;
and the convergence of the rescaled probability $\delta^{-2} \, \Prob[\LR(t^\delta) = z^\delta \cond \LR(0)=z^\delta]$. 
This will be done in Proposition~\ref{prop::loop_conv}, with $z^\delta \to z\in\Omega$ and $t^\delta \to t \in(0,\infty)$ as $\delta \to 0$. 
Denote the limits by 
\begin{align*}
\lim_{\delta \to 0} \Prob[\, \cdot \cond \LR(0)=\LR(t^\delta)=z^\delta] =: \; & \mu^{\textnormal{ND}}_{\Omega}(z,z;t)[\,\cdot\,] ,
\\ 
\lim_{\delta \to 0} \delta^{-2} \, \Prob[\LR(t^\delta)=z^\delta \cond \LR(0)=z^\delta] =: \; & p^{\textnormal{ND}}_{\Omega}(z,z;t)
\end{align*}
(the former as a weak limit of measures). Note that here, we parameterise the random walk $\LR$ by letting $\LR(s)$ be the $[2s/\delta^2]$:th lattice step, for $s \geq 0$. 
We define the corresponding (Brownian) loop measure by
\begin{align}\label{eqn::mixed_loop}
\mu_{\Omega}^{\textnormal{ND}}[\,\cdot\,] 
:= \int_\Omega  \ud z \int_0^\infty \ud t \; \frac{ p_{\Omega}^{\textnormal{ND}}(z,z;t)}{t} \, \mu_{\Omega}^{\textnormal{ND}}(z,z;t)[\,\cdot\,] .
\end{align}
A consequence of our results\footnote{Without the reflecting boundary conditions, this would essentially follow trivially from, e.g.,~\cite{Lawler-Ferreras:RW_loop_soup}.} 
is that the discrete loop measure $\Lambda^{\textnormal{ND}}_{\Omega^\delta}$ converges as $\delta \to 0$ to $\mu_{\Omega}^{\textnormal{ND}}$. 
In the present work, the \textbf{Brownian loop soup} on $\Omega$ with Neumann boundary condition on $\BdryIn$
and Dirichlet boundary condition on $\BdryOut$ is the Poisson point process with intensity measure $\mu_{\Omega}^{\textnormal{ND}}$. 

\medskip 
Historically, the Brownian loop measure $\mu$ introduced in~\cite{LSW:Conformal_restriction_the_chordal_case, Lawler-Werner:The_Brownian_loop_soup} has become an important tool in probability theory and geometry of random fractals. 
It is an infinite, sigma-finite, conformally invariant measure supported on the set of unrooted continuous loops on $\C$. 
For a bounded domain $D \subset \C$, denote by $\mu_{D}$ the restriction of $\mu$ on loops supported on $D$. 
Then, $\mu_{D}$ is a conformally invariant in the sense that for any conformal isomorphism $f \colon D \to D'$, the measure $\mu_{D'}$ coincides with the pushforward of $\mu_{D}$ by $f$.
While the total mass of $\mu_{D}$ is infinite (and can be related to the determinant of the Laplacian~\cite{LeJan:Markov_loops_determinants_and_Gaussian_fields, Dubedat:SLE_and_free_field}),   
for any two disjoint compact sets $V_1,V_2\subset D$, the measure $\mu_{D}(V_1,V_2)$ of loops in $D$ which intersect both $V_1$ and $V_2$ is finite (see~\cite{Field-Lawler:Reversed_radial_SLE_and_the_Brownian_loop_measure}). 
Moreover, the measure satisfies an important restriction covariance property
(see~\cite[Proposition~3]{Werner:The_conformally_invariant_measure_on_self-avoiding_loops} for the case useful to our purposes).

\subsection{Discrete winding identities: an exact formula}
\label{subsec:winding}

We now use Fomin's formula (Theorem~\ref{T:Fomin}) to derive key exact identities, which allow us to express the total winding of all 
the UST branches in terms of the discrete loop soup (Proposition~\ref{prop::total_winding}).

\medskip
First, because of planarity, the only permutations $\sigma$ which contribute to the right hand side of~\eqref{eq:Fomin} are powers of the shift $\sigma = (1\ 2 \cdots n)$; i.e., permutations of the form $\sigma^k$ for $0 \le k \le n-1$.  
Note that the sign is always positive when $n$ is odd and it alternates between $\pm 1$ when $n$ is even. 
This will lead to the difference in Theorem~\ref{thm::total_winding} between these cases, which we will treat separately below.

For every $\beta\in\R$, we let $\smash{\Lambda_\beta=\Lambda_{\Omega^\delta,\beta}^{\textnormal{ND}}}$ be the signed loop measure, i.e., 
for every discrete loop $\ell$ on $\Omega^\delta\setminus\BdryOut^\delta$ with length $|\ell|$ 
and with Neumann boundary condition on $\BdryIn^\delta$ and Dirichlet boundary condition on $\BdryOut^\delta$, using~\eqref{eqn::weight_beta} we define its weight to be
\begin{align}\label{eq:weightssigned}
\Lambda_\beta[\ell] := \frac{w_\beta(\ell)}{|\ell|}.
\end{align}

Denote by $\LL^\delta$ the set of all unrooted loops on $\Omega^\delta\setminus\BdryOut^\delta$. 
Let $\LL_0^\delta\subset\LL^\delta$ be the subset of contractible loops, and $\LL_*^\delta = \LL^\delta\setminus\LL_0^\delta$ the subset of noncontractible loops.

\begin{proposition}\label{prop::total_winding}
For any $\beta\in\R$, and for any odd $n\ge 1$, we have
\begin{align}\label{eqn::scaling_wind_odd}
\E_{\bs{x}, \bs{v}}^\delta\Big[\exp \Big(\ii \beta \sum_{j=1}^{n}2\pi K(\gamma_j^\delta)\Big)\Big] 
= & \; \ee^{\Lambda_0[ \LL_*^\delta] - \Lambda_\beta[ \LL_*^\delta]} \; 
\frac{\det \big( h^\delta_{2\pi\beta}(x^\delta_i, v_j^\delta)\big)_{1\le i,j \le n}}{\det \big( h^\delta_0(x^\delta_i, v_j^\delta)\big)_{1\le i,j \le n}} ;
\end{align}
and for any even $n\ge 2$, we have
\begin{align}\label{eqn::scaling_wind_even}
\E_{\bs{x}, \bs{v}}^\delta\Big[\exp \Big(\ii \beta \sum_{j=1}^{n}2\pi K(\gamma_j^\delta)\Big)\Big] 
= & \; \ee^{\Lambda_\pi[ \LL_*^\delta] - \Lambda_{2\pi \beta+\pi}[ \LL_*^\delta]} \; 
\frac{\det \big( h^\delta_{2\pi \beta+\pi}(x^\delta_i, v_j^\delta)\big)_{1\le i,j \le n}}{\det \big( h^\delta_\pi(x^\delta_i, v_j^\delta)\big)_{1\le i,j \le n}}.
\end{align}
where $h^\delta_\beta$ is defined in Equation~\eqref{eq:h_beta}.
\end{proposition} 

The proof of Proposition~\ref{prop::total_winding} is given in Lemmas~\ref{lem::nodd0}~\&~\ref{lem::neven}.  
We begin with the case where $n$ is odd, which is easier to deal with, since 
(essentially as a consequence from Fomin's formula in Theorem~\ref{T:Fomin})  
the partition function can be expressed exactly as a determinant. For a discrete path $\gamma$, let 
\begin{align} \label{eq:crossings}
K(\gamma) := | \gamma \cap \vec{\LE}{}^\delta_\zeta | - |\gamma\cap\cev{\LE}{}^\delta_\zeta | 
\end{align}
be the (algebraic) number of turns by $\gamma$ in the annulus.

\begin{lemma}\label{lem::nodd0}
When $n\ge 1$ is odd, for the event~\eqref{eqn::defE2}, we have
\begin{align}\label{eqn::odd}
\PP^\delta[E^\delta_{\bs{x}, \bs{e}}] = \det \big(h^\delta_{0}(x^\delta_i,v_j^\delta) \big)_{1\le i,j\le n} ,
\end{align}
given by~\eqref{eq:Fomin} with $\beta=0$.  
Moreover, for any $\beta\in\R$ and for any odd $n\ge 1$, we have
\begin{align}\label{eqn::wind_odd0}
\E_{\bs{x}, \bs{v}}^\delta\Big[\exp \Big(\ii \beta \sum_{j=1}^{n}K(\gamma_j^\delta)\Big)\Big] 
= & \; \ee^{\Lambda_0[ \LL_*^\delta] - \Lambda_\beta[ \LL_*^\delta]} \; 
\frac{\det \big( h^\delta_\beta(x^\delta_i, v_j^\delta)\big)_{1\le i,j \le n}}{\det \big( h^\delta_0(x^\delta_i, v_j^\delta)\big)_{1\le i,j \le n}}.
\end{align}
\end{lemma}

\begin{proof}
Since $n$ is odd, for every $\sigma \in \mathfrak{S}_n$ which contributes to the right hand side of~\eqref{eq:Fomin},
we have $\mathrm{sgn}(\sigma)=1$. Taking $\beta=0$ in~\eqref{eq:Fomin}, and invoking Wilson's algorithm (Theorem~\ref{thm: Wilsons algorithm}), we obtain~\eqref{eqn::odd}.

It remains to prove~\eqref{eqn::wind_odd0}. 
We decompose a path $P$ into its loop-erasure $\Gamma$ and the set of loops erased. 
As in~\cite[Proposition~9.5.1]{Lawler-Limic:Random_walk_modern_introduction}, we have
\begin{align}\label{eqn::path_dec_beta}
\sum_{P \colon \mathrm{LE}(P)=\Gamma}w_\beta(P) = w_\beta(\Gamma) \, 
\exp \Big( \Lambda_\beta[\ell\subset\Omega^\delta\setminus\BdryOut^\delta \;|\; \ell \cap \Gamma \neq \emptyset]\Big). 
\end{align}
It will be convenient to introduce the notation $\Lambda_\beta[ \gamma^\delta_j \,|\, \gamma^\delta_1, \ldots, \gamma^\delta_{j-1}]$ to denote the $\Lambda_\beta$-mass of loops in $\LL^\delta$ which intersect $\gamma^\delta_j$ but not $\gamma^\delta_1, \ldots, \gamma^\delta_{j-1}$.   
Note that the weight of each loop in $\LL_0^\delta$ under $\Lambda_\beta$ is actually independent of $\beta$. 
Thus, if we fix a permutation $\sigma$, and for each $j$ a path $P_j$ connecting $x^\delta_j$ to $\smash{v_{\sigma(j)}^\delta}$, 
then letting $\smash{\gamma^\delta_j = \mathrm{LE}(P_j)}$, we obtain
\begin{align}\label{eq:weightsprod}
w_\beta(P_1) \cdots w_\beta(P_n) 
= \; & \Big(\prod_{j=1}^n w_\beta(\gamma_j^\delta)  \Big) 
\exp \Big( \Lambda_\beta \big[ \gamma^\delta_1 \big]  + \cdots + \Lambda_\beta \big[ \gamma^\delta_n \,|\, \gamma^\delta_0, \ldots, \gamma^\delta_{n-1} \big] \Big) .
\end{align}

Consider first the path $P_1$ and $\gamma^\delta_1 = \mathrm{LE}(P_1)$. 
On the one hand, every loop $\ell\in \LL_*^\delta$ is guaranteed to intersect $\gamma^\delta_1$. 
On the other hand, loops which intersect $\gamma^\delta_1$ but are homotopically trivial do not receive any non-trivial complex weights, i.e., have the same weight under both $\Lambda_0$ and $\Lambda_\beta$. Thus, we have 
\begin{align*}
\Lambda_\beta \big[ \gamma^\delta_1 \big] = \Lambda_\beta \big[ \LL_*^\delta \big] + \Lambda_0 \big[ \LL_0^\delta \cap \{\ell \textnormal{ such that } \ell\cap\gamma^\delta_1\neq\emptyset\} \big].
\end{align*}
Furthermore, for the remaining loops, observe that $\gamma^\delta_1$ cuts $\Omega^\delta$ open into a simply connected domain.
Therefore, every loop which intersects $\gamma^\delta_j$ but not $\gamma^\delta_1, \ldots \gamma^\delta_{j-1}$ (for $j\ge 2$) is necessarily homotopically trivial, and thus receives a weight equal to that under $\Lambda_0$. 
We deduce that for any $\beta$, we have
\begin{align*}
\Lambda_0\big[\LL_0^\delta(\gamma_j^\delta) \big] 
= \Lambda_\beta\big[\LL_0^\delta(\gamma_j^\delta) \big]
:= \; & \Lambda_\beta\big[ \ell \in \LL_0^\delta \textnormal{ such that } \ell \cap \gamma_j^\delta \neq \emptyset 
\textnormal{ and }\ell\cap\gamma_k^\delta=\emptyset\textnormal{ for }1\le k\le j-1 \big]
\\
= \; & 
\begin{cases}
\Lambda_\beta \big[ \gamma^\delta_1 \big] - \Lambda_\beta \big[ \LL_*^\delta \big], & j = 1,\\
\Lambda_0\big[ \gamma^\delta_j \,|\, \gamma^\delta_1, \ldots, \gamma^\delta_{j-1} \big], &  j \geq 2 .
\end{cases}
\end{align*}
and thus, combining with~\eqref{eq:weightsprod} we see that Fomin's formula~\eqref{eq:Fomin} reads
\begin{align*}
\det \big( h^\delta_\beta(x^\delta_i, v_j^\delta)\big)_{1\le i,j \le n} 
= \; & \ee^{\Lambda_\beta[\LL_*^\delta]}\sum_{\bs\gamma^\delta\in E^\delta_{\bs{x}, \bs{e}}} 
\Big( \prod_{j=1}^{n}w_\beta(\gamma_j^\delta) \Big) \; \ee^{\Lambda_\beta[\LL_0^\delta(\gamma_j^\delta) ]} 
\\
=\; & \ee^{\Lambda_\beta[ \LL_*^\delta]}\sum_{\bs\gamma^\delta\in E^\delta_{\bs{x}, \bs{e}}} \; 
\exp\Big(\ii\beta\sum_{j=1}^{n}\big(|\gamma_j^\delta\cap\vec{\LE}{}^\delta_\zeta|-|\gamma_j^\delta\cap\cev{\LE}{}^\delta_\zeta|\big)\Big)
\; \Big( \prod_{j=1}^{n}w_0(\gamma_j^\delta) \Big) \;  \ee^{\Lambda_0[\LL_0^\delta(\gamma_j^\delta) ]} 
\\
=\; & \ee^{\Lambda_\beta[ \LL_*^\delta]}\sum_{\bs\gamma^\delta\in E^\delta_{\bs{x}, \bs{e}}} \; \ee^{\ii{\beta}\sum_{j=1}^{n}K(\gamma_j^\delta)}
\; \Big( \prod_{j=1}^{n}w_0(\gamma_j^\delta) \Big) \;  \ee^{\Lambda_0[\LL_0^\delta(\gamma_j^\delta) ]} 
\\
=\; & \ee^{\Lambda_\beta[ \LL_*^\delta]-\Lambda_0[\LL_*^\delta]} \; 
\E^\delta\big[\ee^{\ii{\beta}\sum_{j=1}^{n}K(\gamma_j^\delta)} \, \one{\{E^\delta_{\bs{x}, \bs{e}}\}}\big] 
\\
=\; & \ee^{\Lambda_\beta[ \LL_*^\delta]-\Lambda_0[\LL_*^\delta]} \; 
\E_{\bs{x}, \bs{v}}^\delta\big[\ee^{\ii{\beta}\sum_{j=1}^{n}K(\gamma_j^\delta)}\big] \, \PP^\delta\big[E^\delta_{\bs{x}, \bs{e}}\big] ,
\end{align*}
recalling that $\E_{\bs{x}, \bs{v}}^\delta$ is the conditional expectation given $E^\delta_{\bs{x}, \bs{e}}$.
This proves~\eqref{eqn::wind_odd0} due to~\eqref{eqn::odd}. 
\end{proof}

We now turn to the case where $n$ is even. 
It includes an additional factor with the signed loop measure~\eqref{eq:weightssigned}
from the alternating signs of the permutations on the right-hand side of Fomin's formula~\eqref{eq:Fomin}. 

\begin{lemma}\label{lem::neven}
When $n\ge 2$ is even, for the event~\eqref{eqn::defE2}, we have
\begin{align}\label{eqn::even}
\PP^\delta[E^\delta_{\bs{x}, \bs{e}}] 
= \ee^{\Lambda_{0}[ \LL_*^\delta] - \Lambda_\pi[ \LL_*^\delta]} \; \det \big(h^\delta_{\pi}(x^\delta_i,v_j^\delta) \big)_{1\le i,j\le n} ,
\end{align}
given by~\eqref{eq:Fomin} with $\beta=\pi$.  
Moreover, for any $\beta\in\R$ and for any even $n\ge 2$, we have
\begin{align}\label{eqn::wind_even}
\E_{\bs{x}, \bs{v}}^\delta \Big[\exp\Big(\ii \beta \sum_{j=1}^{n} K(\gamma_j^\delta) \Big) \Big] 
= & \; \ee^{\Lambda_\pi[ \LL_*^\delta] - \Lambda_{\beta+\pi}[ \LL_*^\delta]} \; 
\frac{\det \big( h^\delta_{\beta+\pi}(x^\delta_i, v_j^\delta)\big)_{1\le i,j \le n}}{\det \big( h^\delta_\pi(x^\delta_i, v_j^\delta)\big)_{1\le i,j \le n}}.
\end{align}
\end{lemma}

\begin{proof}
The proof of Equation~\eqref{eqn::even} is very similar to the proof of Lemma~\ref{eqn::odd}, so we only point out the differences. 
Due to planarity, for every $(\gamma_1^\delta,\ldots,\gamma_n^\delta)\in E^\delta_{\bs{x}, \bs{e}, k}$ realising the event~\eqref{eqn::defE}, we have
\begin{align}\label{eqn::even_aux1}
\sum_{j=1}^{n}\big|\gamma_j^\delta \cap \vec{\LE}{}^\delta_\zeta\big|
\; - \; \sum_{j=1}^{n} \big|\gamma_j^\delta\cap\cev{\LE}{}^\delta_\zeta\big| 
= k \mod 2 , \qquad \textnormal{for } 0\le k\le n-1  .
\end{align}
After replacing $\beta$ by $\beta+\pi$ in Fomin's formula~\eqref{eq:Fomin} 
and taking into account the alternating signs of the permutations on the right-hand side of~\eqref{eq:Fomin}, 
a similar analysis as in the proof of Lemma~\ref{eqn::odd} yields
\begin{align}\label{eqn::even_aux2}
\det \big( h^\delta_{\beta+\pi}(x^\delta_i, v_j^\delta)\big)_{1\le i,j \le n} 
=\; & \ee^{\Lambda_{\beta+\pi}[ \LL_*^\delta]-\Lambda_\pi[\LL_*^\delta]} \; 
\E_{\bs{x}, \bs{v}}^\delta\big[\ee^{\ii{\beta}\sum_{j=1}^{n} K(\gamma_j^\delta)}\big] \, \PP^\delta\big[E^\delta_{\bs{x}, \bs{e}}\big] ,
\end{align}
Equation~\eqref{eqn::wind_even} follows immediately from~\eqref{eqn::even_aux2}, and by taking $\beta=0$ in~\eqref{eqn::even_aux2}, we obtain~\eqref{eqn::even}. 
\end{proof}

\subsection{Convergence to Brownian loop soup: proof of Theorem~\ref{thm::total_winding}}
\label{subsec:total_winding}

Both right-hand sides in Proposition~\ref{prop::total_winding} (corresponding to the case where $n$ is odd and even, respectively) 
are expressed in terms of a ratio of determinants and an exponential of total mass of loops winding around the origin associated with the complex weights~\eqref{eqn::weight_beta}. We will deal with both terms separately, beginning with the loop soup term.
In fact, the latter can simply be re-expressed in terms of the ordinary random walk loop soup (albeit with mixed Neumann and Dirichlet boundary conditions). 
For instance, in the odd case (see Equation~\ref{eqn::scaling_wind_odd}), recalling Campbell's formula for Poisson point processes, we see that
\begin{align}\label{eqn::Campbell}
\E_{\textnormal{RWLS}}^\delta\Big[\exp \Big(\ii \frac{\beta}{2\pi} \sum_{\ell\in\LL_*^\delta} \ph(\ell)\Big)\Big]
= \exp \Big(\Lambda_0 \Big[(\ee^{\ii\frac{\beta}{2\pi}\ph(\ell)}-1)\one{\{\ell\in\LL_*^\delta\}} \Big]\Big)
= \ee^{\Lambda_\beta[\LL_*^\delta]-\Lambda_0[\LL_*^\delta]} ,
\end{align}
where $\E_{\textnormal{RWLS}}^\delta$ is the expectation under the random walk loop soup probability measure $\PP_{\textnormal{RWLS}}^\delta$ with Neumann boundary conditions on $\BdryIn^\delta$ and Dirichlet boundary
conditions on $\BdryOut^\delta$,
and $\ph(\ell)$ is the topological winding of $\ell \in \LL^\delta$ around the origin, equalling $2\pi$ times the signed number of turns of $\ell$ around~$0$.
(When $\beta=0$, the measure $\Lambda_0$ coincides with the random walk loop measure $\Lambda_{\Omega^\delta}^{\textnormal{ND}}$;
and depending on the context it is more natural to write $\Lambda_0$ or $\Lambda_{\Omega^\delta}^{\textnormal{ND}}$, so we keep both notations).

\medskip 
To prove Theorem~\ref{thm::total_winding}, 
our task will thus be to show the convergence of the term in the exponential in the right-hand side to its continuum analogue, 
which involves the Brownian loop soup with Neumann boundary conditions on $\BdryIn$ and Dirichlet boundary
conditions on $\BdryOut$ (recall that we call $\mu_{\Omega}^{\textnormal{ND}}$ the corresponding intensity measure of this Poisson point process). This task will be completed in Lemma~\ref{lem:BLM_conv}.
We first need to prove the convergence of $\PP_{\textnormal{RWLS}}^\delta$ as $\delta \to 0$ when restricted on set of loops with non-trivial winding.  
As mentioned in Section~\ref{subsec:Fomin}, we will prove the convergence of the conditional law $\Prob[\, \cdot \cond \LR(0)=\LR(t^\delta)=z^\delta]$ 
of the random walk $\LR$ to return to its root at time $t^\delta$ to a Brownian bridge of duration $t = \underset{\delta \to 0}{\lim} \, t^\delta$; 
and the convergence of the rescaled probability $\delta^{-2} \, \Prob[\LR(t^\delta)=z^\delta \cond \LR(0)=z^\delta]$.

The convergence takes place in the set of pairs $(\ell,t_\ell)$, where $\ell$ is a loop in the plane of lifetime $t_\ell > 0$, endowed with the usual metric up to reparametrisation, 
\begin{align}\label{eqn::loop_metric}
\dist_\LL \big( (\ell,t_\ell), (\tilde{\ell},t_{\tilde{\ell}}) \big) 
:= |t_\ell - t_{\tilde{\ell}}| \,+\, \inf_{\varrho,\tilde{\varrho}} \sup_{s\in [0,1]} \big|\ell(\varrho(s))-\tilde{\ell}(\tilde{\varrho}(s)) \big|,
\end{align}
where the infimum is taken over all increasing homeomorphisms $\varrho \colon [0,1] \to [0,t_\ell]$ and  $\varrho \colon [0,1] \to [0,t_{\tilde{\ell}}]$.

\begin{proposition}\label{prop::loop_conv}
Suppose $z^\delta \to z\in\Omega$ and $t^\delta \to t\in(0,\infty)$ as $\delta \to 0$. 
Then as $\delta\to 0$,  
the conditional law $\Prob[\, \cdot \cond \LR(0)=\LR(t^\delta)=z^\delta]$ converges weakly 
to the Brownian bridge measure $\mu^{\textnormal{ND}}_{\Omega}(z,z;t)$, 
and the rescaled probability $\delta^{-2} \, \Prob[\LR(t^\delta)=z^\delta \cond \LR(0)=z^\delta]$ converges 
to the transition kernel $p^{\textnormal{ND}}_{\Omega}(z,z,t) \in \R$. 
\end{proposition}

\begin{proof}
Let $0<c<t$. Fix a bounded and measurable function $f$ from $C([0,c],\C)$ to $\R$. 
We choose $\eps>0$ and $\delta_0=\delta_0(\eps)>0$ such that 
$\dist(z^\delta,\partial\Omega^\delta)\ge 2\eps$ for every $\delta\in(0,\delta_0)$, 
and define recursively sequences of stopping times for the random walk: 
we set $T_1^\delta:=\inf\{s>c \colon \LR(s)\in \partial B(z^\delta,\frac{\eps}{2}) \}$; 
and we set 
\begin{align*}
S^\delta_j:=\inf\{s>T^\delta_j \colon \LR(s)\in \partial B(z^\delta,\eps) \} 
\qquad\textnormal{and}\qquad 
T^\delta_{j+1}:=\inf\{s>S^\delta_j \colon \LR(s)\in \partial B(z^\delta,\tfrac{\eps}{2}) \} , \qquad j\ge 1 .
\end{align*}
By the convergence of discrete harmonic measure, there exists a constant $0<q=q(\eps)<1$, such that 
\begin{align}\label{eqn::loop_conv_aux1}
\sup_{\delta\in(0,\delta_0)} 
\max_{w\in\partial B(z^\delta,\eps)} 
\, \Prob_{w}\big[\LR\textnormal{ hits }\partial B(z^\delta,\tfrac{\eps}{2})\textnormal{ before hitting }\BdryOut^\delta\big]\le 1-q.
\end{align} 
where $\Prob_{w}$ denotes the probability measure for $\LR$ starting from $w$. 
We will use the following observations.
\begin{itemize}[leftmargin=*]
\item 
Using the strong Markov property, we see that
\begin{align*} 
\;&\Prob_{z^\delta} \big[f(\LR[0,c]) \, \one{\{\LR(t^\delta)=z^\delta\}} \big]\\
=\;&\sum_{j = 1}^\infty\ProbE_{z^\delta}\Big[[f(\LR[0,c]) \, 
\one{\{T_j^\delta<\infty\}} \, 
\Prob_{\LR(T_j^\delta)} \big[\LR(t^\delta-T_j^\delta) = z^\delta\textnormal{ and }\LR[0,t^\delta-T_j^\delta] \subset B(z^\delta,\eps) \big] \Big]. 
\end{align*}

\item
Denote by $B^{\textnormal{ND}} \sim \ProbBM_{z}$ the Brownian motion on $\Omega$ with reflective boundary $\BdryIn$ started at $z$ and stopped when hitting $\BdryOut$. 
Then, by~\cite{Stroock-Varadhan:Diffusion_processes_with_boundary_conditions} (since we assumed that the reflective boundary $\BdryIn$ is analytic and thus in particular $C^2$ --- see also~\cite{Burdzy-Chen:Discrete_approximations_to_reflected_Brownian_motion} 
and references therein for more recent works on this topic, including an in-depth discussion for the case where the boundary is not necessarily smooth), 
we can couple $\LR[0,t^\delta]$ and $B^{\textnormal{ND}}[0,t]$ together in such a way that $|| \LR - B^{\textnormal{ND}} ||_{{\infty}}\to 0$ as $\delta\to 0$.  
In particular, defining the stopping times $\{T_j\}_{j\ge 1}$ for $B^{\textnormal{ND}}$ similarly as above, 
the law of $(\LR[0,c],T_1^\delta,\LR(T_1^\delta),\ldots, T_N^\delta,\LR(T_N^\delta))$
converges weakly to the law of $(B^{\textnormal{ND}}[0,c],T_1, B^{\textnormal{ND}}(T_1),\ldots, T_N,B^{\textnormal{ND}}(T_N))$ for every $N\in \N$.

\item
Denote by $\hat\LR\sim \ProbHat_{w^\delta}$ the standard random walk on $\delta \Z^2$ started at $\hat\LR(0)=w^\delta\in\LV^\delta$,  
and by $\hat\exitT_\eps^\delta$ the hitting time of $\hat\LR$ at $\partial B(z^\delta,\eps)$. 
Then, we have
\begin{align*}
\;&\Prob_{w^\delta}\big[\LR(s^\delta) = z^\delta\textnormal{ and }\LR[0,s^\delta]\subset B(z^\delta,\eps)\big]\\
=\;& \ProbHat_{w^\delta}\big[\hat\LR(s^\delta) = z^\delta\big] \, - \, \ProbHat_{w^\delta}\big[\one{\{\hat\exitT_\eps^\delta < s^\delta\}}\, 
\ProbHat_{\hat\LR(\hat\exitT_\eps^\delta)}\big[\hat\LR(s^\delta - \hat\exitT_\eps^\delta) = z^\delta]\big] , \qquad s^\delta > 0 .
\end{align*}
Similarly, denote by $\hat{B} \sim \hat\ProbBM_w$ the standard Brownian motion on $\C$ started at $\hat{B}(0) = w\in\Omega$, 
and by $\hat\exitT_\eps$ the hitting time of $\hat{B}$ at $\partial B(z,\eps)$. 
Let $\hat{p}(x,y;u)$ be the transition probability of $\hat{B}$ from $x$ to $y$ at time $u$.
Suppose $w^\delta\in\partial B(z^\delta,\frac{\eps}{2})$ converges to $w\in\partial B(z,\frac{\eps}{2})$ and $s^\delta \to s>0$ as $\delta\to 0$.
By the convergence of the law of $\hat \LR$ to the law of $\hat{B}$ and the local central limit theorem (e.g.,~\cite[Theorem~2.1.1]{Lawler-Limic:Random_walk_modern_introduction}), we~have 
\begin{align*} 
\; & \lim_{\delta\to 0}\frac{1}{\delta^2}\,\Prob_{w^\delta}\big[\LR(s^\delta) = z^\delta\textnormal{ and }\LR[0,s^\delta]\subset B(z^\delta,\eps)\big]
\\
=\; & \hat{p}(w,z;s) \, - \, \hat\ProbEBM_w \big[\one{\{\hat\exitT_\eps<s\}} \, \hat{p}(\hat{B}(\hat\exitT_\eps),z;s-\hat\exitT_\eps) \big].
\end{align*}
Moreover, there exists a constant $C=C(\eps)>0$, such that 
\begin{align}\label{eqn::loop_conv_aux4}
\Prob_{w^\delta}\big[\LR(s^\delta) = z^\delta\textnormal{ and }\LR[0,s^\delta]\subset B(z^\delta,\eps)\big]
\le \ProbHat_{w^\delta}\big[\hat\LR(s^\delta) = z^\delta\big] 
\le C\delta^2 .
\end{align}

\item
By~\eqref{eqn::loop_conv_aux1} and the strong Markov property, we have 
\begin{align}\label{eqn::exp_time}
\Prob_{v}[T_j^\delta<\infty]\le (1-q)^j , \qquad v\in\Omega^\delta , \; \delta\in(0,\delta_0) ,  \; j\in \N .
\end{align}
\end{itemize}
Combining these four observations, taking the test function $f \equiv 1$ 
we obtain the convergence of the rescaled probability 
$\delta^{-2}\,\Prob_{z^\delta}[\LR(t^\delta)=z^\delta]$, as well as
the convergence under the conditional probability measure $\Prob_{z^\delta}[\, \cdot \cond \LR(t^\delta)=z^\delta]$ of the random walk $\LR[0,c]$ to the Brownian motion $B^{\textnormal{ND}}[0,c]$ for every $0<c<t$.

\medskip 
It remains to prove that for every $\eps,\eps'>0$, we can choose $c>0$ which is close to $t$ such that 
\begin{align}\label{eqn::continuity}
\Prob_{z^\delta}\Big[\max_{u\in [c,t^\delta]} |\LR(u)-z^\delta|> \eps \Bigcond \LR(t^\delta)=z^\delta\Big]\le \eps' ,
\end{align}
which is a sort of a continuity estimate. To this end,
note that there exists a constant $C=C(\eps)>0$ such that for all $v\in\partial B(z^\delta,\eps)$, and 
$s \in [0,t^\delta-c]$, and for every $\delta\in(0,\delta_0)$, we have
\begin{align}\label{eqn::loop_conv_aux_RW}
\Prob_{v}[\LR(s)=z^\delta] 
=\;&\sum_{j = 1}^\infty\Prob_{v}[\LR(s)=z^\delta, \, T^\delta_j<\infty, \textnormal{ and } S^\delta_j=\infty]\notag\\
=\;&\sum_{j = 1}^\infty\ProbE_{v}\Big[[\one{\{T_j^\delta<\infty\}} \, \Prob_{\LR(T_j^\delta)}\big[\LR(s-T_j^\delta) = z^\delta\textnormal{ and }\LR[0,s-T_j^\delta] \subset B(z^\delta,\eps) \big]\Big]\notag\\
\le \;&\sum_{j = 1}^\infty\Prob_{v}[T_j^\delta<\infty]\;\max_{u \in [0,t^\delta-c]}
\max_{w\in\partial B(z^\delta,\eps/2)} \ProbHat_{w}[\hat\LR(u)=z^\delta]\notag\\
\le \;&C\delta^2\, \exp\Big( \! -\frac{\eps^2}{2(t^\delta-c)} \Big) .
\qquad \textnormal{[by the local central limit theorem and~\eqref{eqn::exp_time}]}
\end{align}
We can conclude by the following two estimates: 
\begin{itemize}
\item
when $\LR(c) \in B(z^\delta,\frac{\eps}{2})$, we have
\begin{align*}
\Prob_{z^\delta}\Big[\max_{s\in [c,t^\delta]} |\LR(s)-z^\delta|>\eps , \, \LR(c)\in B(z^\delta,\tfrac{\eps}{2}), \textnormal{ and } \LR(t^\delta)=z^\delta\Big]
\le\;&\max_{s\in [0,t^\delta-c]} \max_{v\in\partial B(z^\delta,\eps)} \Prob_{v}[\LR(s)=z^\delta] ;
\end{align*}

\item
when $\LR(c)\notin B(z^\delta,\frac{\eps}{2})$, we have

\begin{align*}
\Prob_{z^\delta}\Big[\max_{s\in [c,t^\delta]} |\LR(s)-z^\delta|>\eps , \, \LR(c)\notin B(z^\delta,\tfrac{\eps}{2}), \textnormal{ and } \LR(t^\delta)=z^\delta\Big] 
\le\;&\max_{s\in [0,t^\delta-c]} \max_{v\in\partial B(z^\delta,\eps/2)} \Prob_{v}[\LR(s)=z^\delta].
\end{align*}
\end{itemize}
Combining this with~\eqref{eqn::loop_conv_aux_RW} and the convergence of $\delta^{-2}\,\Prob_{z^\delta}[\LR(t^\delta)=z^\delta]$, we obtain~\eqref{eqn::continuity}.
\end{proof}

Next, we control the lifetimes $t_\ell$ of loops $\ell$ under $\Lambda_{\Omega^\delta}^{\textnormal{ND}}$. 
For every $c>0$, we denote  the $c$-neighbourhood of the inner boundary $\BdryIn^\delta$ by $\BdryInNbhdDiscr{c} := \{v\in\LV^\delta \;|\; \dist(v,\BdryIn^\delta)<c\}$.

\begin{lemma}\label{lem::time_control}
For every $\varepsilon>0$, there exist $\delta_0=\delta_0(\varepsilon)>0$ and 
a constant $M=M(\varepsilon)>1$ such that $\smash{\underset{\varepsilon \to 0}{\lim} \, M(\varepsilon) \to \infty}$ and 
\begin{align}\label{eqn::time_bound_control}
\Lambda_{\Omega^\delta}^{\textnormal{ND}}\big[\ell\in\LL_*^\delta \textnormal{ and its lifetime } t_\ell\notin \big[\tfrac{1}{M},M\big] \big]<\varepsilon , \qquad \textnormal{for all } \delta \in (0, \delta_0) . 
\end{align}
Furthermore, there exists a constant $c=c(\varepsilon)>0$ such that
$\smash{\underset{\varepsilon \to 0}{\lim} \, c(\varepsilon) \to 0}$ and 
\begin{align}\label{eqn::root_bound_control}
\Lambda_{\Omega^\delta}^{\textnormal{ND}}\big[\ell\in\LL_*^\delta \textnormal{ and the root of }\ell\textnormal{ lies in }\BdryInNbhdDiscr{c}\big]<\varepsilon , \qquad \textnormal{for all } \delta \in (0, \delta_0) .  
\end{align}
Moreover, the estimates~\textnormal{(\ref{eqn::time_bound_control},~\ref{eqn::root_bound_control})} are uniform with respect to the shape of the domains $\Omega^\delta$.
\end{lemma}

To prove Lemma~\ref{lem::time_control}, we need two auxiliary results. 
\begin{itemize}[leftmargin=*]
\item
In Lemma~\ref{lem::Green_mix}, we bound the discrete Green's function $\smash{\Gren_{\Omega^\delta}^{\textnormal{ND}}}$ on $\Omega^\delta$ with Dirichlet boundary condition on $\BdryOut^\delta$ and Neumann boundary condition on $\BdryIn^\delta$, defined as
\begin{align*} 
\Gren_{\Omega^\delta}^{\textnormal{ND}}(z,w)
:=\sum_{j=0}^\infty\Prob_{z}[\LR(j)=w \textnormal{ and } \exitTOut^\delta>j] , \qquad z,w\in \LV^\delta ,
\end{align*} 
where $\exitTOut^\delta$ is the hitting time at $\BdryOut^\delta$ of the random walk $\LR \sim \Prob_{z} = \Prob_{z}[\, \cdot \cond \LR(0)=z]$. 

\item
In Lemma~\ref{lem::root_neighbourhood}, we bound the measure under $\Lambda_{\Omega^\delta}^{\textnormal{ND}}$ of loops in $\BdryInNbhdDiscr{c}$ with non-trivial winding: 
it is very unlikely for small loops to have winding even with Neumann boundary condition on $\BdryIn^\delta$. 
\end{itemize}

\begin{lemma}\label{lem::Green_mix}
For every $c>0$, there exist $\delta_0=\delta_0(c)>0$ and a constant $C=C(c)>0$ such that 
\begin{align*}
\Gren_{\Omega^\delta}^{\textnormal{ND}}(z,w)\le C , \qquad \textnormal{for all } z,w\in\LV^\delta
\textnormal{ with } |z-w|\ge c  \textnormal{ and } w\notin\BdryInNbhdDiscr{c} ,
\textnormal{ and for all } \delta \in (0, \delta_0) . 
\end{align*}
Moreover, the estimate is uniform with respect to the shape of the domains $\Omega^\delta$, provided that $\Omega^\delta \subset B(0,R)$.
\end{lemma}

\begin{proof}
Suppose $\Omega^\delta \subset B(0,R)$ for some $R>0$. Fix $w\in\LV^\delta\setminus\BdryInNbhdDiscr{c}$.  
We know that $\Gren_{\Omega^\delta}^{\textnormal{ND}}(\cdot,w)|_{\BdryIn^\delta}$ takes its maximum at some $v_{\scalebox{0.6}{\textnormal{in}}}\in\BdryIn^\delta$, 
and $\Gren_{\Omega^\delta}^{\textnormal{ND}}(\cdot,w)|_{\Omega^\delta\setminus B(w,\tfrac{c}{2})}$ takes its maximum at some $v_c\in\partial B(w,\tfrac{c}{2})$.

On the one hand, we have
\begin{align*}
\Gren_{\Omega^\delta}^{\textnormal{ND}}(v_{\scalebox{0.6}{\textnormal{in}}},w) 
\le \; & \Prob_{v_{\scalebox{0.6}{\textnormal{in}}}} \big[\LR\textnormal{ hits }\partial B(w,\tfrac{c}{2})\textnormal{ before hitting }\BdryOut^\delta\big] 
\max_{u^\delta\in \partial B(w,\tfrac{c}{2})}\Gren_{\Omega^\delta}^{\textnormal{ND}}(u^\delta,w) \\
\le \; & (1-q) \, \Gren_{\Omega^\delta}^{\textnormal{ND}}(v_c,w) ,
\end{align*}
where $0<q=q(c)<1$ and $\delta \in (0, \delta_0)$ (uniformly on $\Omega^\delta$ subject to the above conditions). 
Hence, we have
\begin{align}\label{eqn::boundary_value}
q \, \Gren_{\Omega^\delta}^{\textnormal{ND}} (v_c, w) \le \Gren_{\Omega^\delta}^{\textnormal{ND}}(v_c,w) - \Gren_{\Omega^\delta}^{\textnormal{ND}}(v_{\scalebox{0.6}{\textnormal{in}}},w).
\end{align}
On the other hand, by the strong Markov property, we have
\begin{align}\label{eqn::boundary_value2}
\Gren_{\Omega^\delta}^{\textnormal{ND}}(v_c,w)-\Gren_{\Omega^\delta}^D(v_c,w)\le \Gren_{\Omega^\delta}^{\textnormal{ND}}(v_{\scalebox{0.6}{\textnormal{in}}},w) ,
\end{align}
where $\Gren^{\textnormal{D}}_{\Omega^\delta}$ is the discrete Green's function on $\Omega^\delta$ with Dirichlet boundary condition on $\partial\Omega^\delta$, defined as 
\begin{align*} 
\Gren_{\Omega^\delta}^{\textnormal{D}}(z,w)
:=\sum_{j=0}^\infty\Prob_{z}[\hat\LR(j)=w \textnormal{ and } \hat\exitT^\delta>j] , \qquad z,w\in \LV^\delta ,
\end{align*}
with $\hat\exitT^\delta$ the hitting time at $\partial\Omega^\delta$ of the standard random walk $\hat\LR$ on $\delta \Z^2$.
Combining~(\ref{eqn::boundary_value},~\ref{eqn::boundary_value2}) and using the strong Markov property, we find that
\begin{align*}
\Gren_{\Omega^\delta}^{\textnormal{ND}}(z,w)\le \Gren_{\Omega^\delta}^{\textnormal{ND}}(v_c,w)
\le \frac{1}{q} \, \Gren_{\Omega^\delta}^{\textnormal{D}}(v_c,w).
\end{align*}
The claim now follows from a standard bound for $\Gren_{\Omega^\delta}^{\textnormal{D}}$
(for instance, see~\cite[Theorem~2.5]{Chelkak-Smirnov:Discrete_complex_analysis_on_isoradial_graphs}; this bound is uniform on $\Omega^\delta$ subject to the above conditions by domain monotonicity of the Green's function with Dirichlet boundary conditions).
\end{proof}

\begin{lemma}\label{lem::root_neighbourhood}
For every $\varepsilon>0$, there exist $\delta_0=\delta_0(\varepsilon)>0$ and a constant $c_0=c_0(\varepsilon)>0$ such that 
\begin{align*}
\Lambda_{\Omega^\delta}^{\textnormal{ND}}\big[\ell\in\LL_*^\delta \textnormal{ such that } \ell\subset\BdryInNbhdDiscr{c}\big]<\varepsilon , \qquad 0<c\le c_0 , \qquad \textnormal{for all } \delta \in (0, \delta_0) . 
\end{align*}
\end{lemma}

\begin{proof}
We bound the left-hand side by controlling the variance of the winding $K(\LR)$ of a random walk $\LR \sim \Prob_{x,v,c}$ starting from $x = x_1^\delta$, conditioned to leave $\BdryInNbhdDiscr{c}$ through a given inner vertex $v \in \BdryInNbhdDiscr{c}$ on the outer boundary of $\BdryInNbhdDiscrGr{c} := \Omega^\delta \cap \BdryInNbhdDiscrDom{c}$, 
where $\BdryInNbhdDiscrDom{c} := \textnormal{Int} ({\underset{v\in \BdryInNbhdDiscr{c}}{\cup}} \, \overline{ f(v)})$, and stopped there.  

Let $\Gamma = \mathrm{LE}(\LR)$ be the loop-erasure of $\LR$. 
By~\cite[Proposition~9.5.1]{Lawler-Limic:Random_walk_modern_introduction}, we have
\begin{align*}
\Lambda^{\textnormal{ND}}_{\BdryInNbhdDiscrGr{c}} \big[(\varphi(\ell))^2 \, \one_{\{\varphi(\ell) \neq 0\}}]
\; = \; 
4\pi^2\big(\textnormal{Var}_{x,v, c}[K(\LR)]-\textnormal{Var}_{x, v, c}[K(\Gamma)| ]\big)
\; \le \; 
4\pi^2 \textnormal{Var}_{x,v,c}[K(\LR)].
\end{align*}
By Markov's inequality, we have
\begin{align*}
\Lambda_{\Omega^\delta}^{\textnormal{ND}}\big[\ell\in\LL_*^{\delta}  \textnormal{ such that } \ell\subset\BdryInNbhdDiscrGr{c}\big] 
\le  \textnormal{Var}_{x, v, c}[K(\LR)] , 
\end{align*} 
and by taking sum over $v$, we obtain
\begin{align*}
\Lambda_{\Omega_\delta}^{ND}\big[\ell\in\LL_*^{\delta}  \textnormal{ such that } \ell\subset\BdryInNbhdDiscrGr{c}\big] 
\le 4\pi^2\textnormal{Var}[K(\LR)].
\end{align*} 
Thus, to conclude, it suffices to prove that for every $\varepsilon'>0$, we can choose $c$ small enough such that 
\begin{align}\label{eqn::wind_control_aux1}
\Prob_{x}\big[|K(\LR)|\ge j\big] \le (\varepsilon')^j , \qquad j\ge 1,
\end{align}
uniformly in $x \in \BdryInNbhdDiscrGr{c}$. Consider two crosscuts $\simpleCurv_1, \simpleCurv_2$ connecting the two boundary components of $\BdryInNbhdDiscrGr{c}$ and such that
they are at macroscopic distance away from each other and from $x$.
If the LERW $\Gamma$ has non-trivial winding, it must cross both crosscuts. 
Hence, we consider stopping times for the random walk $\LR$, defined recursively as
$T_1:=\inf\{t\ge 0 \colon \LR(t)\in \simpleCurv_1\}$; 
and $S_j:=\inf\{t\ge T_{j} \colon \LR(t)\in\simpleCurv_2\}$; 
and $T_{j+1}:=\inf\{t\ge S_{j} \colon \LR(t)\in\simpleCurv_1\}$, for $j\ge 1$.
Now, using the the discrete Beurling estimate and the strong Markov property of $\LR$, we see that
\begin{align*}
\Prob_{x}\big[|K(\mathcal{R})|\ge j\big] \le \Prob_{x}[0<T_1<\cdots<S_{j-1}<\infty]\le(\varepsilon')^j ,
\end{align*}
whenever $c = c(\varepsilon')$ and $\delta_0=\delta_0(\varepsilon')>0$ are small enough. 
This implies the asserted estimate. 
\end{proof}

We are now ready to finish with Lemma~\ref{lem::time_control}. 

\begin{proof}[Proof of Lemma~\ref{lem::time_control}]
Fix $\varepsilon>0$. 
To prove the first asserted estimate~\eqref{eqn::time_bound_control}, thanks to Lemma~\ref{lem::root_neighbourhood} it will be sufficient to consider loops exiting a neighbourhood of the inner boundary $\BdryIn^\delta$ (cf.~\eqref{eqn::truncated_time_bound}).
Denote $c_\Omega := \tfrac{1}{100} \dist(\BdryIn, \BdryOut)$. For $c\in (0, c_\Omega )$, we define two stopping times for the random walk $\LR$ on $\Omega^\delta$: 
\begin{itemize}
\item
if $\LR(0) \notin \BdryInNbhdDiscr{c}$, then we set $\sigma_1=\sigma_1(c):=0$; 
\item
if $\LR(0) \in \BdryInNbhdDiscr{c}$, then we set $\sigma_1=\sigma_1(c)$ to be the first time when the walk 
$\LR$ exits $\BdryInNbhdDiscr{2c}$; and  
\begin{align*}
\sigma_2=\sigma_2(c):=\inf\{t\ge 0 \colon |\LR(\sigma_2)-\LR(\sigma_1)|\ge c/2\} . 
\end{align*}
\end{itemize}
In both cases, we have $\min\{ |\LR(\sigma_2)-\LR(\sigma_1)| , \, |\LR(\sigma_2)-\LR(0)| \} \ge c/2$ 
and $\LR(\sigma_2)\notin\BdryInNbhdDiscr{c/2}$.  
Note that any uncontractible loop $\ell\not\subset\BdryInNbhdDiscr{2c}$ has diameter at least $c$.
With Lemma~\ref{lem::root_neighbourhood} at hand, 
the crucial claim is that for each $\varepsilon'>0$, we can choose $M'=M'(\varepsilon', c)>1$ large enough and $\delta_0'=\delta_0'(\varepsilon')>0$ small enough 
 such that 
\begin{align}\label{eqn::truncated_time_bound}
\Lambda_{\Omega^\delta}^{\textnormal{ND}}\big[\ell\not\subset\BdryInNbhdDiscr{2c}, \, \diam (\ell) \ge c, 
\textnormal{ and its lifetime }t_\ell\notin\big[\tfrac{1}{M'},M'\big] \big] 
\le \varepsilon' \; \max_{\substack{w\notin\BdryInNbhdDiscr{c/2} \\ |w-z|\ge c/2}} \Gren^{\textnormal{ND}}_{\Omega^\delta}(z,w) ,
\end{align}
for all $\delta \in (0, \delta_0')$ (where the maximum of the Green's function is controlled by Lemma~\ref{lem::Green_mix}). 
Note that 
\begin{align*}
& \; 
\Lambda_{\Omega^\delta}^{\textnormal{ND}}\big[\ell\not\subset\BdryInNbhdDiscr{2c}, \, 
\diam (\ell) \ge c, 
\textnormal{ and its lifetime }t_\ell\notin\big[\tfrac{1}{M'},M'\big] \big] 
\\
\le & \; \sum_{z\in\LV^\delta\setminus\BdryOutV^\delta} 
\frac{\delta^2}{2}\sum_{t\in \frac{1}{2}\delta^2\N \setminus[\frac{1}{M'},M']}\frac{1}{t} \, \Prob_{z}\big[\LR[0,t] \not\subset\BdryInNbhdDiscr{2c},\, 
\diam\LR[0,t]\ge c, 
 \textnormal{ and } \LR(t) = z\big] 
 \\
\le & \; \sum_{z\in\LV^\delta\setminus\BdryOutV^\delta} 
\frac{\delta^2}{2}\sum_{t\in \frac{1}{2}\delta^2\N \setminus[\frac{1}{M'},M']}\frac{1}{t} \, \Prob_{z}\big[0\le\sigma_1<\sigma_2<t  \textnormal{ and } \LR(t)=z\big].
\end{align*}
We split the sum over $t$ into the following two parts.
\begin{itemize}
\item
For $t\ge M'$, we have
\begin{align}\label{eqn::large_time_aux}
\; &\sum_{z\in\LV^\delta\setminus\BdryOutV^\delta}
\frac{\delta^2}{2}\sum_{t\in \frac{1}{2}\delta^2\N \cap [M',\infty) }\frac{1}{t} \, \Prob_{z}\big[0\le\sigma_1<\sigma_2<t  \textnormal{ and } \LR(t)=z\big]\notag\\
\le \; & \frac{1}{M'} \sum_{z\in\LV^\delta\setminus\BdryOutV^\delta}\frac{\delta^2}{2}
\max_{\substack{w\notin\BdryInNbhdDiscr{c/2} \\ |w-z|\ge\frac{c}{2}}}
\sum_{j = 1}^\infty \Prob_{w}[\LR(j)=z]\notag\\
= \; & \frac{1}{M'} \sum_{z\in\LV^\delta\setminus\BdryOutV^\delta}\frac{\delta^2}{2}
\max_{\substack{w\notin\BdryInNbhdDiscr{c/2} \\ |w-z|\ge\frac{c}{2}}}
\Gren^{\textnormal{ND}}_{\Omega^\delta}(w,z) 
\; \le \; \frac{\Leb(\tilde \Omega^\delta)}{2M'} 
\max_{\substack{w\notin\BdryInNbhdDiscr{c/2} \\ |w-z|\ge\frac{c}{2}}}
\Gren^{\textnormal{ND}}_{\Omega^\delta}(w,z) ,
\end{align}
where $\Leb(\tilde \Omega^\delta)$ is the Lebesgue measure of the domain\footnote{Note that $\Leb(\tilde \Omega^\delta)$ is uniformly controlled by $\Leb(\Omega)$ by the assumption in Item~\ref{setup2} of the Setup, when $\delta$ is small.}
$\tilde \Omega^\delta := \textnormal{Int} (\smash{\underset{v\in \LV^\delta}{\cup}} \, \overline{ f(v)}) \subset \C$. 

\item
For $t\le\frac{1}{M'}$, we have
\begin{align}
\; & \sum_{z\in\LV^\delta\setminus\BdryOutV^\delta}
\frac{\delta^2}{2}\sum_{t\in \frac{1}{2}\delta^2\N \cap (0,\frac{1}{M'}]} 
\frac{1}{t} \, \Prob_{z}\big[0\le\sigma_1<\sigma_2<t   \textnormal{ and } \LR(t)=z\big]\notag\\
\le \; & \sum_{z\in\LV^\delta\setminus\BdryOutV^\delta}
\frac{\delta^2}{2} \, \ProbE_{z}\bigg[\frac{\one{\{0\le\sigma_1<\sigma_2<\tfrac{1}{M'}\}}}{\sigma_2-\sigma_1} \sum_{j = 1}^\infty \Prob_{\LR(\sigma_2)}\big[\LR(j)=z\big]\bigg]\notag\\
\le \; & \frac{\Leb(\tilde \Omega^\delta)}{2} \max_{z\in\LV^\delta} 
\ProbE_{z}\bigg[\frac{\one{\{0\le\sigma_1<\sigma_2<\tfrac{1}{M'}\}}}{\sigma_2-\sigma_1}\bigg]\,
\max_{\substack{w\notin\BdryInNbhdDiscr{c/2} \\ |w-z|\ge\frac{c}{2}}}
\Gren^{\textnormal{ND}}_{\Omega^\delta}(w,z).
\label{eqn::small_time_aux}
\end{align}
Let us denote by $\hat\LR \sim \ProbHat_{z} = \ProbHat_{z}[\, \cdot \cond \hat\LR(0)=z]$ 
the standard random walk on $\delta\Z^2$, and by $T_c$ the first time when $\hat\LR$ exits $\partial B(0,c/2)$. 
Combining with~\cite[Proposition~2.1.2(b)]{Lawler-Limic:Random_walk_modern_introduction} 
we find that there exists two constants $K_1>0$ and $K_2>0$ such that for each $z\in\LV^\delta$, we have
\begin{align}\label{eqn::CLT_bound_time}
\ProbE_{z}\bigg[\frac{\one{\{0\le\sigma_1<\sigma_2<\tfrac{1}{M'}\}}}{\sigma_2-\sigma_1}\bigg]\le \ProbEHat_0\bigg[\frac{\one{
\{T_c<\tfrac{1}{M'}\}}}{T_c}\bigg]\le K_1  \, \exp \Big( \! -K_2 \, c^2 \, M' \Big) . 
\end{align}
Plugging~\eqref{eqn::CLT_bound_time} into~\eqref{eqn::small_time_aux}, we obtain
\begin{align}\label{eqn::small_time}
\; & \sum_{z\in\LV^\delta\setminus\BdryOutV^\delta}
\frac{\delta^2}{2} \sum_{t\in \frac{1}{2}\delta^2\N \cap (0,\frac{1}{M'}]} 
\frac{1}{t} \, \Prob_{z}\big[0\le\sigma_1<\sigma_2<t  \textnormal{ and } \LR(t)=z\big]\notag\\
\le \; &  \frac{\Leb(\tilde \Omega^\delta)}{2} \, K_1  \, \exp \Big( \! -K_2 \, c^2 \, M' \Big)
\max_{\substack{w\notin\BdryInNbhdDiscr{c/2} \\ |w-z|\ge\frac{c}{2}}}
\Gren^{\textnormal{ND}}_{\Omega^\delta}(w,z).
\end{align}
\end{itemize}
By choosing $M'$ large enough and $\delta_0'$ small enough, these two estimates imply~\eqref{eqn::truncated_time_bound}. 
Now, the asserted estimate~\eqref{eqn::time_bound_control} follows 
taking $c = (c_0/2) \wedge c_\Omega$ with $c_0=c_0(\varepsilon/2)>0$ as in Lemma~\ref{lem::root_neighbourhood}, and
$\varepsilon' = \varepsilon/(2C(c/2))$ with $C = C(c/2)$ as in Lemma~\ref{lem::Green_mix}, 
and choosing $\delta_0=\delta_0(\varepsilon)>0$ suitably small.

\medskip It remains to prove~\eqref{eqn::root_bound_control}. 
Again, thanks to Lemma~\ref{lem::root_neighbourhood} we only consider loops exiting a neighbourhood of the inner boundary $\BdryIn^\delta$. 
Moreover, thanks to Equation~\eqref{eqn::time_bound_control}, we may restrict to lifetimes in $\big(\tfrac{1}{M},M\big)$, 
where $M=M(\varepsilon)$ and $\delta_0=\delta_0(\varepsilon)$ are suitably chosen. 
As in~(\ref{eqn::large_time_aux},~\ref{eqn::small_time_aux}), we obtain
\begin{align}\label{eqn::root_time_bound_control_aux}
& \; 
\Lambda_{\Omega^\delta}^{\textnormal{ND}}\big[\ell\not\subset\BdryInNbhdDiscr{2c}, \, \diam (\ell) \ge c, \textnormal{ its lifetime }t_\ell\in\big(\tfrac{1}{M},M\big) ,
\textnormal{ and its root lies in } \BdryInNbhdDiscr{c'} \big] 
\notag\\
\le & \;  M 
\sum_{z\in \BdryInNbhdDiscr{c'} }
\frac{\delta^2}{2}\sum_{j = 1}^\infty 
 \ProbE_{z} \big[\Prob_{\LR(\sigma_2)}\big[\LR(j)=z\big]\big]
\notag \\
\le & \;  \frac{M}{2} \, \Leb (\BdryInNbhdDiscrDom{c'}) 
\; \max_{\substack{w\notin\BdryInNbhdDiscr{c/2} \\ |w-z|\ge\frac{c}{2}}} \Gren^{\textnormal{ND}}_{\Omega^\delta}(z,w) ,
\qquad  \; c' \in (0,c) , 
\end{align} 
where\footnote{Note that the Lebesgue measure is uniformly controlled by the Lebesgue measure of the corresponding neighbourhood for the inner boundary of $\Omega$, by the assumption in Item~\ref{setup2} of the Setup, when $\delta$ is small.} 
$\BdryInNbhdDiscrDom{c'} := \textnormal{Int} (\smash{\underset{v\in \BdryInNbhdDiscr{c'}}{\cup}} \, \overline{ f(v)}) \subset \C$. 
Choosing $c'$ small enough,~\eqref{eqn::root_bound_control} follows. 
\end{proof}

Combining Proposition~\ref{prop::loop_conv} and Lemma~\ref{lem::time_control} together, we can now deduce the asymptotics of the loop soup terms in Proposition~\ref{prop::total_winding}. This is encapsulated by the following result.

\begin{lemma}\label{lem:BLM_conv}
For any $\beta \in \R$, we have 
\begin{align*}
\lim_{\delta\to 0}\Lambda_{\Omega^\delta}^{\textnormal{ND}}\big[\ee^{\ii\beta\ph(\ell)} \, \one{\{\ell\in\LL_*^\delta\}}\big]
= \;& \mu_\Omega^{\textnormal{ND}} \big[ \ee^{\ii\beta\ph(\ell)} \one{\{\ell\in\LL_*\}} \big]  .
\end{align*}
\end{lemma}

\begin{proof}
For $\varepsilon>0$, take $c=c(\varepsilon)$, and $M=M(\varepsilon)$, 
and $\delta_0=\delta_0(\varepsilon)$ as in Lemma~\ref{lem::time_control}. 
Consider the set 
\begin{align*}
\LA_*^\delta(M,c)
:=\Big\{\ell\in\LL_*^\delta \;|\; \textnormal{ the lifetime } t^\delta\in\big(\tfrac{1}{M},M\big) ,\textnormal{ and the root of }\ell\textnormal{ lies not in }\BdryInNbhdDiscr{c}\Big\} 
\end{align*} 
of discrete loops, and define $\BdryInNbhd{c}:=\{z\in\Omega \;|\; \dist(z,\BdryIn)<c\}$.  
For $\beta \in \R$, we have
\begin{align*}
\lim_{\delta\to 0}\Lambda_{\Omega^\delta}^{\textnormal{ND}}\big[\ee^{\ii\beta\ph(\ell)} \, \one{\{\LA_*^\delta(M,c)\}}\big]
= \;& 
\lim_{\delta\to 0}\sum_{\substack{z^\delta\in\LV^\delta\setminus\BdryOutV^\delta \\ z^\delta\not\in\BdryInNbhdDiscr{c}}}\frac{\delta^2}{2}\sum_{\substack{t^\delta\in \frac{1}{2}\delta^2\N\\\frac{1}{M}\le t^\delta\le M}}\frac{1}{t^\delta}\,
\Prob_{z^\delta}\Big[\ee^{\ii\beta\ph(\ell)} \, \one{\{\ell\in\LL_*^\delta,\, \LR(t^\delta)=z^\delta\}} \Big]\\
= \;& \int_{\Omega \setminus \BdryInNbhd{c}} \ud z \int_{[\tfrac{1}{M},M]} \ud t \; \frac{p^{\textnormal{ND}}_{\Omega}(z,z,t)}{t} \;
\mu^{\textnormal{ND}}_{\Omega}(z,z,t)\big[\ee^{\ii\beta\ph(\ell)} \, \one{\{\ell\in\LL_*\}}\big]  , 
\end{align*}
by Proposition~\ref{prop::loop_conv}, Lemma~\ref{lem::time_control}, and~\eqref{eqn::loop_conv_aux4}.
Taking $\varepsilon\to 0$, so $M(\varepsilon)\to\infty$ and $c(\varepsilon)\to 0$, we obtain 
\begin{align*}
\lim_{\delta\to 0}\Lambda_{\Omega^\delta}^{\textnormal{ND}}\big[\ee^{\ii\beta\ph(\ell)}\one{\{\ell\in\LL_*^\delta\}}\big] 
=\;& \int_{\Omega} \ud z \int_0^\infty \ud t \; \frac{p^{\textnormal{ND}}_{\Omega}(z,z,t)}{t}
\; \mu^{\textnormal{ND}}_{\Omega}(z,z,t)\big[\ee^{\ii\beta\ph(\ell)} \one{\{\ell\in\LL_*\}}\big] \\
=\;& \mu_\Omega^{\textnormal{ND}}\big[\ee^{\ii\beta\ph(\ell)} \one{\{\ell\in\LL_*\}}\big] , \qquad \beta \in \R .
\qquad \textnormal{[by definition~\eqref{eqn::mixed_loop}]}
\end{align*}
This completes the proof. 
\end{proof}

To complete the proof of Theorem~\ref{thm::total_winding}, it remains to deal with the ratio of determinants in the formulas in Proposition~\ref{prop::total_winding}. 
The asymptotics will be obtained in Lemma~\ref{lem::deltato0odd}.

\medskip
Recall that we denote by $\OmegaUniv^\delta = (\LVUniv^\delta,\LEUniv^\delta)$ the universal cover of the graph $\Omega^\delta$, obtained by gluing its copies $\{\Omega^\delta_m\}_{m \in \Z} = \{(\LV^\delta_m,\LE^\delta_m)\}_{m \in \Z}$ together along the cut $\zeta$. 
In particular, there are bijections $\LV^\delta \leftrightarrow \LV^\delta_m$ and $\LE^\delta \leftrightarrow \LE^\delta_m$ given by $v \leftrightarrow v_m$ and $e \leftrightarrow e_m$, and $h^\delta_\beta(\cdot, v)$ defined in~\eqref{eq:h_beta} defines a discrete harmonic function on $\LVUniv^\delta$. 
For each subset $V \subset \LVUniv^\delta$ of vertices and for each vertex $x\in\LVUniv^\delta$, define
\begin{align*}
\HarmUniv_{\Omega^\delta}(x,V) := \PP[\textnormal{ random walk on }\OmegaUniv^\delta\textnormal{ starting from }x\textnormal{ hits }V\textnormal{ before hitting }\BdryOutUniv^\delta\setminus V] , 
\end{align*} 
which is discrete harmonic on $\LVUniv^\delta$ with $\HarmUniv_{\Omega^\delta}(\cdot,V) = 1$ on $V$
and $\HarmUniv_{\Omega^\delta}(\cdot,V) = 0$ (elsewhere) on the outer boundary.
When $V=\{v\}$ is a singleton with $v\in\BdryOutEUniv^\delta$, we simply denote 
$\HarmUniv_{\Omega^\delta}(x,\{v\}) = \HarmUniv_{\Omega^\delta}(x,v)$.

\medskip 
Let $\phi_r \colon \Omega \to \A_r$ be the conformal isomorphism onto the annulus 
$\A_r = A(r,1)$,
so that $\phi_r$ coincides with the restriction to $\Omega$ of the conformal isomorphism 
$\phi \colon \Omega \to \D$ used above. The universal cover of $\A_r$ is the strip
$S_r := \{z=x+\ii y \in \C \cond x\in\R,\, 0< y< -\log r\}$, with covering map is given by $p(z):=e^{-\ii z+\log r}$ for $z\in S_r$. Denote by $\OmegaUniv$ the universal cover of $\Omega$ obtained by gluing $\Omega$ with its copies along $\zeta:=\phi_r^{-1}[r,1]$, and by $\BdryOutUniv$ its outer boundary. 
Finally, let $\phiUniv_r \colon \OmegaUniv \to S_r$ be the associated lift of $\phi_r$. 
Denote by $\Pois_{S_r}$ the \textbf{Poisson kernel} on $S_r$ with Dirichlet boundary condition on $\R - \ii\log r$ and Neumann boundary condition on $\R$: the function $\Pois_{S_r}(\cdot,\cdot)$ is smooth on the set 
\begin{align*}
(\overline{S_r} \times\overline{S_r})\setminus\{(z,z) \;|\; z \in \R - \ii\log r\};
\end{align*}
for every $z\in\R - \ii\log r$, the function $\Pois_{S_r}(\cdot,z)$ is harmonic on $S_r$, and
\begin{align*}
\Pois_{S_r}(\cdot,z)|_{(\R - \ii\log r) \setminus\{z\}}=0
\qquad\textnormal{and}\qquad 
\partial_n \Pois_{S_r}(\cdot,z)|_{\R}=0.
\end{align*}
Note that $\Pois_{S_r}(\cdot,\cdot)$ is well-defined up to a multiplicative constant. 
We choose the normalisation as
\begin{align*}
\int_{-\infty}^{\infty}\Pois_{S_r}(\cdot,x-\ii\log r)\, \ud x = 1.
\end{align*}
Then, we have the explicit formula
\begin{align*}
\Pois_{S_r}(x_1,x_2-\ii\log r)
= \frac{1}{2|\log r|}\textnormal{sech}\Big(\frac{\pi \, (x_2-x_1)}{2|\log r|}\Big) .
\end{align*}

As before, we fix $2n$ distinct vertices $x_1^\delta,\ldots,x_n^\delta \in \BdryIn^\delta$ 
and $v_1^\delta,\ldots,v_n^\delta \in \BdryOut^\delta$, 
by convention labelled in counterclockwise order, 
and we assume that as $\delta \to 0$, each $x^\delta_j$ converges to $x_j\in\BdryIn$ and each $v^\delta_j$ converges to $v_j\in\BdryIn$,
where $x_1,\ldots,x_n \in \BdryIn$ and $v_1,\ldots,v_n \in\BdryOut$ are distinct.
(For the different copies in the universal over, we only use the subscript ``$m$'' trusting that it is clear from context.)
Choosing the branch of the argument function $\arg$ to take values in $[-\pi,\pi)$, set
\begin{align*}
c := \sum_{j=1}^{n}\arg \phi_r(x_j)-\sum_{j=1}^{n}\arg \phi_r(e_j) .
\end{align*}
Define also
\begin{align*}
\Upsilon(\bs{x},\bs{v}) := \prod_{1\le i<j\le n}\sin\bigg(\frac{\arg\phi(x_j)-\arg\phi(x_i)}{2}\bigg)
\prod_{1\le i<j\le n}\sin\bigg(\frac{\arg\phi(v_j)-\arg\phi(v_i)}{2}\bigg) .
\end{align*}

\begin{lemma}\label{lem::deltato0odd}
Suppose $z^\delta \in\LV^\delta$ converge to $z\in\Omega\setminus\zeta$ as $\delta\to 0$. Then, we have
\begin{align}\label{eqn::scaling_det_odd}
\lim_{\delta\to 0} \frac{\det \big( h^\delta_{2\pi\beta}(x^\delta_i, v^\delta_j)\big)_{1\le i,j \le n}}{\prod^n_{j=1}\HarmUniv_{\Omega^\delta}(z^\delta,v_j^\delta)} 
= \frac{\det \big( h_{2\pi\beta}(x_i, v_j)\big)_{1\le i,j \le n}}{\prod^n_{j=1} \Pois_{S_r}\big((p^{-1}\circ \phiUniv_r)(z), \, (p^{-1}\circ \phiUniv_r)(v_j)\big) } .
\end{align}
If $n$ is odd, we have the following two cases as $r \to 0$.
\begin{enumerate}
\item if $\beta \in [0,\frac{1}{2})\cup(\frac{1}{2},1)$, 
and $N_\beta\in\Z$ is such that $|\beta-N_\beta| = \underset{k\in\Z}\min \, |\beta-k|$, then we have
\begin{align}\label{eqn::determinodd}
\det \big( h_{2\pi\beta}(x_i, v_j)\big)_{1\le i,j \le n} 
=\;& \frac{1}{\pi^n} \, \big( \ee^{\ii(\beta-N_\beta)c} + o_r(1)\big) \; 
r^{\frac{n^2-1}{4}+|\beta-N_\beta|} \; \Upsilon(\bs{x},\bs{v}) ;
\end{align}

\item if $\beta=\frac{1}{2}$, then we have
\begin{align}\label{eqn::determinodd2}
\det \big( h_{2\pi\beta}(x_i, v_j)\big)_{1\le i,j \le n} 
= \;& \frac{1}{\pi^n}  \, \big( 2\cos(c/2)+o_r(1) \big) \;
r^{\frac{n^2-1}{4}+\frac{1}{2}} \; \Upsilon(\bs{x},\bs{v}) .
\end{align}
\end{enumerate}
If $n$ is even, and $N_\beta\in\Z$ such that $|\beta-N_\beta-\frac{1}{2}| = \underset{k\in\Z}\min \,|\beta-k-\frac{1}{2}|$, then as $r \to 0$ we have
\begin{align}\label{eqn::determineven}
\det \big( h_{2\pi\beta}(x_i, v_j)\big)_{1\le i,j \le n}
=\;& \frac{1}{\pi^n}  \, \big(\ee^{\ii(\beta-N_\beta+\frac{1}{2})c}+o_r(1)\big) 
\; r^{\frac{n^2}{4}}  \; \Upsilon(\bs{x},\bs{v}) .
\end{align}
\end{lemma}

To derive Lemma~\ref{lem::deltato0odd}, we will to deal with the corresponding quantity on the universal cover.
For $z\in\Omega$ and $w \in \BdryOut$, denoting by $w_m \in \BdryOutUniv$ the $m$:th copy of $w=w_0$, define 
\begin{align}\label{eqn::beta_prob_continuous}
\hUniv_\beta(z, w) := \sum_{m\in\Z} \ee^{\ii m\beta} \, \Pois_{S_r}\big((p^{-1}\circ \phiUniv_r)(z), \, (p^{-1}\circ \phiUniv_r)(w_m)\big) .
\end{align}

\begin{lemma}\label{lem::Poisson_limit}
Suppose $z^\delta, x^\delta \in\LV^\delta$, and $w^\delta \in \BdryOutV^\delta$ converge to $z,x\in\Omega\setminus\zeta$ and $w = w_0\in\BdryOut\setminus\zeta$ as $\delta\to 0$. 
Denote by $w_m^\delta\in \BdryOutVUniv^\delta$ and $w_m\in \BdryOutUniv$ the $m$:th copies of $w^\delta=w_0^\delta$ and $w=w_0$. 
Then, we have
\begin{align}\label{eqn::Poisson_limit}
\frac{\HarmUniv_{\Omega^\delta}(x^\delta,w_m^\delta)}{\HarmUniv_{\Omega^\delta}(z^\delta,w_0^\delta)}
= \frac{\Pois_{S_r}\big((p^{-1}\circ \phiUniv_r)(x), \, (p^{-1}\circ \phiUniv_r)(w_m)\big)}{\Pois_{S_r}\big((p^{-1}\circ \phiUniv_r)(z), \, (p^{-1}\circ \phiUniv_r)(w_0)\big)} .
\end{align}
In particular, for the function defined in~\eqref{eq:h_beta}, we have
\begin{align}\label{eqn::conv_Poisson_limit}
\lim_{\delta\to 0} \frac{h^\delta_\beta(x^\delta, w^\delta)}{\HarmUniv_{\Omega^\delta}(z^\delta,w^\delta)}
= \frac{\hUniv_\beta(x, w)}{\Pois_{S_r}\big((p^{-1}\circ \phiUniv_r)(z), \, (p^{-1}\circ \phiUniv_r)(w)\big)} .
\end{align}
\end{lemma}

\begin{proof}
The proof comprises standard discrete harmonic arguments.  
Note that the right side of~\eqref{eqn::Poisson_limit} is independent of the normalisation of the Poisson kernel. 
Denote by $z_m^\delta\in \LVUniv^\delta$ and $z_m\in \OmegaUniv$ the $m$:th copies of $z^\delta$ and $z$. 
Observe that 
$\HarmUniv_{\Omega^\delta}(z_0^\delta,w_0^\delta) = \HarmUniv_{\Omega^\delta}(z^\delta_m,w^\delta_m)$
and the ratio $\HarmUniv_{\Omega^\delta}(\cdot,w^\delta_m)/\HarmUniv_{\Omega^\delta}(z^\delta_m,w^\delta_m)$
is discrete harmonic on $\LVUniv^\delta$, equalling $1$ at $z^\delta_m$ and $0$ on $\BdryOutVUniv^\delta\setminus \{w^\delta_m\}$. 
By a discrete Harnack-type inequality, for every $\eps>0$, there are constants $A>0$ and $\alpha>0$ such that 
\begin{align*}
\frac{\HarmUniv_{\Omega^\delta}(u,w^\delta_m)}{\HarmUniv_{\Omega^\delta}(z^\delta_m,w^\delta_m)}\le A
\qquad \textnormal{and} \qquad
\bigg| \frac{\HarmUniv_{\Omega^\delta}(u,w^\delta_m)}{\HarmUniv_{\Omega^\delta}(z^\delta_m,w^\delta_m)}
- \frac{\HarmUniv_{\Omega^\delta}(v,w^\delta_m)}{\HarmUniv_{\Omega^\delta}(z^\delta_m,w^\delta_m)} \bigg| \le A \, |u-v|^a , 
\qquad u,v\in\LVUniv^\delta\setminus B(w^\delta_m,\eps) .
\end{align*}
By the Arzela--Ascoli theorem, this (and the pointwise convergence on $\BdryOutVUniv^\delta\setminus \{w^\delta_m\}$) imply the pre-compactness of 
$\big\{\HarmUniv_{\Omega^\delta}(\cdot,w^\delta_m)/\HarmUniv_{\Omega^\delta}(z^\delta_m,w^\delta_m)\big\}_{\delta>0}$ with respect to the uniform topology. Suppose $h$ is any subsequential limit. 
By the uniform convergence, we have $h=0$ on $\BdryOutUniv\setminus \{w_m\}$ and $h(z_m)=1$ at 
$z_m = \lim_{\delta \to 0} z^\delta_m$.
By constructing the harmonic conjugate locally (e.g., see the proof of Proposition 4.2 in~\cite{LSW:Conformal_invariance_of_planar_LERW_and_UST}), we see that $\partial_n (h\circ(\phiUniv_r)^{-1})|_{\R}=0$.  
Thus, $h(x)$ equals the right-hand side of~\eqref{eqn::Poisson_limit}.

For~\eqref{eqn::conv_Poisson_limit}, similarly to~\eqref{eqn::wind_control_aux1}, there exists a constant $q\in (0,1)$ and a constant $C>0$ such that  
\begin{align*}
\frac{\HarmUniv_{\Omega^\delta}(z^\delta, w^\delta_m)}{\HarmUniv_{\Omega^\delta}(z^\delta,w^\delta)}\le C q^{|m|}.
\end{align*}
Combining this with~\eqref{eqn::Poisson_limit}, we obtain~\eqref{eqn::conv_Poisson_limit}.
\end{proof}

\begin{proof}[Proof of Lemma~\ref{lem::deltato0odd}]
The asserted Equation~\eqref{eqn::scaling_det_odd} was proved in Equation~\eqref{eqn::conv_Poisson_limit} in Lemma~\ref{lem::Poisson_limit}.
The asymptotics~\eqref{eqn::determinodd} and~\eqref{eqn::determineven} are given 
by~\cite[Proposition~5.5]{Arista-OConnell:Loop-erased_walks_and_random_matrices}, which we recall briefly below.
By Equation~\eqref{eqn::beta_prob_continuous} and~\cite[Lemma~A.2]{Arista-OConnell:Loop-erased_walks_and_random_matrices}, we have
\begin{align}\label{eqn::odd_det_aux}
\;&\det \big( h_{2\pi\beta}(x_i, v_j)\big)_{1\le i,j \le n} 
\\
\notag
=\;&\frac{1}{(2\pi)^n}\sum_{\substack{k_1<\cdots<k_n\\ k_i\in\Z}} 
\det \big( \ee^{\ii (\beta-k_i) \arg(\phi_r(x_j))} \big)_{1\le i,j\le n} 
\; \det\big( \ee^{-\ii (\beta-k_i) \arg(\phi_r(v_j))} \big)_{1\le i,j\le n} 
\; \prod_{i=1}^{n}\textnormal{sech}\big(|\log r|(\beta-k_i)\big).
\end{align}
We first consider the case of odd $n$.
\begin{itemize}[leftmargin=*]
\item If $\beta\in [0,\frac{1}{2})\cup(\frac{1}{2},1)$, the minimal power of $r$ is obtained at
$(k_1,\ldots,k_n)=(-\frac{n-1}{2}+N_\beta,\ldots,\frac{n-1}{2}+N_\beta)$.
\cite[Proposition~5.5]{Arista-OConnell:Loop-erased_walks_and_random_matrices} thus yields
\begin{align*}
\;& \det \big( \ee^{\ii (\beta-k_i) \arg(\phi_r(x_j))} \big)_{1\le i,j\le n} 
\; \det \big(\ee^{-\ii (\beta-k_i) \arg(\phi_r(v_j))} \big)_{1\le i,j\le n} 
\; \prod_{i=1}^{n}\textnormal{sech}\big(|\log r|(\beta-k_i)\big) 
\\
=\;&\big(2^n \ee^{\ii(\beta-N_\beta)c}+o_r(1)\big) 
\; r^{\frac{n^2-1}{4}+|\beta-N_\beta|} \; \Upsilon(\bs{x},\bs{v})  , \qquad r \to 0 .
\end{align*}
By Lemma~\ref{lem::det_cont} given below, there exists a constant $C>0$ such that
\begin{align*}
&\big|\det \big(\ee^{\ii (\beta-k_i) \arg(\phi_r(x_j))} \big)_{1\le i,j\le n}\big| \; 
\big| \det \big(\ee^{-\ii (\beta-k_i) \arg(\phi_r(v_j))} \big)_{1\le i,j\le n}\big|
\;\prod_{i=1}^{n}\textnormal{sech}\big(|\log r|(\beta-k_i)\big)\\
\le \;& C \; r^{\sum_{i=1}^{n}|\beta-k_i|} \;\prod_{i=1}^{n}(|k_i|+1)^n \; \Upsilon(\bs{x},\bs{v})  , \qquad r \to 0.
\end{align*}
Plugging this into~\eqref{eqn::odd_det_aux}, we obtain~\eqref{eqn::determinodd}.

\item If $\beta=\frac{1}{2}$, the minimal power of $r$ is obtained at
$(k_1,\ldots,k_n)=(-\frac{n-1}{2}+\beta-\frac{1}{2},\ldots,\frac{n-1}{2}+\beta-\frac{1}{2})$, 
or $(k_1,\ldots,k_n)=(-\frac{n-1}{2}+\beta+\frac{1}{2},\ldots,\frac{n-1}{2}+\beta+\frac{1}{2})$. 
Equivalently, $\beta-\frac{1}{2}$ and $\beta+\frac{1}{2}$ are both integers closest to $\beta$. 
By a similar computation as above in the proof of~\eqref{eqn::determinodd}, we obtain~\eqref{eqn::determinodd2}.
\end{itemize}
Next, we consider the case of even $n$. 
In this case, the minimal power of $r$ is obtained at $(k_1,\ldots,k_n)=(-\frac{n}{2}+N_\beta+1,\ldots,\frac{n}{2}+N_\beta)$, and~\cite[Proposition~5.5]{Arista-OConnell:Loop-erased_walks_and_random_matrices} yields
\begin{align*}
\;&\det \big(\ee^{\ii (\beta-k_i) \arg(\phi(x_j))}\big)_{1\le i,j\le n}
\; \det\big(\ee^{-\ii (\beta-k_i) \arg(\phi(v_j))}\big)_{1\le i,j\le n}
\; \prod_{i=1}^{n}\textnormal{sech}\big(|\log r|(\beta-k_i)\big) 
\\
=\;&\big(2^n \ee^{\ii(\beta-N_\beta-\frac{1}{2})c}+o_r(1)\big)
\; r^{\frac{n^2}{4}} \; \Upsilon(\bs{x},\bs{v}) , \qquad r \to 0 .
\end{align*}
By a similar argument as in the proof of~\eqref{eqn::determinodd}, we obtain~\eqref{eqn::determineven}. This completes the proof. 
\end{proof}

\begin{proof}[Proof of Theorem~\ref{thm::total_winding}]
This is a collection of Proposition~\ref{prop::total_winding}  and Lemmas~\ref{lem:BLM_conv} and~\ref{lem::deltato0odd}.
\end{proof}

\begin{lemma}\label{lem::det_cont}
Suppose $f_1,\ldots f_n$ are analytic functions on $\partial\D$, and let $f_i^{(j)}$ denote the $j$-th derivative of $f_i$.
Then, there exists a constant $C>0$, which only depends on $n$, such that
\begin{align}\label{eqn::genbound}
\big|\det \big(f_i(z_j)\big)_{1\le i,j\le n}\big| \le C \, \Big(\prod_{1\le i<j\le n} |z_j-z_i| \Big) \bigg( \prod^n_{i=1} \max_{1\le j\le n}\max_{z\in\partial\D}|f^{(j-1)}_i(z)| \bigg) .
\end{align}
\end{lemma}

\begin{proof}
Writing 
\begin{align*}
\det \big(f_i(z_j)\big)_{1\le i,j\le n} = \; & \Big(\prod_{1\le i<j\le n}(z_j-z_i) \Big) \; \det \big(F_{i,j}(z_j)\big)_{1\le i,j\le n} ,
\\
F_{i,1}(w) := f_i(w) ,
\qquad 
F_{i,j}(w) := \; & \frac{F_{i,j-1}(w)-F_{i,j-1}(z_{j-1})}{w-z_{j-1}}, \qquad i \in \{1,\ldots, n\} , \; j  \in \{2,\ldots, n\} ,
\end{align*}
we inductively find that
\begin{align*}
F_{i,j}(z_j) =\int_0^1 \lambda_1^{j-1} \ud\lambda_1\int_0^1 \lambda_2^{j-2}\ud\lambda_2 \cdots\int^1_{0}\ud\lambda_{j-1}
\; f_i^{(j-1)}\bigg(\sum^{j-1}_{k=1}\lambda_1\cdots\lambda_{k-1}(1-\lambda_k)z_k \, + \, \lambda_1\cdots\lambda_{j-1}z_j\bigg) ,
\end{align*}
and expanding the determinant yields~\eqref{eqn::genbound}. 
\end{proof}

\section{Scaling limit of LERW rays: proof of Theorem~\ref{thm::conv_curves}}
\label{sec::curves}

In this section, we will use the same setup and notation as in Sections~\ref{sec::main}--\ref{sec::annulus}.
Our aim is to establish the convergence of the hitting points of the UST branches at the outer boundary (Proposition~\ref{prop::hittingpoints} in Section~\ref{sec::det}), and the convergence of UST branches themselves (Proposition~\ref{prop::scaling_curve} in Section~\ref{sec::scaling}).
We finish with the proof of Theorem~\ref{thm::conv_curves} in the end of Section~\ref{sec::scaling}.
Knowing the scaling limit of the curves, we compute their winding variance in Section~\ref{subsec:truncated_winding} (yielding Proposition~\ref{coro::truncated_winding}).

\subsection{Scaling limit of the law of hitting points}\label{sec::det}

Next, we derive the scaling limit of the conditional law of the hitting points of the UST branches, 
$\bs\gamma^\delta(\bs\exitT^\delta) = (\gamma_1^\delta(\exitT_1^\delta),\ldots,\gamma_n^\delta(\exitT_n^\delta))$, to the outer boundary $\BdryOut^\delta$ 
on the event $E^\delta_{\bs{x}}$, by first letting $\delta\to 0$ and then shrinking $\diam(\BdryIn) \to 0$. 
Let $\alpha_1,\ldots,\alpha_n$ be $n$ disjoint arcs located on $\BdryOut$ in counterclockwise order, and suppose that 
$\alpha^\delta_1,\ldots,\alpha_n^\delta$ are $n$ disjoint arcs on $\BdryOut^\delta$ such that 
$\dist_X(\alpha_j^\delta,\alpha_j)\to 0$ as $\delta\to 0$. 
Consider the event
\begin{align*}
E^\delta_{\bs{x}, \bs{\alpha}, k} 
:=\{\vec\gamma^\delta_j\textnormal{ hits }\BdryOut^\delta\textnormal{ at }\alpha_{j+k}^\delta,\textnormal{ for each }1\le j\le n\},
\qquad 1\le k\le n ,
\end{align*}
where we use the convention that $\alpha^\delta_{j+k}=\alpha^\delta_{j+k-n}$ if $j+k>n$. 
Define also the event
\begin{align}\label{eqn::defE3}
E^\delta_{\bs{x}, \bs{\alpha}} 
= E(\Omega^\delta;x_1^\delta,\ldots,x_n^\delta;\alpha_1^\delta,\ldots,\alpha_n^\delta) 
:= \bigcup_{k=0}^{n-1}E^\delta_{\bs{x}, \bs{\alpha}, k}.
\end{align}
Recall that we denote by $\OmegaFill$ the simply connected domain bounded by $\BdryOut$, and we let
$\phi \colon \OmegaFill \to \D$ be the conformal bijection sending it onto the disc $\D = \{z \in \C \,|\, |z|<1\}$
such that $\phi(0)=0$ and $\phi'(0)>0$. 
Let $\phi_r \colon \Omega \to \A_r$ be the conformal isomorphism onto the annulus 
$\A_r = A(r,1)$, so that $\phi_r$ coincides with the restriction to $\Omega$ of the conformal isomorphism $\phi \colon \Omega \to \D$. 

\medskip
In the odd case, Proposition~\ref{prop::hittingpoints} gives the asymptotics as $r \to 0$ claimed in Theorem~\ref{thm::partfun}.
Observe, however, that in the even case, the Brownian loop soup term depends on $r$, because $\Omega$ does.  
Therefore, we cannot directly infer the asymptotics as $r \to 0$ from Proposition~\ref{prop::hittingpoints} in the even case. 

\begin{proposition}\label{prop::hittingpoints}
For every $n\ge 1$, the following scaling limit holds. 
When $n$ is odd, as $r \to 0$, we have
\begin{align}\label{eqn::hitting_odd_aux01}
\lim_{\delta\to 0} \PP^\delta[E^\delta_{\bs{x}, \bs{\alpha}}]
= \; & \frac{r^{\frac{n^2-1}{4}}}{\pi^{n}} \, \prod_{1\le i<j\le n}\sin\bigg(\frac{\arg\phi(x_j)-\arg\phi(x_i)}{2}\bigg)
\\
\nonumber 
\; & \times 
\int_{\UC^n}\one{\{\ee^{\ii \bs{\theta}}\in\LX_n\}} 
\one{\{\ee^{\ii\theta_j}\in\phi(\alpha_j),\, \forall~j\}} 
\prod_{1\le j<k\le n} \sin\Big(\frac{\theta_j - \theta_k}{2}\Big) \prod_{j=1}^{n}\ud\theta_j 
\; + \; o\big(r^{\frac{n^2-1}{4}}\big) . 
\end{align}
When $n$ is even, as $r \to 0$, we have
\begin{align}\label{eqn::hitting_even_aux01}
\lim_{\delta\to 0} \PP^\delta[E^\delta_{\bs{x}, \bs{\alpha}}]
= \; & \exp \Big(\! -2\mu_\Omega^{\textnormal{ND}}[\{\ell\in\LL_* \cond \tfrac{\ph(\ell)}{2\pi} \textnormal{ is odd }\}] \Big) 
\Bigg( \frac{r^{\frac{n^2}{4}}}{\pi^{n}} \, \prod_{1\le i<j\le n}\sin\bigg(\frac{\arg\phi(x_j)-\arg\phi(x_i)}{2}\bigg)
\\
\nonumber 
\; & \quad \times 
\int_{\UC^n}\one{\{\ee^{\ii \bs{\theta}}\in\LX_n\}} 
\one{\{\ee^{\ii\theta_j}\in\phi(\alpha_j),\, \forall~j\}} 
\prod_{1\le j<k\le n} \sin\Big(\frac{\theta_j - \theta_k}{2}\Big) \prod_{j=1}^{n}\ud\theta_j 
\; + \; o\big(r^{\frac{n^2-1}{4}}\big) 
\Bigg)  . 
\end{align}
Moreover, for every $n\ge 1$, the following scaling limit holds: 
\begin{align}
\nonumber
\lim_{r \to 0}\lim_{\delta\to 0} 
\frac{\PP^\delta[E^\delta_{\bs{x}, \bs{\alpha}}]}{\PP^\delta[E^\delta_{\bs{x}}]} 
= \; & \frac{2^{\frac{n(n-1)}{2}}}{\LZ_n}
\int_{\UC^n}\one{\{\ee^{\ii \bs{\theta}}\in\LX_n\}} 
\one{\{\ee^{\ii\theta_j}\in\phi(\alpha_j),\, \forall~j\}} 
\prod_{1\le j<k\le n} \sin\Big(\frac{\theta_j - \theta_k}{2}\Big) \prod_{j=1}^{n}\ud\theta_j \\
\label{eqn::scalingprob}
= \; & \frac{1}{\LZ_n}
\int_{\UC^n}\one{\{\ee^{\ii \bs{\theta}}\in\LX_n\}} 
\one{\{\ee^{\ii\theta_j}\in\phi(\alpha_j),\, \forall~j\}} 
\prod_{1\le j<k\le n}\big|\ee^{\ii\theta_j}-\ee^{\ii\theta_k}\big|\prod_{j=1}^{n}\ud\theta_j,
\end{align}
where
\begin{align*}
\LZ_n := \; & \int_{\UC^n}\one{\{\ee^{\ii \bs{\theta}}\in\LX_n\}} 
\prod_{1\le j<k\le n}\big|\ee^{\ii\theta_j}-\ee^{\ii\theta_k}\big|\prod_{j=1}^{n}\ud\theta_j \\
= \; & 2^{\frac{n(n-1)}{2}} \int_{\UC^n}\one{\{\ee^{\ii \bs{\theta}}\in\LX_n\}} 
\prod_{1\le j<k\le n} \sin\Big(\frac{\theta_j - \theta_k}{2}\Big) \prod_{j=1}^{n}\ud\theta_j .
\end{align*}
In particular, the conditional law of $\bs\gamma^\delta(\bs\exitT^\delta)$ given $E^\delta_{\bs{x}}$ converges weakly when first letting $\delta\to 0$ and then shrinking $\diam(\BdryIn)\to 0$. 
The density of the limit points $\bs{y}=(y_1,\ldots,y_n)$ on $\BdryOut^n$ is obtained from~\eqref{eqn::density} by conformal invariance;
the density of $\phi(\bs{y}) = (\phi(y_1),\ldots,\phi(y_n))$ on $\UC^n$ equals $\rho(\phi(\bs{y}))$ in~\eqref{eqn::density}.
\end{proposition}

\begin{proof}
\textbf{$n$ odd.} We first consider the case where $n$ is odd. 
For a boundary arc $\alpha \subset \partial B(0,1)$, denote by $h_{\smash{\A_{r}}}(\cdot,\alpha)$ 
the bounded harmonic function on $\smash{\A_{r}}$ with mixed Dirichlet-Neumann boundary conditions:
\begin{align*}
h_{\smash{\A_{r}}}(\cdot,\alpha)|_{\partial B(0,1)\setminus\alpha}=0,
\qquad h_{\smash{\A_{r}}}(\cdot,\alpha)|_{\alpha}=1 ,
\qquad\textnormal{and}\qquad \partial_n h_{\smash{\A_{r}}}(\cdot,\alpha)|_{\partial B(0,r)}=0.
\end{align*}
After taking the sum over the edges along $\alpha^\delta_j$ for $j = 1, \ldots, n$, combining~\eqref{eqn::odd} from Lemma~\ref{lem::nodd0} 
and~\eqref{eqn::determinodd} from Lemma~\ref{lem::deltato0odd}, 
the asserted Equation~\eqref{eqn::hitting_odd_aux01} follows from the exact determinantal identity
\begin{align}\label{eqn::hitting_odd_aux1}
\lim_{\delta\to 0} \PP^\delta[E^\delta_{\bs{x}, \bs{\alpha}}]
= \; & \det \big( h_{\A_r}(\phi(x_i),\phi(\alpha_j)) \big)_{1\le i,j\le n} .
\end{align}
For the denominator $\PP^\delta[E^\delta_{\bs{x}}]$, take a large $m\in\N$ and set 
$\xi_{j,m}:=\phi^{-1}\big(\ee^{\ii\frac{2\pi j}{m}},\ee^{\ii\frac{2\pi(j+1)}{m}}\big)$,
and let $\xi_{j,m}^\delta$ be its discrete approximation on $\BdryOut^\delta$, for $0\le j\le m-1$. 
Write $\bs{\xi} = (\xi_{1,m}^\delta, \ldots , \xi_{m,m}^\delta)$. 
By~\eqref{eqn::hitting_odd_aux1}, we find 
\begin{align}\label{eqn::hitting_odd_aux2}
\lim_{\delta\to 0} \sum_{0\le m_1\cdots<m_n\le m-1} \PP^\delta[E^\delta_{\bs{x}, \bs{\xi}}]
= \; & \sum_{0\le m_1\cdots<m_n\le m-1} \det\big(h_{\A_r}(\phi(x_i),\phi(\xi_{m_j,m}))\big)_{1\le i,j\le n}.
\end{align}
Consider now the event 
\begin{align*}
A^\delta_{\bs{\xi}}
:= \big\{\textnormal{there exist at least two hitting points of }\{\gamma_1^\delta,\ldots,\gamma^\delta_n\} \textnormal{ in }\xi^\delta_{j,m}\textnormal{ for some }0\le j\le m-1 \big\}. 
\end{align*}
Then, up to a constant independent of $n$ and $m$, we have
\begin{align}\label{eqn::hitting_odd_aux3}
\limsup_{\delta\to 0} \PP^\delta[A^\delta_{\bs{\xi}}] 
\leq \; & n^2 \, \sum_{j=0}^{m-1} \max_{1\le i\le n} \Big( h_{\A_r}(\phi(x_i),\phi(\xi_{j,m})) \Big)^2
\lesssim \frac{n^2}{m} .
\end{align}
Combining~(\ref{eqn::hitting_odd_aux01},~\ref{eqn::hitting_odd_aux2},~\ref{eqn::hitting_odd_aux3}) together, we obtain
\begin{align*}
\lim_{\delta\to 0} \PP^\delta[E^\delta_{\bs{x}}] 
= \; & \lim_{m\to\infty}\sum_{0\le m_1\cdots<m_n\le m-1} \det\big(h_{\A_r}(\phi(x_i),\phi(\xi_{m_j,m}))\big)_{1\le i,j\le n}
\\
= \; & \frac{r^{\frac{n^2-1}{4}}}{\pi^{n}}  \, \prod_{1\le i<j\le n}\sin\bigg(\frac{\arg\phi(x_j)-\arg\phi(x_i)}{2}\bigg)
\\
 \; & \times 
\int_{\UC^n}\one{\{\ee^{\ii \bs{\theta}}\in\LX_n\}} 
\one{\{\ee^{\ii\theta_j}\in\phi(\alpha_j),\, \forall~j\}} 
\prod_{1\le j<k\le n} \sin\Big(\frac{\theta_j - \theta_k}{2}\Big) \prod_{j=1}^{n}\ud\theta_j 
\; + \; o\big(r^{\frac{n^2-1}{4}}\big) . 
\end{align*}
For odd $n$, the asserted Equation~\eqref{eqn::scalingprob} now follows combining with the limit~\eqref{eqn::hitting_odd_aux1}.

Now suppose that \textbf{$n$ is even.} The proof is similar. 
Taking the sum over the edges along $\alpha^\delta_j$ for $j = 1, \ldots, n$, combining~\eqref{eqn::even} from Lemma~\ref{lem::neven}
and~\eqref{eqn::determineven} from Lemma~\ref{lem::deltato0odd} with Lemma~\ref{lem:BLM_conv}, 
the asserted Equation~\eqref{eqn::hitting_even_aux01} follows from 
\begin{align*}
\lim_{\delta\to 0} \PP^\delta[E^\delta_{\bs{x}, \bs{\alpha}}]
= \; & \exp \Big(\! -2\mu_\Omega^{\textnormal{ND}}[\{\ell\in\LL_* \cond \tfrac{\ph(\ell)}{2\pi} \textnormal{ is odd }\}] \Big) 
\; \det \big( h_{\A_r}(\phi(x_i),\phi(\alpha_j)) \big)_{1\le i,j\le n} .
\end{align*}
The argument to obtain the asserted Equation~\eqref{eqn::scalingprob} is the same as in the odd case.
\end{proof}

We now establish a relation between Brownian motion with mixed boundary conditions and Brownian motion with Dirichlet boundary conditions. 
Recall that we denote by $\OmegaFill$ the domain bounded by $\BdryOut$.

\begin{proposition}\label{prop::BM_mix_Diri}
Suppose $\simpleCurv$ is a simple curve on $\OmegaFillCl$ which only intersects $\BdryOut$ at its starting point and does not hit $0$. Then, for every $\beta\in\R$, we have
\begin{align}\label{eqn::conv_loop_measure_odd}
\lim_{\diam (\BdryIn)\to 0} \mu^{\textnormal{ND}}_\Omega \big[ \one{\{\ell\in\LL_* \textnormal{ and } \ell\cap\simpleCurv\neq\emptyset\}} \; \ee^{\ii\beta\ph(\ell)} \big] 
= \mu_{\OmegaFill} \big[ \one{\{\ell\in\LL_* \textnormal{ and } \ell\cap\simpleCurv\neq\emptyset\}} \; \ee^{\ii\beta\ph(\ell)} \big].
\end{align}
In particular, we have 
\begin{align}\label{eqn::scaling_constant_poly}
\lim_{\diam (\BdryIn)\to 0} \exp \Big( 2\mu_{\Omega}^{\textnormal{ND}} \big[\ell \in \LL_* \textnormal{ such that } \tfrac{1}{2\pi}\ph(\ell)\textnormal{ is odd and }\ell\cap\simpleCurv\neq\emptyset \big] \Big) 
= \bigg|\frac{\phi_\simpleCurv'(0)}{\phi'(0)}\bigg|^{1/4} ,
\end{align}
where $\phi \colon \OmegaFill \to \D$ and $\phi_\simpleCurv \colon \OmegaFill\setminus\simpleCurv \to \D$ are any conformal isomorphisms fixing the origin. 
\end{proposition}

\begin{proof}
We take $\Omega^\delta=\Omega\cap\delta\Z^2$ and take $\simpleCurv^\delta$ to be an approximation of $\simpleCurv$ such that $\dist(\simpleCurv^\delta,\simpleCurv)\le 2\delta$. 
Rescaling the loops on $\Z^2$ as described Section~\ref{subsec:Fomin} 
and denoting by $\Lambda^{\textnormal{D}}_{\Omega^\delta}$ the restriction of the random walk loop measure on loops contained in $\Omega^\delta\setminus\partial\Omega^\delta$, from the definitions we obtain
\begin{align*}
\Lambda^{\textnormal{ND}}_{\Omega^\delta}[ \, \cdot \, \one{\{\ell\textnormal{ is contained in }\Omega^\delta\setminus\partial\Omega^\delta\}}]=\Lambda^{\textnormal{D}}_{\Omega^\delta}[ \, \cdot \,] . 
\end{align*}
In particular, we have
\begin{align*}
\;&\Lambda^{\textnormal{ND}}_{\Omega^\delta}\big[\one{\{\ell\in\LL_*^\delta \textnormal{ and } \ell\cap\simpleCurv^\delta\neq\emptyset\}}\,\ee^{\ii\beta\ph(\ell)}\big]\\
=\;&\Lambda^{\textnormal{D}}_{\Omega^\delta}\big[\one{\{\ell\in\LL_*^\delta \textnormal{ and } \ell\cap\simpleCurv^\delta\neq\emptyset\}}\,\ee^{\ii\beta\ph(\ell)}\big]
\; + \;\Lambda^{\textnormal{ND}}_{\Omega^\delta}\big[\one{\{\ell\in\LL_*^\delta \textnormal{ and } \ell\cap\simpleCurv^\delta\neq\emptyset, \textnormal{ and }\ell\cap\BdryIn^\delta\neq\emptyset\}}\,\ee^{\ii\beta\ph(\ell)}\big] .
\end{align*}
For the first term, by~\cite[Theorem~1.1]{Lawler-Ferreras:RW_loop_soup} we have
\begin{align}\label{eqn::BM_loop_control}
\lim_{\delta\to 0}\Lambda^{\textnormal{D}}_{\Omega^\delta} \big[\one{\{\ell\in\LL_*^\delta \textnormal{ and } \ell\cap\simpleCurv^\delta\neq\emptyset\}}\,\ee^{\ii\beta\ph(\ell)} \big]
= \mu_{\Omega} \big[\one{\{\ell\in\LL_* \textnormal{ and } \ell\cap\simpleCurv\neq\emptyset\}}\,\ee^{\ii\beta\ph(\ell)} \big].
\end{align}
Consider then the second term. 
Let $d=\frac{1}{4}\dist(\simpleCurv,0)$ and choose $\eps_0 \in (0,\frac{d}{100})$ and $\delta_0' \in (0,\frac{1}{4}d)$ such that $\BdryIn\cup\BdryIn^\delta\subset B(0,\eps_0)$ for all $\delta \in (0,\delta_0')$.  
Note that any loop on $\Omega^\delta$ which intersects both $\BdryIn^\delta$ and $\simpleCurv^\delta$ cannot be contained in $\BdryInNbhdDiscr{d}$, and we must have 
$\diam(\ell) \ge d/2$.  
Fix $\varepsilon>0$. By~\eqref{eqn::truncated_time_bound} and Lemma~\ref{lem::Green_mix},
we can choose $M=M(\varepsilon, d)>1$ large enough and $\delta_0=\delta_0(\varepsilon,d)>0$ small enough such that 
\begin{align*}
\Lambda_{\Omega^\delta}^{\textnormal{ND}}\big[\ell\not\subset\BdryInNbhdDiscr{d}, \, \diam (\ell) \ge d/2, 
\textnormal{ and its lifetime }t_\ell\notin\big[\tfrac{1}{M},M\big] \big] 
\le \; & \varepsilon , \qquad \textnormal{ for all } \delta \in (0, \delta_0) . 
\end{align*}
By~\eqref{eqn::root_time_bound_control_aux} and Lemma~\ref{lem::Green_mix}, 
we can furthermore choose $c=c(\varepsilon,d)>0$ small enough such that 
\begin{align*}
\Lambda_{\Omega^\delta}^{\textnormal{ND}}\big[\ell\not\subset\BdryInNbhdDiscr{d}, \, \diam (\ell) \ge d/2, \textnormal{ its lifetime }t_\ell\in\big(\tfrac{1}{M},M\big) ,
\textnormal{ and its root lies in } \BdryInNbhdDiscr{c} \big] 
\le \; & \varepsilon , 
\end{align*}
for all $\delta \in (0, \delta_0)$.
Now, let $\exitTIn$ be the hitting time of the random walk $\LR$ at $\BdryIn^\delta$. As before, we have
\begin{align}
\nonumber
\;&\Lambda_{\Omega^\delta}^{\textnormal{ND}} \big[\ell\in\LL_*^\delta, \textnormal{ its lifetime } t_\ell\in \big(\tfrac{1}{M},M\big), 
\textnormal{ its root lies not in }\BdryInNbhdDiscr{c}, \textnormal{ and } \ell\cap\BdryIn^\delta\neq\emptyset\big] \\
\nonumber
= \; & \sum_{\substack{z\in\LV^\delta\setminus\BdryOutV^\delta \\ z\not\in\BdryInNbhdDiscr{c}}}
\frac{\delta^2}{2}\sum_{t\in \frac{1}{2}\delta^2\N \cap[\frac{1}{M},M]}\frac{1}{t} \, \Prob_{z}\big[ \LR[0,t]\cap\BdryIn^\delta\neq\emptyset \textnormal{ and }  \LR(t) = z \big] \\
\nonumber
\leq \; & \sum_{\substack{z\in\LV^\delta\setminus\BdryOutV^\delta \\ z\not\in\BdryInNbhdDiscr{c}}}
\frac{\delta^2}{2} \, M \, 
\ProbE_{z}\bigg[\one{\{\exitTIn<\infty\}} \sum_{j = 1}^\infty\Prob_{\LR(\exitTIn)}[\LR(j)=z] \bigg] \\
\leq \; &  \frac{M}{2}\Leb(\tilde \Omega^\delta)\; \Big(\max_{z\not\in\BdryInNbhdDiscr{c}}\Prob_{z}[\exitTIn<\infty] \Big) \; \underbrace{\Big(\max_{ \substack{ z\notin\BdryInNbhdDiscr{c} \\ |w-z|\ge c}}\Gren^{\textnormal{ND}}_{\Omega^\delta}(w,z)\Big)}_{\substack{ \leq \; C(c) \; \textnormal{by Lemma~\ref{lem::Green_mix}} \\ \textnormal{whenever $\delta$ is small enough}}}
\label{eqn::hitting_aux1}
\end{align}
Taking $\delta \to 0$, we obtain $\Leb(\tilde \Omega^\delta) \to \Leb(\Omega)$ and 
\begin{align} \label{eqn::hitting_aux}
\limsup_{\delta \to 0} \max_{z\not\in\BdryInNbhdDiscr{c}}\Prob_{z}[\exitTIn<\infty]
\leq \frac{\log c}{\log \diam (\BdryIn)}.
\end{align}
Combining~(\ref{eqn::BM_loop_control},~\ref{eqn::hitting_aux1},~\ref{eqn::hitting_aux}) together, 
by first letting $\delta\to 0$, then letting $\diam (\BdryIn)\to 0$ and finally letting $\varepsilon\to 0$, we complete the proof of~\eqref{eqn::conv_loop_measure_odd}.

For~\eqref{eqn::scaling_constant_poly}, note that by~\eqref{eqn::conv_loop_measure_odd}, we have
\begin{align} \label{eqn::conv_loop_odd_aux}
\begin{split}
\lim_{\diam(\BdryIn)\to 0} \; & \mu_{\Omega}^{\textnormal{ND}} \big[ \ell \in \LL_* \textnormal{ such that }  \tfrac{1}{2\pi}\ph(\ell)\textnormal{ is odd and }\ell\cap\simpleCurv\neq\emptyset \big]
\\
= \; & \mu_{\OmegaFill} \big[ \ell \in \LL_* \textnormal{ such that }  \tfrac{1}{2\pi}\ph(\ell)\textnormal{ is odd and }\ell\cap\simpleCurv\neq\emptyset \big].
\end{split}
\end{align}
Combining this with an argument similar to~\cite[Proof of Proposition~4.3]{BCL:Spin_systems_from_loop_soups}, we are able to complete the proof.
Indeed, using the restriction covariance property~\cite[Proposition~3]{Werner:The_conformally_invariant_measure_on_self-avoiding_loops} 
combined with~\cite[Theorems~5.1~\&~5.2]{Garban-Trujillo-Ferreras:The_expected_area_of_the_filled_planar_Brownian_loop_is_Pi_over_5} 
(which rely on a direct computation and a result of Yor~\cite{Yor:Loi_de_lindice_du_lacet_brownien_et_distribution_de_Hartman-Watson} for the law of the index of a Brownian loop around a fixed point), 
we can write
\begin{align*}
\textnormal{\eqref{eqn::conv_loop_odd_aux}}
= \; & \sum_{k=-\infty}^{\infty} \mu_{\OmegaFill} \big[ \ell \in \LL_* \textnormal{ such that }  \tfrac{1}{2\pi}\ph(\ell) = 2k+1 \textnormal{ and }\ell\cap\simpleCurv\neq\emptyset \big]
\\
= \; & \sum_{k=-\infty}^{\infty} \frac{1}{2\pi^2 (2k+1)^2} \, \log \bigg|\frac{\phi_\simpleCurv'(0)}{\phi'(0)}\bigg|
\; = \; \frac18 \, \log \bigg|\frac{\phi_\simpleCurv'(0)}{\phi'(0)}\bigg| ,
\end{align*}
which yields~\eqref{eqn::scaling_constant_poly}. 
\end{proof}

In Lemma~\ref{lem::scaling_constant}, we will provide a self-contained proof of~\eqref{eqn::scaling_constant_poly} via a martingale technique.

\subsection{Scaling limit of LERW rays: proof of Theorem~\ref{thm::conv_curves}}\label{sec::scaling}

We now consider the collection $\bs{\eta} = (\eta_1,\ldots,\eta_n)$ of $n$ random curves in the unit disc $\D$ as in Section~\ref{subsec:curves_intro}. 
To describe the scaling limit of the UST branches $\bs\gamma^\delta = (\gamma_1^\delta,\ldots,\gamma_n^\delta)$, 
it is convenient to consider their time-reversals $\cev{\bs\gamma}^\delta = (\cev{\gamma}_1^\delta,\ldots,\cev{\gamma}_n^\delta)$, 
reparametrised by common capacity time to obtain curves $\bs{\eta}^\delta = (\eta_1^\delta, \ldots, \eta_n^\delta)$ such that
$-\log \CR (0; \D\setminus \bs{\eta}^\delta[0, t] ) = nt$, 
where the conformal radius $\CR (0; \D\setminus \bs{\eta}^\delta[0, t] ) = 1/g'(0)$ is defined by the distortion $g'(0)$
of the conformal isomorphism $g \colon \D\setminus \bs{\eta}^\delta[0, t]  \to \D$ such that $g(0)=0$ and $g'(0)>0$.

To begin, we will prove the convergence of $\cev{\bs\gamma}^\delta$, by first letting $\delta\to 0$ and then $\diam(\BdryIn)\to 0$ (Proposition~\ref{prop::scaling_curve}). 
Then, we prove that after reparametrisation, the limiting curve has the same law as $\bs{\eta}$ (Lemma~\ref{lemma::together}), 
with starting point density~\eqref{eqn::density} and Loewner driving functions given by Dyson Brownian motion~\eqref{eqn:DBM}
--- that is, a $n$-sided radial $\SLE_2$ process.  
Let us first briefly recall the definition of $n$-sided radial $\SLE_\kappa$ (see~\cite{Healey-Lawler:N_sided_radial_SLE} for more details)
and $\SLE_\kappa$ with force points (see~\cite{Schramm-Wilson:SLE_coordinate_changes} for more details).

\bigskip

Consider $n$ simple curves $\bs{\slecurv} = (\slecurv_1,\ldots,\slecurv_n)$ in $\D$ such that each $\slecurv_j$ intersects $\partial\D$ only at its starting point, and they are disjoint before reaching $0$. 
They can be described by the \textbf{radial Loewner flow} $g_t$ (i.e., the conformal isomorphism $g_t \colon \D\setminus \bs{\slecurv}[0, t] \to \D$ such that $g_t(0)=0$ and $g_t'(0)>0$), which satisfies the flowing ODE
\begin{align}\label{eqn::ODE}
\partial_t g_t(z)=g_t(z)\sum_{j=1}^n\frac{\ee^{\ii\Theta_j(t)}+g_t(z)}{\ee^{\ii\Theta_j(t)}-g_t(z)},
\qquad t\ge 0 ,
\qquad \textnormal{where } \ee^{\ii\Theta_j(t)}=g_t(\slecurv_j(t)) .
\end{align}
Here and throughout, the time-parametrisation of the curves is the common parametrisation (by capacity).
If there is just one curve, we parametrise it by default by its capacity. 
For $\kappa\in (0,4]$, the \textbf{$n$-sided radial $\SLE_\kappa$} process is obtained by taking the Loewner driving functions to be the unique strong solution $\bs\Theta(t) = (\Theta_1(t),\ldots,\Theta_n(t))$ to the SDEs\footnote{We emphasize that the time $t$ in the present article corresponds to time $\frac{\kappa}{4} t$ in~\cite{Healey-Lawler:N_sided_radial_SLE}, 
and $\Theta_j(t)$ corresponds to $2\theta_t^j$ in~\cite{Healey-Lawler:N_sided_radial_SLE}. 
The choice of our normalisation is more suitable in the setting of flow-line couplings/imaginary geometry.}
(compare with~\eqref{eqn:DBM}) 
\begin{align}\label{eqn::drivingfunction}
\ud\Theta_i(t) = \sqrt \kappa \, \ud W_i(t) + 2 \sum_{\substack{1\leq j\leq n\\j\neq i}} \cot\bigg(\frac{\Theta_i(t)-\Theta_j(t)}{2}\bigg)\ud t, \qquad 1 \leq i \leq n ,
\end{align}
where $(W_i(t) \colon t\ge 0)$ for $1\le i\le n$ are $n$ independent Brownian motions on $\R$ with $W_i(0) = \theta_i$ for each $i$.

We denote the $n$-sided radial $\SLE_\kappa$ curves in common parametrisation by $\bs{\eta}(t) = (\eta_1(t),\ldots,\eta_n(t))$.
Up to any time before the curves meet, its law (for $\kappa\in (0,4]$) is absolutely continuous with respect to that of $n$ independent radial $\SLE_\kappa$ curves in the common parametrisation, 
which we shall denote by $\bs{\slecurv}(t) = (\slecurv_1(t),\ldots,\slecurv_n(t))$ throughout; see~\cite[Theorem~3.12]{Healey-Lawler:N_sided_radial_SLE}. 

\bigskip

We will also use $\SLE_\kappa$ curves with force points. 
For $p\ge 2$ and $\bs{\rho} := (\rho_2, \ldots, \rho_p) \in\R^{p-1}$, we define radial $\SLE_\kappa(\bs{\rho})$ 
in $(\D; \ee^{\ii\bs{\theta}}; 0)$ with $p$ distinct boundary force points $(\ee^{\ii\theta_2}, \ldots, \ee^{\ii\theta_p}) \in \UC^{p-1}$
as the curve $\slecurv$ started at $\ee^{\ii\theta_1} \in \UC$ whose Loewner driving function $\Theta$ solves the SDEs 
\begin{align}\label{eqn::SLE_kappa_rho}
\ud \Theta(t) = \; & \sqrt{\kappa} \, \ud W(t) 
- \sum_{j=2}^p \frac{\rho_j}{2} \cot \bigg(\frac{V_{j}(t) - \Theta(t)}{2}\bigg) \ud t , \qquad \Theta(0)=\theta_1, 
\\
\ud V_{j}(t) =  \; &  \cot \bigg(\frac{V_{j}(t) - \Theta(t)}{2}\bigg) \ud t, \qquad V_{j}(0)=\theta_j, 
\qquad j\in \{2, \ldots, p\} ,
\end{align}
where $\ee^{\ii V_{j}(t)} = g_t(\ee^{\ii\theta_j})$ are the time-evolutions of the force points, and $W$ is a one-dimensional Brownian motion;
the associated Loewner flow $g_t \colon \D\setminus \slecurv[0, t] \to \D$ (with $g_t(0)=0$ and $g_t'(0)>0$) satisfies the ODE
\begin{align}\label{eqn::ODE1}
\partial_t g_t(z) = g_t(z) \, \frac{\ee^{\ii\Theta(t)}+g_t(z)}{\ee^{\ii\Theta(t)}-g_t(z)},
\qquad t\ge 0 .
\end{align}
We define radial $\SLE_\kappa(\bs{\rho})$ in $\OmegaFill$ started at $z_1 \in \BdryOut$ with force points $z_2,\ldots,z_n \in \BdryOut$ (all distinct) 
as the pushforward measure of $\SLE_\kappa(\bs{\rho})$ from $(\D; \ee^{\ii\bs{\theta}}; 0)$ 
by the conformal bijection $\phi^{-1} \colon \D \to \OmegaFill$, where $\phi(z_j) =: \ee^{\ii\theta_j}$, for each $j$.
We will be interested in the case where $\rho_j=2$ for all $j$.

\bigskip

Now, let $\bs z = (z_1,\ldots,z_n)$ be $n$ distinct points located on $\BdryOut$ in counterclockwise order, and let $\kappa \in (0,4]$.
(Later, we will set $\kappa=2$.) 

\begin{definition} \label{D:inductive}
For $\eps>0$, we sample random curves $(\eta^{\bs z}_{1,\eps},\ldots,\eta^{\bs z}_{n,\eps})$ as follows.
\begin{itemize}[leftmargin=*]
\item
The law of $\eta^{\bs z}_{1,\eps}$ is that of radial $\SLE_\kappa(2,\ldots,2)$ from $z_1$ to $0$ on $\OmegaFill$ 
stopped when hitting $\partial B(0,\eps)$, with boundary force points $z_2,\ldots, z_n$. 
We let $T_{1,\eps}$ be a stopping time for $\eta^{\bs z}_{1,\eps}$ before the hitting time.

\item
Inductively, for $2\le j\le n$, the conditional law of $\eta^{\bs z}_{j,\eps}$ given $(\eta^{\bs z}_{1,\eps},\ldots,\eta^{\bs z}_{j-1,\eps})$ 
is that of the law of radial $\SLE_\kappa(2,\ldots,2)$ from $z_{j}$ to $0$ on $\OmegaFill\setminus \cup_{i=1}^{j-1}\eta^{\bs z}_{i,\eps}$ stopped when hitting $\partial B(0,\eps)$, 
with force points $z_{j+1},\ldots, z_n,\eta^{\bs z}_{1,\eps}(T_{1,\eps}),\ldots,\eta^{\bs z}_{j-1,\eps}(T_{j-1,\eps})$. 
We let $T_{j,\eps}$ be a stopping time for $\eta^{\bs z}_{j,\eps}$ before the hitting time.
\end{itemize}

We call this the \textbf{inductive construction} in $\OmegaFill$ from $\bs{z}$ to $0$.
\end{definition}

The next result gives the key scaling limit description for the LERW rays.
The convergence takes place weakly for probability measures on the complete and separable metric space $(X, \dist_X)$ of planar oriented unparametrised curves with the standard metric 
\begin{align}\label{eqn::curve_metric}
\dist_X(\eta, \tilde{\eta}):=\inf_{\varrho,\tilde{\varrho}}\sup_{t\in [0,1]}|\eta(\varrho(t))-\tilde{\eta}(\tilde{\varrho}(t))|, \qquad \eta, \tilde{\eta} \in X ,
\end{align}
where the infimum is taken over all increasing homeomorphisms $\varrho, \tilde{\varrho} \colon [0,1]\to [0,1]$. 
Abusing notation, we also write $\TVnormInline{\eta - \tilde{\eta}}$ for the total variation distance between the laws of two random variables $\eta, \tilde{\eta}$ in $X$.
For many curves, we use the metric space $X_{0}^n = X_{0}^n(\Omega) \subset X^n$ comprising $n$-tuples of disjoint simple unparameterised curves which only intersect $\BdryOut$ at their starting points, endowed with 
the metric
\begin{align}\label{eqn::curve_metric_many}
\dist_X((\eta_1, \ldots, \eta_N),(\tilde{\eta}_1, \ldots, \tilde{\eta}_N) ) := \max_{1 \leq j \leq N} \dist_X(\eta_j, \tilde{\eta}_j) .
\end{align}
Note that the metric space $(X_{0}^n, \dist_X)$ is not complete, but $\SLE_\kappa$ curves for $\kappa \in (0,4]$ belong to this space almost surely up to any stopping time before the curves meet.

\begin{proposition}\label{prop::scaling_curve}
Conditionally on the event $E^\delta_{\bs{x}}$ defined in~\eqref{eqn::disjointE}, 
the law of $(\cev{\gamma}_1^{\delta},\ldots,\cev{\gamma}_n^{\delta})$ converges weakly in $(X_{0}^n,\dist_X)$ and locally uniformly 
to a limit $(\cev{\gamma}_1,\ldots,\cev{\gamma}_n)$ 
by first letting $\delta\to 0$ and then shrinking $\diam(\BdryIn) \to 0$. 
Moreover, the following properties hold almost surely.
\begin{enumerate}
\item \label{item::scaling_curve1}
Each $\cev{\gamma}_j$ is a simple curve and intersects $\BdryOut$ only at its starting point, for $1 \leq j \leq n$. 
Moreover, the curves $(\cev{\gamma}_1[0,t_1],\ldots,\cev{\gamma}_n[0,t_n])$ are disjoint at any times before reaching $0$.

\item \label{item::scaling_curve2}
Let $z_j \in \BdryOut$ be the starting point of $\cev{\gamma}_j$, 
for $1 \leq j \leq n$.
Then, the density of $(\phi(z_1),\ldots,\phi(z_n))$ \textnormal{(}with respect to the product Lebesgue measure on $\partial\D$\textnormal{)} 
is given  by~\eqref{eqn::density}. For every $\eps>0$, the conditional law of $(\cev{\gamma}_1,\ldots,\cev{\gamma}_n)$ given $(z_1,\ldots,z_n)$ stopped when hitting $\partial B(0,\eps)$ 
is that of $(\eta^{\bs z}_{1,\eps},\ldots,\eta^{\bs z}_{n,\eps})$ for $\kappa=2$.
\end{enumerate}
\end{proposition}

Before proving Proposition~\ref{prop::scaling_curve}, we first give another description of the continuum curves $(\eta^{\bs z}_{1,\eps},\ldots,\eta^{\bs z}_{n,\eps})$.

\begin{lemma}\label{lem::flowline_nsided}
Fix $\kappa \in (0,4]$.
Let $\bs z = (z_1,\ldots,z_n) \in (\BdryOut)^n$ be $n$ distinct points in counterclockwise order. Let $\bs{\slecurv}^{\bs{z}} = (\slecurv_1^{\bs{z}},\ldots,\slecurv_n^{\bs{z}})$ be $n$ independent radial $\SLE_\kappa$ curves from $\bs{z}$ to $0$ on $\OmegaFill$.
For $\eps>0$ and $1 \leq j \leq n$, let $T_{j,\eps}$ be a stopping time for $\slecurv_j^{\bs{z}}$ before the hitting time to $B(0,\eps)$. 
Then, the law of the curves in the inductive construction, $\bs{\eta}^{\bs{z}} [ \bs{0}, \bs{T}_\eps] : = (\eta^{\bs z}_{1,\eps}[0,T_{1,\eps}],\ldots,\eta^{\bs z}_{n,\eps}[0,T_{n,\eps}])$, coincides with the law of the curves
$\bs{\slecurv}^{\bs{z}}[\bs{0}, \bs{T}_\eps] = (\slecurv_1^{\bs{z}}[0, T_{1,\eps}], \ldots, \slecurv_n^{\bs{z}}[0, T_{n,\eps}])$
weighted by the martingale
\begin{align}\label{eqn::RN_n_curves}
\;&R( \OmegaFill; \bs{\slecurv}^{\bs{z}}[\bs{0}, \bs{T}_\eps]) 
\notag\\
:=\;& \one{\{I_n^\eps\}} \; 
\bigg(\frac{\psi_\eps'(0)}{\phi'(0)}\bigg)^{\frac{n^2-1}{2\kappa}+\frac{(6-\kappa)(\kappa-2)}{8\kappa}} 
\; \exp \left( \frac{(3\kappa-8)(6-\kappa)}{4\kappa} 
\, \mu_{\OmegaFill\setminus \bs{\slecurv}^{\bs{z}}[\bs{0}, \bs{T}_\eps]} \Big[(N(\ell)-1) \; 
\one{\{\ell \cap \bs{\slecurv}^{\bs{z}}[\bs{0}, \bs{T}_\eps] \neq\emptyset\}} \Big] \right)
\notag\\
\;&\times \prod_{j=1}^{n}\Big(\big(\psi_\eps\circ\psi_{j,\eps}^{-1}\big)'(\psi_{j,\eps}(\slecurv_j^{\bs{z}}(T_{j,\eps})))\Big)^{\frac{6-\kappa}{2\kappa}} 
\; 
\left|\frac{\underset{1\le i<j\le n}{\prod} \, \sin\big(\frac{1}{2} (\arg\psi_\eps(\slecurv_j^{\bs{z}}(T_{j,\eps}))-\arg\psi_\eps(\slecurv_i^{\bs{z}}(T_{i,\eps}))) \big)}{\underset{1\le i<j\le n}{\prod} \,  \sin\big(\frac{1}{2} (\arg\phi(z_j)-\arg\phi(z_i)) \big)}\right|^{\frac{2}{\kappa}} 
\notag\\
\;&\times\prod_{j=1}^{n} \bigg(\frac{\psi'_{j,\eps}(0)}{\phi'(0)}\bigg)^{\frac{(6-\kappa)(2-\kappa)}{8\kappa}} ,
\end{align}
where $\psi_\eps \colon \OmegaFill\setminus\bs{\slecurv}^{\bs{z}}[\bs{0}, \bs{T}_\eps] \to \D$ 
is the conformal isomorphism such that $\psi_\eps(0)=0$ and $\psi'_\eps(0)>0$, 
and $\psi_{j,\eps} \colon \OmegaFill\setminus\slecurv_j^{\bs{z}}[0,T_{j,\eps}] \to \D$ 
is the conformal isomorphism such that $\psi_{j,\eps}(0)=0$ and $\psi'_{j,\eps}(0)>0$ for each $j$,  
and $N(\ell)$ is the number of curves in $\{\slecurv_1^{\bs{z}}[0,T_{1,\eps}],\ldots,\slecurv_n^{\bs{z}}[0,T_{n,\eps}]\}$ which intersect the loop $\ell$, and lastly, 
\begin{align*}
I_n^\eps:=\{\slecurv_i^{\bs{z}}[0,T_{i,\eps}]\cap\slecurv_j^{\bs{z}}[0,T_{j,\eps}]=\emptyset\textnormal{ for }1\le i<j\le n\}.
\end{align*}
In other words, we have
\begin{align*}
\frac{\ud \, \bs{\eta}^{\bs{z}} [\bs{0}, \bs{T}_{\eps}]}
{\ud \, \bs{\slecurv}^{\bs{z}}[\bs{0}, \bs{T}_\eps]} 
= R( \OmegaFill; \bs{\slecurv}^{\bs{z}}[\bs{0}, \bs{T}_\eps]) ,
\end{align*}
where we use $\mathrm{d} X/ \ud \, Y$ with an abuse of notation to denote the Radon--Nikodym derivative of the law of the random variable $X$ with respect to that of $Y$. 
\end{lemma}

\begin{proof}
By conformal invariance, we may assume that $\OmegaFill=\D$. 
Write $D_{j,\eps} := \D\setminus\bigcup_{i=1}^{j}\eta^{\bs z}_{i,\eps}[0,T_{i,\eps}]$ for each~$j$. 
Fix $j\in \{1,\ldots, n\}$. 
We shall prove that, given $(z_1,\ldots,z_n)$, the conditional law of $\eta^{\bs z}_{j,\eps}[0,T_{j,\eps}]$ given $\eta^{\bs z}_{1,\eps}[0,T_{1,\eps}],\ldots,\eta^{\bs z}_{j-1,\eps}[0,T_{j-1,\eps}]$ is that of $\slecurv_{j}^{\bs{z}}[0,T_{j,\eps}]$ weighted~by
\begin{align}\label{eqn::RN_IM_1}
\frac{R(\D;\eta^{\bs z}_{1,\eps}[0,T_{1,\eps}],\ldots,\eta^{\bs z}_{j-1,\eps}[0,T_{j-1,\eps}], \slecurv_j^{\bs{z}}[0,T_{j,\eps}], z_{j+1},\ldots, z_n)}{R(\D;\eta^{\bs z}_{1,\eps}[0,T_{1,\eps}],\ldots,\eta^{\bs z}_{j-1,\eps}[0,T_{j-1,\eps}],z_{j},\ldots, z_n)}.
\end{align}
Denote by $\phi_{j,\eps} \colon D_{j,\eps} \to \D$ the conformal isomorphism such that $\phi_{j,\eps}(0)=0$ and $\phi_{j,\eps}'(0)>0$.
Denote by $\eta_{j}$ the radial $\SLE_\kappa$ from $z_{j}$ to $0$ on $D_{j-1,\eps}$. 
Then, by~\cite[Proposition 2.2]{Healey-Lawler:N_sided_radial_SLE}, it follows that the Radon--Nikodym derivative of $\eta_j$ with respect to $\slecurv_{j}^{\bs{z}}$ is given by 
\begin{align}\label{eqn::RN_IM_2}
\;& J(\eta^{\bs z}_{1,\eps}[0,T_{1,\eps}],\ldots,\eta^{\bs z}_{j-1,\eps}[0,T_{j-1,\eps}]; \slecurv_j^{\bs{z}}[0,T_{j,\eps}])
\\
:=\;&\one{\Big\{\slecurv_{j}^{\bs{z}}[0,T_{j,\eps}]\cap\bigcup_{i=1}^{j-1}\eta^{\bs z}_{i,\eps}[0,T_{i,\eps}]=\emptyset\Big\}}
\bigg(\frac{(\phi_{j,\eps} \circ \psi_{j,\eps}^{-1})'(\psi_{j,\eps}(\slecurv_j^{\bs{z}}(T_{j,\eps})))}{\phi'_{j-1,\eps}(z_j)}\bigg)^{\frac{6-\kappa}{2\kappa}}\bigg(\frac{(\phi_{j,\eps} \circ \psi_{j,\eps}^{-1})'(0)}{\phi'_{j-1,\eps}(0)}\bigg)^{\frac{(6-\kappa)(\kappa-2)}{8\kappa}}\notag\\
\;&\times \exp \bigg(\frac{(3\kappa-8)(6-\kappa)}{4\kappa} \; \mu_{D_{j-1,\eps}}\Big[\ell\in\LL \textnormal{ such that } \ell \cap \bigcup_{i=1}^{j-1}\eta^{\bs z}_{i,\eps}[0,T_{i,\eps}]\neq \emptyset\textnormal{ and }\ell \cap \slecurv_{j}^{\bs{z}}[0,T_{j,\eps}] \neq \emptyset \Big] \bigg)
\notag
\end{align}

Next, denote 
\begin{align*}
(w_1,\ldots,w_{n-1}) := (\phi_{j-1,\eps}(\eta^{\bs z}_{1,\eps}(T_{1,\eps})),\ldots,\phi_{j-1,\eps}(\eta^{\bs z}_{j-1,\eps}(T_{j-1,\eps})),\phi_{j-1,\eps}(z_j),\ldots,\phi_{j-1,\eps}(z_n)) ,
\end{align*} 
and denote by $(\tilde \Theta \colon t\ge 0)$ the driving function of $\phi_{j-1,\eps}(\eta_{j})$ and by $(\tilde g_t \colon t\ge 0)$ the corresponding radial Loewner flow as in~\eqref{eqn::ODE1}. 
By It\^o's formula, the process $\{M(t\wedge T_{j,\eps})\}_{t\ge 0}$ defined by 
\begin{align*}
M(t) := \;& \ee^{\frac{n^2-1}{2\kappa}t} \;  \prod_{i=1}^{n-1} |\tilde g'_t(w_i)|^{\frac{6-\kappa}{2\kappa}} \;
\prod_{i=1}^{n-1} \bigg|\sin\left(\frac{\arg\tilde g_t(w_i)-\tilde \Theta(t)}{2}\right)\bigg|^{\frac{2}{\kappa}}
\prod_{1\le i<l\le n-1}\bigg|\sin\left(\frac{\arg\tilde g_t(w_l)-\arg\tilde g_t(w_i)}{2}\right)\bigg|^{\frac{2}{\kappa}}
\end{align*}
is a martingale for $\phi_{j-1,\eps}(\eta_{j})$. 
By Girsanov's theorem, 
the conditional law of $\phi_{j-1,\eps}(\eta^{\bs z}_{j,\eps}[0,T_{j,\eps}])$ given $\eta^{\bs z}_{1,\eps}[0,T_{1,\eps}],\ldots,\eta^{\bs z}_{j-1,\eps}[0,T_{j-1,\eps}]$ is that of 
$\phi_{j-1,\eps}(\slecurv_{j}^{\bs{z}}[0,T_{j,\eps}])$ tilted by $M(T_{j,\eps})$.
Combining with~\eqref{eqn::RN_IM_2}, we obtain~\eqref{eqn::RN_IM_1}. 
Lastly, taking the product of~\eqref{eqn::RN_IM_1} over $1\le j\le n$, we complete the proof.
\end{proof}

We now consider the case of one curve, $n=1$, where the limit $\cev{\gamma}_1$, if stopped when hitting $\partial B(0,\eps)$ at time $\cev{\exitTT}_{\eps}$,
is just a radial $\SLE_2$ curve $\eta$ stopped at the same\footnote{For simplicity of notation, we will denote the stopping times for various curves by the same symbol, though for different curves the hitting time has different values. We will parameterise curves by capacity.} 
hitting time $\exitTT_{\eps}$,
so that $\cev{\gamma}_1[0,\cev{\exitTT}_{\eps}] \sim \eta[0,\exitTT_{\eps}]$.

\begin{proof}[Proof of Proposition~\ref{prop::scaling_curve}, $n=1$]
Obtaining tightness of $\{\cev{\gamma}_1^\delta\}_{\delta>0}$ in $(X, \dist_X)$ is standard~\cite[Theorem~3.9]{LSW:Conformal_invariance_of_planar_LERW_and_UST}. 
The input of the proof is the Beurling estimate for simple random walk (see~\cite[Lemma~3.2]{Berestycki-Liu:Piecewise_Temperleyan_dimers_and_multiple_SLE8} 
for the proof with the presence of a reflective boundary). 
Suppose $\cev{\gamma}_1 \in X$ is any subsequential limit. 
To show the convergence of the whole sequence, we will prove that for every $\eps > \diam(\BdryIn)$, we have
\begin{align}\label{eqn::TV_conv_1}
\lim_{\diam(\BdryIn)\to 0} \TVnorm{\eta[0,\exitTT_{\eps}] - \cev{\gamma}_1[0,\cev{\exitTT}_{\eps}]} = 0 .
\end{align}

Let $\OmegaFill^\delta$ be the discrete simply connected domain bounded by $\BdryOut^\delta$ on $\delta\Z^2$. 
For $w^\delta\in\Omega^\delta$, denote by $\Gamma_{w}^\delta$ (resp.~$\Gamma_{\scalebox{0.6}{\textnormal{fill}},w}^\delta$) the loop-erasure of a random walk $\LR\sim \Prob_{w}$ on $\Omega^\delta$ (resp.~$\OmegaFill^\delta$) starting from $w^\delta$ and ending at $\BdryOut^\delta$, reflected at $\BdryIn^\delta$. 
Denote by $\cev{\Gamma}_{w}^\delta$ (resp. $\cev{\Gamma}_{\scalebox{0.6}{\textnormal{fill}},w}^\delta$) the time-reversal of $\Gamma_{w}^\delta$ (resp. $\Gamma_{\scalebox{0.6}{\textnormal{fill}},w}^\delta$). 

Let $s > 0$ be such that $\diam(\BdryIn^\delta)<\frac{s}{2}$ for all $\delta$ small enough.  
Take $w\in\OmegaFill$ such that $\dist(w,0)=s$ and $w^\delta$ on $\Omega^\delta$ such that $|w-w^\delta|\le 2\delta$. 
Since the random walk converges to Brownian motion, we have
\begin{align}\label{eqn::TV_discrete1}
\begin{split}
\limsup_{\delta\to 0} \TVnorm{ \Gamma_{w}^\delta - \Gamma_{\scalebox{0.6}{\textnormal{fill}},w}^\delta }
\le \; & \Prob_{w} \big[\LR\textnormal{ hits }\BdryIn^\delta\textnormal{ before hitting }\BdryOut^\delta \big] 
\\
\le \; & \frac{\log (s/\dist(0,\BdryOut))}{\log (\diam(\BdryIn)/\dist(0,\BdryOut))}.
\end{split}
\end{align}
By Wilson's algorithm (Theorem~\ref{thm: Wilsons algorithm}), 
for every discrete simple curve $\simpleCurv^\delta$ on $\Omega^\delta$ started at $\BdryOut^\delta$, we have 
\begin{align*}
\PP^\delta \big[\cev{\gamma}_1^\delta[0,\cev{\exitTT}_{\eps}^\delta] = \simpleCurv^\delta\big] 
= \E^\delta \Bigg[\frac{\Prob_{x_1}[\LR\textnormal{ hits }\BdryOut^\delta\cup\simpleCurv^\delta\textnormal{ at }\simpleCurv^\delta(\cev{\exitTT}_{\eps}^\delta)]}{\Prob_{w}[\LR\textnormal{ hits }\BdryOut^\delta\cup\simpleCurv^\delta\textnormal{ at }\simpleCurv^\delta(\cev{\exitTT}_{\eps}^\delta)]} 
\; \one{\big\{\cev{\Gamma}_{w}^\delta[0,\cev{\exitTT}_{\eps}^\delta] = \simpleCurv^\delta\big\}}\Bigg] ,
\end{align*}
where $x_1^\delta \in \BdryIn^\delta \subset B(0,\eps)$ is the eventual endpoint of $\cev{\gamma}_1^\delta$. 
Thus, the Radon--Nikodym derivative of $\cev{\gamma}_1^\delta[0,\cev{\exitTT}_{\eps}^\delta]$ with respect to $\cev{\Gamma}_{w}^\delta[0,\cev{\exitTT}_{\eps}^\delta]$ is given by 
\begin{align*}
A^\delta\big( \cev{\Gamma}_{w}^\delta[0,\cev{\exitTT}_{\eps}^\delta] \big) 
:= \frac{\Prob_{x_1}[\LR\textnormal{ hits }\BdryOut^\delta\cup\cev{\Gamma}_{w}^\delta[0,\cev{\exitTT}_{\eps}^\delta]\textnormal{ at }\cev{\Gamma}_{w}^\delta(\cev{\exitTT}_{\eps}^\delta)]}{\Prob_{w}[\LR\textnormal{ hits }\BdryOut^\delta\cup\cev{\Gamma}_{w}^\delta[0,\cev{\exitTT}_{\eps}^\delta]\textnormal{ at }\cev{\Gamma}_{w}^\delta(\cev{\exitTT}_{\eps}^\delta)]} .
\end{align*}
By a Harnack-type inequality, there exist constants $C_1,C_2>0$ such that 
\begin{align*}
\big| A^\delta\big( \cev{\Gamma}_{w}^\delta[0,\cev{\exitTT}_{\eps}^\delta] \big) - 1 \big|\le C_1\Big(\frac{s}{\eps}\Big)^{C_2} ,
\end{align*}
for every realisation of $\cev{\Gamma}_{w}^\delta[0,\cev{\exitTT}_{\eps}^\delta]$. 
This implies that
\begin{align}\label{eqn::TV_discrete2}
\limsup_{\delta\to 0} \TVnorm{ \cev{\gamma}_1^{\delta}[0,\cev{\exitTT}_{\eps}^\delta] - \cev{\Gamma}_{w}^\delta[0,\cev{\exitTT}_{\eps}^\delta] } 
\le C_1\Big(\frac{s}{\eps}\Big)^{C_2}.
\end{align}
Next, choose a discrete approximation $o^\delta$ of $0$ on $\OmegaFill^\delta$. 
By the discrete Beurling estimate, there exist constants $C_3>0$ and $C_4>0$ such that in $\OmegaFill^\delta$, we have
\begin{align*}
\PP^\delta \Big[\cev{\Gamma}_{\scalebox{0.6}{\textnormal{fill}},o}^\delta[0,\cev{\exitTT}_{\scalebox{0.6}{\textnormal{fill}},\eps}^\delta]\neq\cev{\Gamma}_{\scalebox{0.6}{\textnormal{fill}},w}^\delta[0,\cev{\exitTT}_{\scalebox{0.6}{\textnormal{fill}},\eps}^\delta]
\Big]
\le C_3\Big(\frac{s}{\eps}\Big)^{C_4}.
\end{align*}
By the well-known result in~\cite[Theorem~3.9]{LSW:Conformal_invariance_of_planar_LERW_and_UST}, 
the LERW $\cev{\Gamma}_{\scalebox{0.6}{\textnormal{fill}},o}^\delta$ on $\OmegaFill^\delta$ converges in law to the radial $\SLE_2$ curve $\eta$ both in $(X, \dist_X)$ and locally uniformly. 
In particular, combining with~(\ref{eqn::TV_discrete1},~\ref{eqn::TV_discrete2}), we conclude that
there exist constants $C_5>0$ and $C_6>0$ such that we have
\begin{align}\label{eqn::TV_discrete4}
\TVnorm{\eta[0,\exitTT_{\eps}] - \cev{\gamma}_1[0,\cev{\exitTT}_{\eps}]} \le C_5\Big(\frac{s}{\eps}\Big)^{C_6}.
\end{align}
Taking $s \to 0$ along with $\diam(\BdryIn)\to 0$, this shows~\eqref{eqn::TV_conv_1} and completes the proof of the identification of the scaling limit.
The properties~\ref{item::scaling_curve1}~\&~\ref{item::scaling_curve2} follow from properties of the radial $\SLE_2$~\cite{LSW:Conformal_invariance_of_planar_LERW_and_UST, Rohde-Schramm:Basic_properties_of_SLE}. 
\end{proof}

The next lemma establishes tightness of the multiple UST branches.

\begin{lemma}\label{lem::scaling_one_step}
Conditionally on the event $E^\delta_{\bs{x}}$, 
the family $\{(\cev{\gamma}_1^\delta,\ldots,\cev{\gamma}_n^\delta)\}_{\delta>0}$ is tight in $(X_0^n, \dist_X)$. 
\end{lemma}

\begin{proof}
Take disjoint subdomains $U_1,\ldots, U_n$ of $\Omega$ such that $x_j\in\BdryIn\cap U_j$ and $\BdryOut\cap \partial U_j \neq \emptyset$, 
and discrete approximations $U_1^\delta,\ldots, U_n^\delta$ on $\Omega^\delta$ such that 
$x_j^\delta \in\BdryIn^\delta\cap U_j^\delta$ and $\BdryOut^\delta\cap \partial U_j^\delta \neq \emptyset$, 
and $\dist(\partial U_j^\delta,\partial U_j)\to 0$ as $\delta\to 0$, for $1\le j\le n$. 
There exists a constant $C=C(\Omega)>0$ such that
\begin{align}\label{eqn::lower_bound_prob}
\PP^\delta [E^\delta_{\bs{x}}] 
\; \ge \; \prod_{j=1}^{n} \Prob_{x_j}[\LR\textnormal{ hits }\partial U_j^\delta\cap\BdryOut^\delta \textnormal{ when it exits }U_j^\delta ]
\; \ge \; C 
\end{align}
for all $\delta$ small enough. 
Thus, the tightness of $\{(\cev{\gamma}_1^\delta,\ldots,\cev{\gamma}_n^\delta)\}_{\delta>0}$ 
in $(X_0^n, \dist_X)$ is readily obtained from the case of $n=1$,
and any subsequential limit comprises simple curves $(\cev{\gamma}_1,\ldots,\cev{\gamma}_n)$. 
To show that they are also disjoint, by symmetry it suffices to show that $\cev{\gamma}_1$ avoids the other curves. 
Indeed, the discrete Beurling estimate implies that there exist constants $C_1,C_2>0$ such that 
\begin{align*}
\;& \PP^\delta \big[ \PP^\delta \big[ \{\gamma_j^\delta\textnormal{ hits the } \eps\textnormal{-neighbourhood of }\gamma_1^\delta \} \cap E^\delta_{\bs{x}} \cond\gamma_1^\delta ] \big]
\\
\le \;& \Prob_{x_j}[\LR\textnormal{ ends at }\BdryOut^\delta\textnormal{ and hits the } \eps\textnormal{-neighbourhood of }\gamma_1^\delta]
\; \le \; C_1 \prod_{i=2}^{n} \eps^{C_2} , \qquad \eps>0 .
\end{align*}
By letting $\delta\to 0$, we find that $\cev{\gamma}_1\cap\cev{\gamma}_j=\emptyset$ almost surely for all $2\le j\le n$. 
Thus, $(\cev{\gamma}_1,\ldots,\cev{\gamma}_n) \in X_0^n$. 
\end{proof}

We are now ready to finish the proof of Proposition~\ref{prop::scaling_curve}.
We use the same notation as in the above proofs, and recall Wilson's algorithm (Theorem~\ref{thm: Wilsons algorithm}) relating LERW to the UST branches. 

\begin{proof}[Proof of Proposition~\ref{prop::scaling_curve}, $n\ge 2$]
By conformal invariance, we may assume that $\Omega=\A_r$.  
We proceed by induction on $n\ge 1$. The base case $n = 1$ was established above. 
Fix $k\ge 2$ and suppose that the claims in Proposition~\ref{prop::scaling_curve} hold for all $1\le n\le k-1$. 
For $k$ disjoint simple discrete curves $(\simpleCurv^{\delta}_1,\ldots,\simpleCurv^\delta_k)$, we have
\begin{align*}
\;&\PP^\delta \big[ (\cev{\gamma}_1^{\delta}[0,\cev{\exitTT}_{\eps}^\delta],\ldots,\cev{\gamma}_k^{\delta}[0,\cev{\exitTT}_{\eps}^\delta])=(\simpleCurv_1^\delta,\ldots,\simpleCurv_k^\delta) \big] 
\\
= \;& \frac{1}{\PP^\delta [E^\delta_{\bs{x}}]}\frac{\#\{\textnormal{spanning trees $\LT^\delta$ on }\A_r^\delta \textnormal{ such that } \cev{\bs\gamma}^\delta \in E^\delta_{\bs{x}}\textnormal{ and }(\cev{\gamma}_1^{\delta}[0,\cev{\exitTT}_{\eps}^\delta],\ldots,\cev{\gamma}_k^{\delta}[0,\cev{\exitTT}_{\eps}^\delta])=(\simpleCurv_1^\delta,\ldots,\simpleCurv_k^\delta)\}}{\#\{\textnormal{spanning trees $\LT^\delta$ on }\A_r^\delta\}}\\
=\;&h_{\A_r^\delta,\bs\simpleCurv^\delta}(x_1^\delta,\simpleCurv_1^\delta(\cev{\exitTT}_{\eps}^\delta))
\; \frac{M(\bs\simpleCurv^\delta;\cev{\exitTT}_{\eps}^\delta)}{p^\delta(\A_r^\delta;\bs{x}^\delta;\hat{\bs{x}}^\delta)} 
\; \frac{\#\{\textnormal{spanning trees $\LT^\delta$ on }\A_r^\delta\cup\bs\simpleCurv^\delta\}}{\#\{\textnormal{spanning trees $\LT^\delta$ on }\A_r^\delta\cup\hat{\bs\simpleCurv}^\delta\}}
\\
\;&\times\frac{\PP^\delta [E(\A_r^\delta\cup\hat{\bs\simpleCurv}^\delta;x_2^\delta,\ldots,x_k^\delta;\simpleCurv_2^\delta(\cev{\exitTT}_{\eps}^\delta),\ldots,\simpleCurv_k^\delta(\cev{\exitTT}_{\eps}^\delta))]}{\PP^\delta [E^\delta_{\hat{\bs{x}}}]}
\; \frac{\#\{\textnormal{spanning trees $\LT^\delta$ on }\A_r^\delta\cup\hat{\bs\simpleCurv}^\delta\}}{\#\{\textnormal{spanning trees $\LT^\delta$ on }\A_r^\delta\}} ,
\end{align*}
where $\eps > 2r$,
and we denote by $\bs\simpleCurv^\delta = \bigcup_{j=1}^{k}\simpleCurv_j^\delta$, 
and $\hat{\bs\simpleCurv}^\delta = \bigcup_{j=2}^{k}\simpleCurv_j^\delta$, and 
by $\A_r^\delta\cup\bs\simpleCurv^\delta$ (resp.~$\A_r^\delta\cup\hat{\bs\simpleCurv}^\delta$)
the graph obtained by wiring the outer boundary $\AnnBdryOut_r^\delta$ and $\bs\simpleCurv^\delta$ (resp.~$\hat{\bs\simpleCurv}^\delta$) together;
and where 
\begin{align*}
h_{\A_r^\delta\cup\bs\simpleCurv^\delta}(x_1^\delta,\simpleCurv_1^\delta(\cev{\exitTT}_{\eps}^\delta)) 
:= \; & \Prob_{x_1} \big[ \LR\textnormal{ hits }\AnnBdryOut_r^\delta\cup\bs\simpleCurv^\delta\textnormal{ through }\simpleCurv_1^\delta(\cev{\exitTT}_{\eps}^\delta) \big]
\; \in \; (0,1) ,
\\
M(\bs\simpleCurv^\delta;\cev{\exitTT}_{\eps}^\delta)
:= \; & \frac{1}{h_{\A_r^\delta\cup\bs\simpleCurv^\delta}(x_1^\delta,\simpleCurv_1^\delta(\cev{\exitTT}_{\eps}^\delta))}
\; \frac{\PP^\delta [E(\A_r^\delta\cup\bs\simpleCurv^\delta;x_1^\delta,\ldots,x_k^\delta;\simpleCurv_1^\delta(\cev{\exitTT}_{\eps}^\delta),\ldots,\simpleCurv_k^\delta(\cev{\exitTT}_{\eps}^\delta))]}{\PP^\delta [E(\A_r^\delta\cup\hat{\bs\simpleCurv}^\delta;x_2^\delta,\ldots,x_k^\delta;\simpleCurv_2^\delta(\cev{\exitTT}_{\eps}^\delta),\ldots,\simpleCurv_k^\delta(\cev{\exitTT}_{\eps}^\delta))]} 
\; \in \; (0,1) ,
\\ 
p^\delta(\A_r^\delta;\bs{x}^\delta;\hat{\bs{x}}^\delta) 
:= \; & \frac{\PP^\delta [E^\delta_{\bs{x}}]}{\PP^\delta [E^\delta_{\hat{\bs{x}}}]} 
\; = \; \frac{\PP^\delta [E(\A_r^\delta;x_1^\delta,\ldots,x_k^\delta)]}{\PP^\delta [E(\A_r^\delta;x_2^\delta,\ldots,x_k^\delta)]} 
\; \in \; (0,1) .
\end{align*}

Consider now random curves $\cev{\bs{\Gamma}}^\delta = (\cev{\Gamma}_1^\delta,\ldots,\cev{\Gamma}_k^\delta)$ such that
\begin{itemize}[leftmargin=*]
\item the law of $(\cev{\Gamma}_2^\delta,\ldots,\cev{\Gamma}_k^\delta)$ is the conditional law of $(\cev{\gamma}_2^\delta,\ldots,\cev{\gamma}_k^\delta)$ given the event $E^\delta_{\hat{\bs{x}}} = E(\A_r^\delta;x_2^\delta,\ldots,x_k^\delta)$;

\item given $(\cev{\Gamma}_2^\delta,\ldots,\cev{\Gamma}_k^\delta)$, the law of $\cev{\Gamma}_1^\delta$ is that of the time-reversed LERW $\cev{\Gamma}_{x_1}^\delta$ starting from $x_1^\delta$ and ending at $\AnnBdryOut_r^\delta\cup\bigcup_{j=2}^{k}\cev{\Gamma}_j^\delta$, reflected at $\AnnBdryIn_r^\delta$.
\end{itemize}
Then, the Radon--Nikodym derivative of 
$\cev{\bs{\gamma}}^\delta[0,\cev{\exitTT}_{\eps}^\delta] = (\cev{\gamma}_1^\delta[0,\cev{\exitTT}_{\eps}^\delta],\ldots,\cev{\gamma}_k^\delta[0,\cev{\exitTT}_{\eps}^\delta])$ with respect to $\cev{\bs{\Gamma}}^\delta[0,\cev{\exitTT}_{\eps}^\delta]$ is 
\begin{align*}
L(\cev{\bs{\Gamma}}^\delta[0,\cev{\exitTT}_{\eps}^\delta]) 
:= \frac{ M(\cev{\bs{\Gamma}}^\delta[0,\cev{\exitTT}_{\eps}^\delta];\cev{\exitTT}_{\eps}^\delta) }{ p^\delta(\A_r^\delta;\bs{x}^\delta;\hat{\bs{x}}^\delta) }
\; \one{ \Big\{\cev{\Gamma}^\delta_1[0,\cev{\exitTT}_{\eps}^\delta] \cap \bigcup_{j=2}^{k}\cev{\Gamma}^\delta_j[0,\cev{\exitTT}_{\eps}]=\emptyset \Big\} }
\; \one{ \big\{ (\cev{\Gamma}^\delta_1(0),\ldots,\cev{\Gamma}^\delta_k(0))\in\LX_k \big\} }.
\end{align*}
Similarly to~\eqref{eqn::lower_bound_prob}, there exists a constant $C=C(r)>0$, which only depends on $r$, such that
\begin{align*}
p^\delta(\A_r^\delta;\bs{x}^\delta;\hat{\bs{x}}^\delta) \ge \PP^\delta [E^\delta_{\bs{x}}] \ge C .
\end{align*}
Combining this with the fact that $M(\cev{\bs{\Gamma}}_1^\delta[0,\cev{\exitTT}_{\eps}^\delta];\cev{\exitTT}_{\eps}^\delta)  \leq 1$, we see that $L(\cev{\bs{\Gamma}}^\delta[0,\cev{\exitTT}_{\eps}^\delta]) \le C$.

\medskip

Lemma~\ref{lem::scaling_one_step} implies that both families 
$\{\cev{\bs{\Gamma}}^\delta\}_{\delta>0}$ and $\{\cev{\bs{\gamma}}^\delta\}_{\delta>0}$ are tight in $(X_0^k, \dist_X)$. 
It follows similarly as in the proof of~\cite[Theorem~3.9]{LSW:Conformal_invariance_of_planar_LERW_and_UST} 
that the families are also tight in the uniform topology of curves with common parametrisation. 
Take any respective subsequential limits $\cev{\bs{\Gamma}}^{(r)} = (\cev{\Gamma}_1^{(r)},\ldots,\cev{\Gamma}_k^{(r)})$ and $\cev{\bs{\gamma}}^{(r)} = (\cev{\gamma}_1^{(r)},\ldots,\cev{\gamma}_k^{(r)})$.
By Skorokhod's representation theorem, without loss of generality we may assume that along subsequences, 
$\{\cev{\bs{\Gamma}}^\delta\}_{\delta>0}$ and $\{\cev{\bs{\gamma}}^\delta\}_{\delta>0}$ 
converge almost surely to $\cev{\bs{\Gamma}}^{(r)}$ and $\cev{\bs{\gamma}}^{(r)}$, respectively, in both $(X_0^k, \dist_X)$ and locally uniformly.
The induction hypothesis now implies the following facts. 
\begin{itemize}[leftmargin=*]
\item We have 
$\cev{\bs{\Gamma}}^{(r)}[0,\cev{\exitTT}_{\eps}] \to \cev{\bs{\Gamma}}[0,\cev{\exitTT}_{\eps}]$ as $r \to 0$ in total variation distance.
In particular, there exists a coupling $\LQ$ such that almost surely, for every $t\ge 0$, the curves $\cev{\bs{\Gamma}}^{(r)}[0,t]$ agree with the curves $\cev{\bs{\Gamma}}[0,t]$ when $r>0$ is small enough. 

\item The conditional law of $\cev{\Gamma}_1[0,\cev{\exitTT}_{\eps}]$ given $(\cev{\Gamma}_2[0,\cev{\exitTT}_{\eps}],\ldots,\cev{\Gamma}_k[0,\cev{\exitTT}_{\eps}])$ is that of 
the radial $\SLE_2$ curve $\slecurv_1[0,\exitTT_{\eps}]$ in $\hat{D}_\eps := \D\setminus\bigcup_{j=2}^{k}\cev{\Gamma}_j[0,\cev{\exitTT}_{\eps}]$ from $\cev{\Gamma}_1(0)$ to the origin.

\item The conditional law of the starting point $\cev{\Gamma}_1(0)$ is given by the harmonic measure in $\hat{D}_\eps$ seen from~$0$.
Thus, its density on $\UC=\partial\D$ equals 
\begin{align*}
\rho(z_1) = \frac{\hat\phi'_\eps(z_1)}{2\pi} \; \one{ \big\{(z_1,\cev{\Gamma}_2(0),\ldots,\cev{\Gamma}_k(0))\in\LX_k \big\} } ,
\end{align*}
where $\hat\phi_{\eps} \colon \hat{D}_\eps \to \D$ is the uniformising map such that $\hat\phi_{\eps}(0)=0$ and $\hat\phi_{\eps}'(0)>0$. 
\end{itemize}
We will also consider the uniformising map $\phi_{\eps} \colon \D \setminus \bigcup_{j=1}^{k}\cev{\Gamma}_j[0,\cev{\exitTT}_{\eps}] \to \D$ such that $\phi_{\eps}(0)=0$ and $\phi_{\eps}'(0)>0$,
and the conformal isomorphisms 
\begin{itemize}
\item $\phi^{(r)}_{\eps} \colon \A_r \setminus\bigcup_{j=1}^{k}\cev{\Gamma}_j[0,\cev{\exitTT}_{\eps}] \to \A_{r_{\eps}}$, such that $\phi^{(r)}_{\eps}(r) =: r({\cev{\exitTT}_{\eps}}) =: r_{\eps}$, so that $\frac{r_\eps}{r} \to \phi_\eps'(0)$ as $r\to 0$; and

\item $\hat\phi^{(r)}_{\eps} \colon \A_r\setminus \bigcup_{j=2}^{k}\cev{\Gamma}_j[0,\cev{\exitTT}_{\eps}] \to \A_{\hat{r}_{\eps}}$, such that $\hat\phi^{(r)}_{\eps}(r) =: \hat{r}({\cev{\exitTT}_{\eps}}) =: \hat{r}_{\eps}$, so that $\frac{\hat{r}_\eps}{r} \to \hat\phi_\eps'(0)$ as $r\to 0$.
\end{itemize}

Next, as in Proposition~\ref{prop::hittingpoints} and Lemma~\ref{lem::Poisson_limit}, 
one can prove that  
$M(\cev{\bs{\Gamma}}_1^\delta[0,\cev{\exitTT}_{\eps}^\delta];\cev{\exitTT}_{\eps}^\delta)$ converges almost surely to a limit $M_r(\cev{\bs{\Gamma}}_1[0,\cev{\exitTT}_{\eps}];\cev{\exitTT}_{\eps}) \leq 1$ as $\delta\to 0$,
and $p^\delta(\A_r^\delta;\bs{x}^\delta;\hat{\bs{x}}^\delta) \to p(\A_r;\bs{x};\hat{\bs{x}})$ 
 --- so by the dominated convergence theorem, the Radon--Nikodym derivative of
$\cev{\bs{\gamma}}^{(r)}[0,\cev{\exitTT}_{\eps}]$ with respect to $\cev{\bs{\Gamma}}^{(r)}[0,\cev{\exitTT}_{\eps}]$ is
\begin{align*}
L_r(\cev{\bs{\Gamma}}[0,\cev{\exitTT}_{\eps}]) 
:= \frac{ M_r(\cev{\bs{\Gamma}}[0,\cev{\exitTT}_{\eps}];\cev{\exitTT}_{\eps}) }{ p(\A_r;\bs{x};\hat{\bs{x}}) }
\; \one{ \Big\{\cev{\Gamma}_1[0,\cev{\exitTT}_{\eps}] \cap \bigcup_{j=2}^{k} \cev{\Gamma}_j[0,\cev{\exitTT}_{\eps}]=\emptyset \Big\} }
\; \one{ \big\{ (\cev{\Gamma}_1(0),\ldots,\cev{\Gamma}_k(0))\in\LX_k \big\} }.
\end{align*}

A computation similar to the one presented in the proof of Lemma~\ref{lem::scaling_constant} yields
\begin{align*}
\frac{ \underset{1\le i<j\le k}{\prod} \sin \Big( \frac{1}{2} \big( \arg \phi_\eps^{(r)} (\cev{\Gamma}_j(\cev{\exitTT}_{\eps}))
- \arg \phi_\eps^{(r)} (\cev{\Gamma}_i(\cev{\exitTT}_{\eps})) \big) \Big) }
{ \underset{2\le i<j\le k}{\prod} \sin \Big( \frac{1}{2} \big( \arg \hat\phi_\eps^{(r)} (\cev{\Gamma}_j(\cev{\exitTT}_{\eps}))
- \arg \hat\phi_\eps^{(r)} (\cev{\Gamma}_i(\cev{\exitTT}_{\eps})) \big) \Big)  }
\; & \; \overset{r \to 0}{\longrightarrow} \; 
\frac{ \underset{1\le i<j\le k}{\prod} \sin \Big( \frac{1}{2} \big( \arg \phi_\eps (\cev{\Gamma}_j(\cev{\exitTT}_{\eps}))
- \arg \phi_\eps (\cev{\Gamma}_i(\cev{\exitTT}_{\eps})) \big) \Big) }
{ \underset{2\le i<j\le k}{\prod} \sin \Big( \frac{1}{2} \big( \arg \hat\phi_\eps (\cev{\Gamma}_j(\cev{\exitTT}_{\eps}))
- \arg \hat\phi_\eps (\cev{\Gamma}_i(\cev{\exitTT}_{\eps})) \big) \Big)  }
\\
\underset{1\le i<j\le k}{\prod} \sin \Big( \tfrac{1}{2} \big( \arg \phi_\eps^{(r)} (x_j)
- \arg \phi_\eps^{(r)} (x_i) \big) \Big) 
\; & \; \overset{r \to 0}{\longrightarrow} \; 
\underset{1\le i<j\le k}{\prod} \sin \Big( \tfrac{1}{2} \big( \arg x_j
- \arg x_i \big) \Big) 
\\
\underset{2\le i<j\le k}{\prod} \sin \Big( \tfrac{1}{2} \big( \arg \hat\phi_\eps^{(r)} (x_j)
- \arg \hat\phi_\eps^{(r)} (x_i) \big) \Big) 
\; & \; \overset{r \to 0}{\longrightarrow} \; 
\underset{2\le i<j\le k}{\prod} \sin \Big( \tfrac{1}{2} \big( \arg x_j
- \arg x_i \big) \Big) 
\end{align*}
Combining this with Proposition~\ref{prop::hittingpoints}, Lemma~\ref{lem::deltato0odd}~\&~\ref{lem::Poisson_limit}, 
and Equation~\eqref{eqn::scaling_constant_poly} in Proposition~\ref{prop::BM_mix_Diri}
(also proven in Lemma~\ref{lem::scaling_constant} in Appendix~\ref{app::scaling_constant}), we find that
\begin{align*}
L_r(\cev{\bs{\Gamma}}[0,\cev{\exitTT}_{\eps}]) 
= \; & \frac{1}{2\pi} \, \frac{\LZ_{k-1}}{\LZ_k}\;(1+o_r(1)) \; 
\; \one{ \Big\{\cev{\Gamma}_1[0,\cev{\exitTT}_{\eps}] \cap \bigcup_{j=2}^{k} \cev{\Gamma}_j[0,\cev{\exitTT}_{\eps}]=\emptyset \Big\} }
\; \one{ \big\{ (\cev{\Gamma}_1(0),\ldots,\cev{\Gamma}_k(0))\in\LX_k \big\} }
\\
\times \; & 
\frac{(\phi_\eps'(0))^{\frac{k^2-1}{4}}}{(\hat\phi_\eps'(0))^{\frac{(k-1)^2-1}{4}}}
\; 
\prod_{j=2}^{k} \big| (\phi_\eps \circ \hat\phi_\eps^{-1})'(\hat\phi_\eps(\cev{\Gamma}_j(\cev{\exitTT}_{\eps}))) \big|
\frac{ \underset{1\le i<j\le k}{\prod} \sin \Big( \frac{1}{2} \big( \arg \phi_\eps (\cev{\Gamma}_j(\cev{\exitTT}_{\eps}))
- \arg \phi_\eps (\cev{\Gamma}_i(\cev{\exitTT}_{\eps})) \big) \Big) }
{ \underset{2\le i<j\le k}{\prod} \sin \Big( \frac{1}{2} \big( \arg \hat\phi_\eps (\cev{\Gamma}_j(\cev{\exitTT}_{\eps}))
- \arg \hat\phi_\eps (\cev{\Gamma}_i(\cev{\exitTT}_{\eps})) \big) \Big)  } .
\end{align*}
By Fatou's lemma, under the coupling $\LQ$, we have
\begin{align}\label{eqn::Fatou}
1 = \liminf_{r\to 0} \LQ \big[ L_r(\cev{\bs{\Gamma}}[0,\cev{\exitTT}_{\eps}]) \big] 
\ge \LQ \big[ L(\cev{\bs{\Gamma}}[0,\cev{\exitTT}_{\eps}]) \big] ,
\end{align}
where $L_r(\cev{\bs{\Gamma}}[0,\cev{\exitTT}_{\eps}]) \to L(\cev{\bs{\Gamma}}[0,\cev{\exitTT}_{\eps}]) \leq C$ as $r \to 0$, almost surely.

We use Lemma~\ref{lem::flowline_nsided} to identify the limit.
Let $\bs{\slecurv}^{\bs z} := (\slecurv_1^{z_1},\ldots,\slecurv_k^{z_k})$ be $k$ independent radial $\SLE_\kappa$ curves on $\D$ in the common parametrisation started at $\bs{z} = (z_1, \ldots,z_k)$. For any bounded continuous function $f$, we have
\begin{align}
\label{eqn::martingale_induction}
\;&\LQ \big[ f(\cev{\bs{\Gamma}}[0,\cev{\exitTT}_{\eps}]) \, L(\cev{\bs{\Gamma}}[0,\cev{\exitTT}_{\eps}])\big]
\\
\notag
= \;& \frac{\LZ_{k-1}}{\LZ_{k}} \int_{\UC} \E \Big[ \rho(z_1) \, f( \slecurv_1^{z_1}[0,\exitTT_{\eps}], \cev{\Gamma}_2[0,\cev{\exitTT}_{\eps}], \ldots, \cev{\Gamma}_k[0,\cev{\exitTT}_{\eps}] )
\, L( \slecurv_1^{z_1}[0,\exitTT_{\eps}], \cev{\Gamma}_2[0,\cev{\exitTT}_{\eps}], \ldots, \cev{\Gamma}_k[0,\cev{\exitTT}_{\eps}] )
\\
\notag
\;& \qquad\qquad\qquad\quad \times 
J(\cev{\Gamma}_2[0,\cev{\exitTT}_{\eps}], \ldots, \cev{\Gamma}_k[0,\cev{\exitTT}_{\eps}]; \slecurv_1^{z_1}[0,\exitTT_{\eps}])
\Big] 
\; |\ud z_1|
\qquad\qquad\qquad\qquad\;\;
\textnormal{[by~\eqref{eqn::RN_IM_2}~\&~symm.]}
\\
\notag
= \;& \frac{1}{\LZ_{k}} 
\int_{\LX_k}
\underset{2\le i<j\le k}{\prod} \sin \Big( \frac{\arg z_j - \arg z_i}{2}\Big) 
\; \E \Big[ f( \bs{\slecurv}^{\bs z}[0,\exitTT_{\eps}] ) \, S(\bs{\slecurv}^{\bs z}[0,\exitTT_{\eps}]) \Big] \;\prod_{j=1}^{k} |\ud z_j| ,
\qquad\;
\textnormal{[ind.~hypo.~\&~Lem~\ref{lem::flowline_nsided}]}
\end{align}
where using the notation from~(\ref{eqn::RN_IM_2},~\ref{eqn::RN_n_curves}), we have
\begin{align*}
S(\bs{\slecurv}^{\bs z}[0,\exitTT_{\eps}]) 
:= \rho(z_1) \, L(\bs{\slecurv}^{\bs z}[0,\exitTT_{\eps}]) \, J(\bs{\slecurv}^{\bs z}[0,\exitTT_{\eps}]) 
\, R( \D; \slecurv_2^{z_2}[0,\exitTT_{\eps}], \ldots, \slecurv_k^{z_k}[0,\exitTT_{\eps}]) .
\end{align*}
A direct computation shows that
\begin{align*}
S(\bs{\slecurv}^{\bs z}[0,\exitTT_{\eps}]) 
=\frac{ \underset{1\le i<j\le k}{\prod} \sin \Big( \tfrac{1}{2} \big( \arg z_j
- \arg z_i \big) \Big) }
{ \underset{2\le i<j\le k}{\prod} \sin \Big( \tfrac{1}{2} \big( \arg z_j
- \arg z_i \big) \Big) }
\; R( \D; \bs{\slecurv}^{\bs z}[0,\exitTT_{\eps}] ) ,
\end{align*}
where $\bs{\slecurv}^{\bs z}[0,\exitTT_{\eps}] = (\slecurv_1^{z_1}[0,\exitTT_{\eps}], \slecurv_2^{z_2}[0,\exitTT_{\eps}], \ldots, \slecurv_k^{z_k}[0,\exitTT_{\eps}])$.
Taking $f \equiv 1$, Lemma~\ref{lem::flowline_nsided} yields
\begin{align}\label{eqn::uniform_inte}
\LQ \big[ L(\cev{\bs{\Gamma}}[0,\cev{\exitTT}_{\eps}]) \big] = 1 ,
\end{align}
and we see that weighted by $L(\cev{\bs{\Gamma}}[0,\cev{\exitTT}_{\eps}])$, 
the law of $\cev{\bs{\Gamma}}[0,\cev{\exitTT}_{\eps}]$ satisfies the properties~\ref{item::scaling_curve1}~\&~\ref{item::scaling_curve2} in the statement of Proposition~\ref{prop::scaling_curve}.
Moreover, by~(\ref{eqn::Fatou},~\ref{eqn::uniform_inte}), under the coupling $\LQ$, the family $\{L_r(\cev{\bs{\Gamma}}[0,\cev{\exitTT}_{\eps}]) \}_{r>0}$ is uniformly integrable.
Hence, we deduce that the law of $\cev{\bs{\gamma}}^{(r)}[0,\cev{\exitTT}_{\eps}]$ also converges as $r \to 0$ in total variation distance: 
\begin{align*}
\lim_{r\to 0} \TVnorm{ \cev{\bs{\Gamma}}[0,\cev{\exitTT}_{\eps}] - \cev{\bs{\gamma}}^{(r)}[0,\cev{\exitTT}_{\eps}] }
= \; & \lim_{r\to 0} \sup_A \Big| \PP \big[ \cev{\bs{\Gamma}}[0,\cev{\exitTT}_{\eps}] \in A \big] -  \PP \big[ \cev{\bs{\gamma}}^{(r)}[0,\cev{\exitTT}_{\eps}] \in A \big] \Big|
\\
= \; & \lim_{r\to 0} \sup_A \Big| \LQ \big[ \cev{\bs{\Gamma}}^{(r)}[0,\cev{\exitTT}_{\eps}] \in A \big] - \LQ \big[ L_r(\cev{\bs{\Gamma}}[0,\cev{\exitTT}_{\eps}]) \, \one{\{\cev{\bs{\Gamma}}^{(r)}[0,\cev{\exitTT}_{\eps}] \in A\}} \big] \Big|
\\
\leq \; & \lim_{r\to 0} \Big| \LQ \big[ | 1 - L_r(\cev{\bs{\Gamma}}[0,\cev{\exitTT}_{\eps}]) | \big]  \Big| 
\; = \; 0 
\end{align*}
by the dominated convergence theorem and the choice of the coupling.
This finishes the induction step.
\end{proof}

We now turn to the reparameterised  discrete curves $(\eta_1^\delta,\ldots,\eta_n^\delta)$ in $\OmegaFill=\D$.  
To prove Theorem~\ref{thm::conv_curves}, it suffices to prove that their limit has the same law as the $n$-sided radial $\SLE_2$ curves $\bs{\eta}[0,\exitTT_{\eps}]$. We will, in fact, prove a more general result.

\begin{lemma}\label{lemma::together}
Fix $\kappa \in (0,4]$. 
Let $\bs{z} \in \UC^n$ be $n$ distinct points in counterclockwise order.
Fix $\eps>0$. 
Let the curves $(\eta^{\bs z}_{1,\eps},\ldots,\eta^{\bs z}_{n,\eps})$ be reparameterised by common parametrisation, to obtain $(\hat\eta^{\bs z}_{1,\eps},\ldots,\hat\eta^{\bs z}_{n,\eps})$. 
Then, up to the hitting time $\exitTT_{\eps}$ to $\partial B(0,\eps)$, the law of $(\hat\eta^{\bs z}_{1,\eps},\ldots,\hat\eta^{\bs z}_{n,\eps})$ is that of the $n$-sided radial $\SLE_\kappa$.
\end{lemma}

\begin{proof}
Let $\bs{\slecurv} = (\slecurv_1,\ldots,\slecurv_n)$ be $n$ independent radial $\SLE_\kappa$ curves from $\bs{z}$ to $0$ on $\D$, 
and $\bs{\eta} = (\eta_1,\ldots,\eta_n)$ an $n$-sided radial $\SLE_\kappa$ from $\bs{z}$ to $0$ on $\D$, 
both with the common parametrisation. 
For each $t \geq 0$, by~\cite[Propositions~3.4~\&~3.6]{Healey-Lawler:N_sided_radial_SLE} 
the Radon--Nikodym derivative of $\bs{\eta}[0,t]$ with respect to $\bs{\slecurv}[0,t]$ is given by 
the same formal expression  $R(\D;\bs{\slecurv}[0,t])$ in Equation~\eqref{eqn::RN_n_curves} in Lemma~\ref{lem::flowline_nsided}. 
The difficulty is that we cannot stop the $n$-sided SLE (run with common parametrisation) at 
a stopping time in such a way that all $n$ curves stop when they first reach $\partial B(0, \eps)$, 
in order to compare with the expression of the Radon--Nikodym derivative in that lemma and thus conclude. 
We deal with this issue as follows. 
Let 
\begin{align*}
\exitTT_{2\eps} := \inf \big\{ t>0 \colon \bs{\eta}[0, t] \cap B(0, 2\eps) \neq \emptyset \big\}
\end{align*}
be the first time when \emph{one} of the $n$ curves forming the $n$-sided radial $\SLE_\kappa$ tuple $\bs \eta$ in common parametrisation hits $\partial B( 0, 2\eps)$. 
Given these $n$ curves up to time $\exitTT_{2\eps}$, we then concatenate to them $n$ additional curves 
${\bs \eta}^{\bs w}_{\eps} = (\eta^{\bs w}_{1,\eps},\ldots,\eta^{\bs w}_{n,\eps})$ 
which have the law of the inductive construction of Definition~\ref{D:inductive} in the multi-slit domain $\D \setminus \bs{\eta}[ 0, \exitTT_{2\eps}]$ towards 0$,$ 
starting from the tips of $\bs w = \bs{\eta} (\exitTT_{2\eps})$, and stopped when they respectively hit $\partial B(0, \eps)$. 
For each $j$, we denote by $\tilde \eta_j [ 0, \tilde T_{j, \eps}]$ the concatenation of $\eta_j [ 0, \exitTT_{2\eps}]$ with $\eta^{\bs w}_{j,\eps} [ 0, T_{j,\eps}]$. 
Altogether, this defines $n$ curves run until they each hit the ball $B(0, \eps)$, which we view as curves up to time-reparametrisation. 
We will consider the Radon--Nikodym derivative 
$\ud \tilde{\bs \eta}[ \bs{0}, \tilde{\bs{T}}_\epsilon] / \ud \bs{\slecurv}[\bs{0}, \tilde{\bs{T}}_\epsilon]$. 
Note that by the domain Markov property, 
conditional on $\bs{\slecurv}[0,\exitTT_{2\eps}]$, the law of each $\slecurv_j[0,T_{j,\eps}]$ is that of 
the radial $\SLE_\kappa$ curve $\slecurv_{j,\eps}[0, T_{j,\eps}]$ in the slit domain $\D \setminus \slecurv_j[0,\exitTT_{2\eps}]$ towards $0$. 
In order to apply Lemma~\ref{lem::flowline_nsided}, we need to consider independent radial $\SLE_\kappa$ curves 
${\bs \slecurv}^{\bs w}_{\eps} := (\slecurv^{\bs w}_{1,\eps},\ldots,\slecurv^{\bs w}_{n,\eps})$ 
in the multi-slit domain $\D \setminus \bs{\slecurv}[ 0, \exitTT_{2\eps}]$ towards $0$, 
starting from the tips $\bs w = \bs{\slecurv} (\exitTT_{2\eps})$, and stopped when they respectively hit $\partial B(0, \eps)$. 
We have
\begin{align}
\frac{\ud \, \tilde{\bs \eta}[\bs{0}, \tilde{\bs{T}}_\eps] }{ \ud \, \bs{\slecurv}[\bs{0}, \tilde{\bs{T}}_\eps]}
= \; & \frac{\ud \, \bs{\eta} [ 0, \exitTT_{2\eps}]}{\ud \, \bs{\slecurv}[0,\exitTT_{2\eps}] } 
\; \frac{\ud \, {\bs \eta}^{\bs w}_{\eps}[ \bs{0}, \bs{T}_\eps]}{\ud \, {\bs \slecurv}^{\bs w}_{\eps}[\bs{0}, \bs{T}_\eps]} 
\; \prod_{j=1}^{n}\frac{\ud \, \slecurv^{\bs w}_{j,\eps} [ 0, T_{j,\eps}]}{\ud \, \slecurv_{j,\eps}[0, T_{j,\eps}]}
\nonumber 
\\
= \; & R(\D;\bs{\slecurv}[0,\exitTT_{2\eps}]) 
\; R(\D\setminus \bs{\slecurv}[0,\exitTT_{2\eps}];{\bs \slecurv}^{\bs w}_{\eps}[\bs{0}, \bs{T}_\eps]) 
\; \prod_{j=1}^{n}\frac{\ud \, \slecurv^{\bs w}_{j,\eps} [ 0, T_{j,\eps}]}{\ud \, \slecurv_{j,\eps}[0, T_{j,\eps}]}
,
\label{eq:RNtriple}
\end{align}
using~\cite[Proposition 3.4 \& 3.6]{Healey-Lawler:N_sided_radial_SLE} for the first term and Lemma~\ref{lem::flowline_nsided} for the second term.
To analyse the third term, consider the following conformal isomorphisms preserving the origin and with positive derivative there:
\begin{itemize}
\item $\psi_{\bs{\slecurv}[0,\exitTT_{2\eps}]} \colon \D\setminus \bs{\slecurv}[0,\exitTT_{2\eps}] \to \D$; 

\item $\psi_{\slecurv_j[0, \exitTT_{2\eps}]} \colon \D\setminus\slecurv_j[0,\exitTT_{2\eps}] \to \D$, for $1\le j\le n$;

\item $\psi_j \colon \D\setminus (\slecurv_j[0,\exitTT_{2\eps}]\cup \slecurv^{\bs w}_j[0,T_{j,\eps}]) \to \D$, for $1\le j\le n$;

\item $\hat \psi_j \colon \D\setminus(\bs\slecurv[0,\exitTT_{2\eps}]\cup\slecurv^{\bs w}_j[0,T_{j,\eps}]) \to \D$, for $1\le j\le n$.
\end{itemize}
Then, by definition of $R$ in Equation~\eqref{eqn::RN_n_curves}, we find that
\begin{align}
\;& \frac{R(\D;\bs{\slecurv}[0,\tilde{\bs{T}}_\eps])}{R(\D;\bs{\slecurv}[0,\exitTT_{2\eps}]) \; R(\D\setminus \bs{\slecurv}[0,\exitTT_{2\eps}];\bs{\slecurv}^{\bs w}_ {\eps}[\bs{0},\bs{T}_\eps])}
\notag\\
= \;& \one{\{I_n^\eps\}}\;
\prod_{j=1}^{n}\bigg(\frac{(\hat \psi_j\circ\psi_j^{-1})'(\psi_j(\slecurv^{\bs w}_j(T_{j,\eps})))}{( \psi_{\bs\slecurv[0,\exitTT_{2\eps}]}\circ\psi_{\slecurv_j[0,\exitTT_{2\eps}]}^{-1})'(\psi_{\slecurv_j[0,\exitTT_{2\eps}]}(\slecurv_j(\exitTT_{2\eps})))}\bigg)^{\frac{6-\kappa}{2\kappa}}\;
\prod_{j=1}^{n}\bigg(\frac{\hat \psi_j'(0)/\psi_j'(0)}{\psi'_{\bs\slecurv[0,\exitTT_{2\eps}]}(0)/\psi_{\slecurv_j[0,\exitTT_{2\eps}]}'(0)}\bigg)^{\frac{(6-\kappa)(\kappa-2)}{8\kappa}}
\notag\\
\;&\times\prod_{j=1}^{n}\exp \bigg( \frac{(3\kappa-8)(6-\kappa)}{4\kappa} \; 
\mu_{\D\setminus \bs\slecurv[0,\exitTT_{2\eps}]} \Big[\ell \cap \slecurv_j^{\bs w}[0, T_{j,\eps}] \neq\emptyset \textnormal{ and } \ell \cap \bs\slecurv[0,\exitTT_{2\eps}] \neq\emptyset \Big] \bigg).
\label{eqn::Raux_0}
\end{align}
Thus, by \cite[Proposition 2.2]{Healey-Lawler:N_sided_radial_SLE} and~\eqref{eqn::Raux_0}, we have
\begin{align*}
R(\D;\bs{\slecurv}[0,\exitTT_{2\eps}]) 
\; R(\D\setminus \bs{\slecurv}[0,\exitTT_{2\eps}];{\bs \slecurv}^{\bs w}_{\eps}[\bs{0}, \bs{T}_\eps]) 
\; \prod_{j=1}^{n}\frac{\ud \, \slecurv^{\bs w}_{j,\eps} [ 0, T_{j,\eps}]}{\ud \, \slecurv_{j,\eps}[0, T_{j,\eps}]}
= R(\D;\bs{\slecurv}[0,\tilde{\bs{T}}_\eps]) .
\end{align*}
Combining with~\eqref{eq:RNtriple}, we deduce that 
\begin{align*}
\frac{\ud \, \tilde{\bs \eta}[\bs{0}, \tilde{\bs{T}}_\eps] }{ \ud \, \bs{\slecurv}[\bs{0}, \tilde{\bs{T}}_\eps]}
= R(\D; \bs{\slecurv}[\bs{0}, \tilde{\bs{T}}_\eps]) 
= \frac{\ud \, {\bs \eta}^{\bs z} [ \bs{0}, \tilde{\bs{T}}_\eps] }{ \ud \bs{\slecurv}[\bs{0}, \tilde{\bs{T}}_\eps]} ,
\end{align*}
which implies that $\tilde{\bs{\eta}}[ \bs{0}, \tilde{\bs{T}}_\eps]$ has the same law as ${\bs \eta}^{\bs z} [ \bs{0}, \tilde{\bs{T}}_\eps]$. 
By construction, considering $\tilde{\bs{\eta}}$ with the common parametrisation, we see that the law of $(\hat\eta^{\bs z}_{1,\eps},\ldots,\hat\eta^{\bs z}_{n,\eps})$ is that of the $n$-sided radial $\SLE_\kappa$, as desired.
\end{proof}

Now, we finish the proof of Theorem~\ref{thm::conv_curves}.
\begin{proof}[Proof of Theorem~\ref{thm::conv_curves}]
Lemma~\ref{lem::scaling_one_step} gives the tightness of $\{(\cev{\gamma}_1^{\delta},\ldots,\cev{\gamma}_n^{\delta})\}_{\delta>0}$. The convergence after shrinking $\BdryIn$ to a point is given by Proposition~\ref{prop::scaling_curve} and Lemma~\ref{lemma::together}.
\end{proof}

\subsection{Variance of the winding: proof of Proposition~\ref{coro::truncated_winding}}
\label{subsec:truncated_winding}

Consider the circular Dyson Brownian motion with parameters $\kappa > 0$ and $\beta \geq 1$, 
that is, the unique strong solution $(\Theta_1,\ldots,\Theta_n)$ to the SDE system 
\begin{align}\label{eqn:DBM_gen}
\ud\Theta_i(t) = \sqrt \kappa \, \ud W_i(t) + \frac{\beta \kappa}{4}\sum_{j\neq i} \cot\bigg(\frac{\Theta_i(t) - \Theta_j(t)}{2}\bigg)\ud t, \qquad 1 \leq i \leq n ,
\end{align}
where $(W_i(t) \colon t\ge 0)$ for $1\le i\le n$ are $n$ independent Brownian motions on $\R$. 
(The existence of a unique strong solution for any $\beta \geq 1$ follows from the analogous result for Dyson Brownian motion~\cite{AGZ:An_introduction_to_random_matrices}). 
First, we consider the asymptotics of the variance of this Dyson Brownian motion.

\begin{lemma}\label{lem::Dyson_driving}
For any $\kappa > 0$ and $\beta \geq 1$, we have 
\begin{align}\label{eqn::varthetai}
\lim_{t\to\infty} \frac{1}{t} \, \mathrm{Var}\big[\Theta_i(t)\big] = \frac{\kappa}{n} , \qquad 1 \leq i \leq n.
\end{align}
\end{lemma}

\begin{proof}
Let us first consider the case of $i=1$. 
The Dyson SDEs~\eqref{eqn:DBM_gen} yield
\begin{align}\label{eqn::SDEalpha}
\Sigma(t) := \; & \sum_{j=1}^{n} \Theta_j(t) 
&& \Longrightarrow \qquad 
\ud \Sigma(t) = \sum_{j=1}^{n} \ud  \Theta_j(t) 
= \sqrt \kappa \, \sum_{j=1}^{n} \ud W_j(t) ,
\\
\nonumber
\Delta_j(t) := \; & \Theta_j(t) - \Theta_1(t) 
&& \Longrightarrow \qquad 
\ud \Delta_j(t) = \sqrt \kappa \, \ud ( W_j(t) - W_1(t)) 
\\
\nonumber
& && \qquad\qquad\qquad\qquad
+ \frac{\beta \kappa}{4} \sum_{\substack{2 \le l \le n \\ l \neq j}}
\bigg( \cot\bigg(\frac{\Delta_j(t) - \Delta_l(t)}{2}\bigg)\ud t + \cot\bigg(\frac{\Delta_l(t) }{2}\bigg)\bigg)\ud t
\end{align}
for $2 \leq j \leq n$. Since $W_1,\ldots,W_n$ are $n$ independent Brownian motions, we see that $\Sigma$ and
$\Delta_2, \ldots, \Delta_n$ are independent.  In particular, we have
\begin{align}\label{eqn::varalpha}
\mathrm{Var} \Big[ \Sigma(t) - \sum_{j=2}^{n} \Delta_j(t) \Big]
= \mathrm{Var} [ \Sigma(t) ] + \mathrm{Var} \Big[ \sum_{j=2}^{n} \Delta_j(t) \Big]
= \kappa n t + \mathrm{Var} \Big[ \sum_{j=2}^{n} \Delta_j(t) \Big] ,
\end{align}
where $\mathrm{Var}[\Sigma(t)] = \kappa n t$ by~\eqref{eqn::SDEalpha}.
On the one hand, note that $\Delta_j(t) = \Theta_j(t) - \Theta_1(t) \in [0,2\pi]$ for every $t\ge 0$ and for $2 \leq j \leq n$. 
Hence, there exists a constant $M > 0$ such that
\begin{align}\label{eqn::varbeta}
\mathrm{Var} \Big[ \sum_{j=2}^{n} \Delta_j(t) \Big] \le Mn^2.
\end{align}
On the other hand, note that 
\begin{align*}
\Theta_1(t) = \frac{1}{n} \Big(\Sigma(t) - \sum_{j=2}^{n} \Delta_j(t) \Big) .
\end{align*}
Combining this with~(\ref{eqn::varalpha},~\ref{eqn::varbeta}), we obtain
\begin{align}\label{eqn::vartheta1}
\lim_{t\to\infty} \frac{1}{t} \, \mathrm{Var}\big[\Theta_i(t)\big] = \frac{ \kappa }{n}.
\end{align}
Lastly, as $\Theta_j(t) = \Delta_j(t) + \Theta_1(t)$ for $2 \leq j \leq n$, Equation~\eqref{eqn::vartheta1} also yields~\eqref{eqn::varthetai} for $j \neq 1$.
\end{proof}

For each $n \geq 1$, we sample the starting points $(\Theta_1(0),\ldots,\Theta_n(0)) = (\theta_1,\ldots,\theta_n)$ from the density
\begin{align*}
\rho(\theta_1,\ldots,\theta_n) = \frac{1}{\mathcal{Z}_n}\prod_{1\le i\le j\le n}\sin\bigg(\frac{\theta_j-\theta_i}{2}\bigg) .
\end{align*}
Then, if $(\Theta_1(t),\ldots,\Theta_n(t))$ for $t\ge 0$ evolves according to~\eqref{eqn:DBM_gen}, then 
$\big((\ee^{\ii\Theta_1(t)},\ldots,\ee^{\ii\Theta_n(t)}) \colon t\ge 0\big)$ is a stationary Markov process on $\LX_n$. 
Now, we are ready to prove Proposition~\ref{coro::truncated_winding}.
The proof is very similar to the proof of~\cite[Equation~(7.3)]{Schramm:Scaling_limits_of_LERW_and_UST}.
We now take $\kappa=2$ and $\beta=4$.

\begin{proof}[Proof of Proposition~\ref{coro::truncated_winding}]
Let $\bs\eta$ be a $n$-sided radial $\SLE_2$, whose Loewner driving process as in~\eqref{eqn::ODE} 
is given by the circular Dyson Brownian motion with parameters $\kappa=2$ and $\beta=4$ in~\eqref{eqn:DBM_gen}.
Fix $T > 0$ and for $0\le t\le T$, define $h_t(w):=g_{T-t}(g_T^{-1}(w))$ for $w\in\UC$, 
so that $h_t(\ee^{\ii W_j(T)}) :=g_{T-t}(\eta_j(T))$, for $1 \leq j \leq n$.
We can write $h_t(\ee^{\ii W_j(T)}) = \exp ( X_j(t) + \ii Y_j(t) )$, where 
$(X_j(t) \colon 0<t\le T)$ and $(Y_j(t) \colon 0<t\le T)$ satisfy 
\begin{align}\label{eqn::log_aux}
\partial_t X_j(t) + \ii\partial_t Y_j(t) 
= \sum_{i=1}^{n} \frac{\sinh ( X_j(t) ) + \ii \sin \big(\Theta_i(T-t) - Y_j(t) \big)}{\cosh ( X_j(t) ) - \cos \big(\Theta_i(T-t) - Y_j(t) \big)} , \qquad 1 \leq j \leq n.
\end{align}
Like in the proof of~\cite[Theorem~7.2]{Schramm:Scaling_limits_of_LERW_and_UST}, 
because $\partial_t X_j(t)< 0$ for $0<t\le T$, we can write $Y_j$  as a function of $X_j$, say $Y_j=F(X_j)$.  By~\eqref{eqn::log_aux}, we have
\begin{align*}
|F'(X_j(t))|\le \frac{n}{|\sinh X_j(t)|} .
\end{align*}
Now, let $T_j$ to be the first time when $X_j = -1$. 
Then, as in~\cite[Equation~(7.7)]{Schramm:Scaling_limits_of_LERW_and_UST}, we find that
\begin{align}\label{eqn::Y_bound1}
|Y_j(T) - Y_j(T_j)| \lesssim \int_{-\infty}^{-1}\frac{ \ud x }{|\sinh x|}< \infty.
\end{align}
Now, set $\psi_j(s,t)=g_{T-t}(\eta_j(T-s))$ for $0\le s\le t\le T$. Note that the image of $\psi_j$ does not contain $0$ and $\psi_j$ is a homotopy between $(\psi_j(s,T) \colon 0\le s\le T)$ and $(\psi_j(0,T-t) \colon 0\le t\le T)\cup (\psi_j(t,t) \colon 0\le t\le T)$. 
Thus, we can compute the topological winding $\ph_j(T)$ of $\eta_j[0,T]$ around the origin as
\begin{align*}
\ph_j(T) := \arg \eta_j(T) - \arg \eta_j(0)
= \Theta_j(T) - \Theta_j(0) + Y_j(T) - Y_j(T_j) + Y_j(T_j) - Y_j(0) .
\end{align*}
From~\eqref{eqn::Y_bound1} and the proof of Lemma~\ref{lem::Dyson_driving}, 
if $f$ is a bounded continuous test function, then we have
\begin{align*}
\lim_{T\to\infty} \E \Big[ f\Big(\frac{\ph_1(T)}{\sqrt{2T/n}},\ldots, \frac{\ph_n(T)}{\sqrt{2T/n}}\Big) \Big]
= \lim_{T\to\infty} \E \Big[ f\Big(\frac{\Sigma(T)}{\sqrt{2nT}},\ldots, \frac{\Sigma(T)}{\sqrt{2nT}}\Big) \Big]
= \lim_{T\to\infty} \E \big[ f(X,\ldots,X) \big] ,
\end{align*}
where $\Sigma(t) := \sum_{j=1}^{n} \Theta_j(t)$ as in~\eqref{eqn::SDEalpha} in the proof of Lemma~\ref{lem::Dyson_driving}, 
and where $X\sim \LN (0,1)$. 
\end{proof}

\section{Flow-line coupling for $n$-sided radial SLE with GFF}
\label{sec::GFF}

Lastly, we derive a connection of the $n$-sided radial $\SLE_\kappa$, for $\kappa \in (0,4)$,
with a variant $\field$ of the Gaussian free field (GFF) having a singularity at the origin with additive monodromy.
Here, $\bs{\eta}$ represent ``flow lines'' of the ``ill-defined'' vector field of type $\ee^{\ii \field}$.
The intuition from the dimer height function goes as follows (cf.~Figure~\ref{F:dimers}).
While the discrete dimer height function $\gff^\delta$ does define flow lines $\fline^\delta \colon [0,t] \to \C$ 
by solving the related flow ODE $\frac{\ud}{\ud t} \fline^\delta(t) = \exp\big(\ii \gff^\delta(\fline^\delta(t))\big)$, 
its scaling limit --- the GFF --- is not a function but rather a random Schwartz distribution.
The theory of flow lines for the GFF was extensively developed by Miller and Sheffield~\cite{Miller-Sheffield:Imaginary_geometry1, Miller-Sheffield:Imaginary_geometry4}, and earlier by Dub\'edat~\cite{Dubedat:SLE_and_free_field}.
The series of works~\cite{BLR:Dimers_and_imaginary_geometry, BLR:Dimers_on_Riemann_surfaces_2, BLR:Dimers_on_Riemann_surfaces_1} 
addresses the case of present interest from the point of view of the dimer height function.  
We refer to these works for more background and details, and keep this section very concise.

\subsection{Gaussian free field}

Suppose $D\subset\C$ is a regular domain (in the sense, say of~\cite{Berestycki-Powell:Gaussian_free_field_and_Liouville_quantum_gravity}, so that the Dirichlet Green's function is finite). 
The \textbf{zero-boundary GFF} $\gff$ on $D$ is a random distribution such that for every smooth compactly supported test function $f\in C^{\infty}_c(D)$ on $D$, 
the quantity $(\gff,f)$ is a Gaussian random variable with mean zero and variance obtained from the Green's function $\Gren_D$ on $D$ with Dirichlet boundary condition: 
\begin{align*}
\E[(\gff,f)]=0 
\qquad \textnormal{and} \qquad 
\E[(\gff,f)^2] = \int_D \int_D f(z) \, \Gren_D (g_t(z),g_t(w)) \, f(w) \; \ud z \, \ud w .
\end{align*}
Here, we normalise the Dirichlet Green's function so that $\Gren_D(z,w)= - \log |z-w| + O(1)$ as $|z-w| \to 0$. 
The GFF with boundary data $u$ is the sum of the zero-boundary GFF on $D$ and harmonic extension of $u$ to $D$. 
See~\cite{Berestycki-Powell:Gaussian_free_field_and_Liouville_quantum_gravity} for a more detailed introduction.

\medskip
In~\cite{Miller-Sheffield:Imaginary_geometry4}, the authors establish the theory of \textbf{flow lines} for the Gaussian free field with a singularity at one interior point. 
The procedure goes roughly as follows. The GFF can be coupled with a radial $\SLE_\kappa(\rho)$ process by taking the field to be $\gff + \alpha\arg(\cdot)$ on $\D$, where $\alpha=\alpha(\rho) \in \R$ is the strength of the singularity at the origin 
and one assumes certain boundary data on $\partial\D$ depending on $\kappa$~\cite[Proposition 3.1]{Miller-Sheffield:Imaginary_geometry4}.
In the present article, we consider a similar setting: 
we will derive a coupling of the $n$-sided $\SLE_\kappa$ curves $\bs{\eta} = (\eta_1, \ldots, \eta_n)$ 
with a Gaussian free field having suitable boundary data and singularity at the origin, yielding Theorem~\ref{thm::coupling}.
Although the value of $\kappa=2$ is the most natural in the context of the uniform spanning tree,
with the same effort we may consider arbitrary $\kappa \in (0,4)$.
The (critical) case of $\kappa=4$ would correspond to a level-line with monodromy~\cite{Dubedat:SLE_and_free_field}, which we shall not consider here.

\subsection{GFF flow lines with additive monodromy: proof of Theorem~\ref{thm::coupling}}

Let us first recall the ingredients of the claim in Theorem~\ref{thm::coupling}. 
Fix $\kappa \in (0, 4)$ and $\ee^{\ii\bs{\theta}}\in\LX_n$, and set
\begin{align*}
\chi_\kappa:=\frac{2}{\sqrt\kappa}-\frac{\sqrt\kappa}{2} 
\qquad \textnormal{and} \qquad
\lambda_\kappa:=\frac{\pi}{\sqrt\kappa} .
\end{align*}
Consider the harmonic function $\harm \colon \D \to \R$ defined for $w \in \D$ as in~\eqref{eqn::bound}:
\begin{align*}
\harm(w) = \harm_{\bs{\theta}}(w) 
:= -\frac{1}{\sqrt \kappa} \sum_{j=1}^n\arg \bigg(\frac{\ee^{\ii\theta_j}-w}{1-\overline{w}\ee^{\ii\theta_j}}\bigg) 
= -\frac{2}{\sqrt\kappa} \sum_{j=1}^n\arg(\ee^{\ii\theta_j}-w) + \frac{1}{\sqrt\kappa}\sum_{j=1}^n\theta_j.
\end{align*}
Let $\gff$ be a GFF field with Dirichlet boundary conditions in $\D$, and define the field as in~\eqref{eq:field}, 
\begin{align*}
\field(\cdot) := \gff(\cdot) + \Big(\chi_\kappa+\frac{n}{\sqrt\kappa}\Big)\arg(\cdot)+\harm_{\bs{\theta}}(\cdot) ,
\end{align*}
having additive monodromy at the origin due to the multivalued function $\arg(\cdot)$. 
We consider flow lines $\bs{\fline} = (\fline_1,\ldots,\fline_n)$ of the formal vector field of type $\ee^{\ii \field}$, starting from $\ee^{\ii\bs{\theta}}$ and with respective angles
\begin{align*}
\alpha_j = \frac{2\lambda_\kappa (j-1)}{\chi_\kappa} , \qquad 1 \leq j \leq n .
\end{align*}
These are random curves in the space $X_0^n$ (with metric~\eqref{eqn::curve_metric_many}) coupled with $\field$ so that for all stopping times $\bs{\tau} = (\tau_1, \ldots, \tau_n)$,  
conditionally given $\bs{\fline}[0, \bs{\tau}] = (\fline_1[0, \tau_1], \ldots, \fline_n[0, \tau_n])$, we have
\begin{align*}
\field|_{\D_{\bs{\tau}}} = \tilde \field \circ g_{\bs{\tau}} - \chi_\kappa \arg g'_{\bs{\tau}},  
\end{align*}
where $\tilde \field$ has the same law as~\eqref{eq:field}, and $g_{\bs{\tau}} \colon \D_{\bs{\tau}} \to \D$ is the unique conformal isomorphism uniformising the complement of $\D_{\bs{\tau}} := \D \setminus \bs{\fline}[0, \bs{\tau}]$ 
normalised such that $g_{\bs{\tau}}(0) = 0$ and $g'_{\bs{\tau}} (0) >0$.

\medskip 
We claim that, after parameterising $\bs{\fline}$ with the common (capacity) parametrisation, their law agrees with that of the $n$-sided radial $\SLE_\kappa$ curves $\bs{\eta}$,
whose Loewner driving process is the circular Dyson Brownian motion with parameters $\kappa=2$ and $\beta=4$ as in~\eqref{eqn:DBM}. 
In essence, this is encapsulated in the asserted Equation~\eqref{eq:IG} in Theorem~\ref{thm::coupling}, which we now proceed to prove.

\begin{proof}[Proof of Theorem~\ref{thm::coupling}]
The strategy is same as the proof of~\cite[Theorem~1.1]{Miller-Sheffield:Imaginary_geometry1} and~\cite[Proposition~3.1]{Miller-Sheffield:Imaginary_geometry4}. Fix $\eps>0$ and write $\bs z = \ee^{\ii\bs{\theta}}$. 
As in Lemma~\ref{lemma::together}, the law of the curves $(\hat\eta^{\bs z}_{1,\eps},\ldots,\hat\eta^{\bs z}_{n,\eps})$ with common parametrisation is that of the $n$-sided radial $\SLE_\kappa$ curves $(\eta_1,\ldots,\eta_n)$, 
up to the hitting time $\exitTT_{\eps}$ to $\partial B(0,\eps)$. 
Let $(\eta^{\bs z}_{1,\eps},\ldots,\eta^{\bs z}_{n,\eps})$ be the curves related to $(\hat\eta^{\bs z}_{1,\eps},\ldots,\hat\eta^{\bs z}_{n,\eps})$ as in Lemma~\ref{lemma::together}, 
let $T_{1,\eps} = T < \exitTT_{\eps}$ be a stopping time for the first curve $\eta^{\bs z}_{1,\eps}$, and let $T_{j,\eps}=0$ for all $2\le j\le n$. 
Let $(\Theta(t) \colon 0\le t\le T)$ be the Loewner driving function of $\eta^{\bs z}_{1,\eps}$, 
and $(g_t \colon 0\le t\le T)$ the associated radial Loewner flow~\eqref{eqn::ODE1}.
Also, let $\ee^{\ii V_{j}(t)} = g_t(\ee^{\ii\theta_j})$ be the time-evolutions of the force points of $\eta^{\bs z}_{1,\eps}$, for $2\le j\le n$.

By considering a partition of unity, for the GFF we only need to consider test functions $f$ whose support is simply connected and does not contain $0$.
Let $\Gren_\D(\cdot,\cdot)$ denote the Green's function on $\D$ with Dirichlet boundary condition. 
For fixed $a\in\R$, consider the process
\begin{align*}
M(t) 
:= \; & \E\bigg[ \exp \Big( \ii a 
\Big( \gff \circ g_t (\cdot) + \big(\chi_\kappa+\tfrac{n}{\sqrt\kappa}\big) \arg g_t(\cdot) + \harm_{\Theta(t),V_2(t),\ldots,V_n(t)}(g_t(\cdot)) - \chi_\kappa \arg g'_t(\cdot)   
\Big) , \; f(\cdot) \Big) \Bigcond \eta_1[0,t] \bigg]
\\
= \; & \exp \Big( \ii a
\Big( \big(\chi_\kappa+\tfrac{n}{\sqrt\kappa}\big) \arg g_t(\cdot) + \harm_{\Theta(t),V_2(t),\ldots,V_n(t)}(g_t(\cdot)) - \chi_\kappa \arg g'_t(\cdot)   
\Big) , \; f(\cdot)  \Big) \\
\; & \times \exp \Big( \! - \frac{a^2}{2} \int_\D\int_\D f(z) \, \Gren_\D (g_t(z),g_t(w)) \, f(w) \; \ud z \, \ud w
\Big) , \qquad 0\le t\le T .
\end{align*}
To establish the coupling, it suffices to show that $(M(t \wedge T) \colon t\ge 0)$ is a martingale. 
By It\^o's formula and Fubini's theorem, this is equivalent to verifying that for all $z,w\in\D$, 
the process $(H(t \wedge T) \colon t\ge 0)$, 
\begin{align*}
H(t) := 
\big(\chi_\kappa+\tfrac{n}{\sqrt\kappa}\big) \arg g_t(z)
+ \harm_{\Theta(t),V_2(t),\ldots,V_n(t)}(g_t(z)) - \chi_\kappa \arg g'_t(z) , \qquad 0\le t\le T ,
\end{align*}
is a martingale, and the quadratic variation of the harmonic function agrees with the Hadamard formula
\begin{align}\label{eqn::drift2}
\ud \big\langle \harm_{\Theta(t),V_2(t),\ldots,V_n(t)}(g_t(z)) , \; \harm_{\Theta(t),V_2(t),\ldots,V_n(t)}(g_t(w)) \big\rangle 
= - \ud \Gren_\D (g_t(z),g_t(w)) .
\end{align}
Let $W$ be the Brownian motion in the SDE~\eqref{eqn::SLE_kappa_rho} for $\Theta$.
By a direct computation, we find 
\begin{align}\label{eqn::11}
\;&\ud \arg\big(\ee^{\ii \Theta(t)}-g_t(z)\big) 
= \sqrt\kappa\;\Re \bigg(\frac{\ee^{\ii \Theta(t)}}{\ee^{\ii \Theta(t)}-g_t(z)} \bigg) \ud W(t)
\notag\\
+ \;& \bigg(\Re\bigg(\frac{\ee^{\ii \Theta(t)}}{\ee^{\ii \Theta(t)}-g_t(z)}\sum_{j=2}^n\cot\bigg(\frac{\Theta(t)-V_{j}(t)}{2}\bigg)\bigg)
- \Im\bigg(\frac{g_t(z)}{\ee^{\ii \Theta(t)}-g_t(z)}\bigg)^2 \!\!
+ \bigg(\frac{\kappa}{2}-1\bigg)\Im\bigg(\frac{\ee^{\ii \Theta(t)}g_t(z)}{(\ee^{\ii \Theta(t)}-g_t(z))^2}\bigg)\bigg)\ud t 
\notag\\
=\;&\sqrt\kappa\;\Re \bigg(\frac{\ee^{\ii \Theta(t)}}{\ee^{\ii \Theta(t)}-g_t(z)} \bigg) \ud W(t)
-\bigg(\sum_{j=2}^n\Im\bigg(\frac{\ee^{\ii \Theta(t)}(\ee^{\ii \Theta(t)}+\ee^{\ii V_{j}(t)})}{(\ee^{\ii \Theta(t)}-g_t(z))(\ee^{\ii \Theta(t)}-\ee^{\ii V_{j}(t)})} \bigg)
\\
\;&  \qquad\qquad\qquad\qquad\qquad\qquad\qquad\qquad
+ \Im\bigg(\frac{g_t(z)}{\ee^{\ii \Theta(t)}-g_t(z)}\bigg)^2 \!\!
- \bigg(\frac{\kappa}{2}-1\bigg)\Im\bigg(\frac{\ee^{\ii \Theta(t)}g_t(z)}{(\ee^{\ii \Theta(t)}-g_t(z))^2}\bigg)\bigg)\ud t ,
\notag
\end{align}
and for $2\le j\le n$, 
\begin{align}\label{eqn::12}
\ud \arg\big(\ee^{\ii V_{j}(t)}-g_t(z)\big) 
= 2 \, \Im \bigg( \frac{e^{2\ii\Theta(t)}}{(\ee^{\ii \Theta(t)}-g_t(z))(\ee^{\ii \Theta(t)}-\ee^{\ii V_{j}(t)})} \bigg) \ud t .
\end{align}
Moreover, from the Loewner equation~\eqref{eqn::ODE1} we have
\begin{align}
\notag
\ud \arg g_t(z) = \;&\Im \bigg( \frac{2\ee^{\ii \Theta(t)}}{\ee^{\ii \Theta(t)}-g_t(z)}\bigg) \ud t ,
\\ 
\notag
\ud \arg g'_t(z) = \;& \bigg(\Im\bigg( \frac{2\ee^{\ii \Theta(t)}}{\ee^{\ii \Theta(t)}-g_t(z)} \bigg) 
+ \Im \bigg( \frac{2\ee^{\ii \Theta(t)}g_t(z)}{(\ee^{\ii \Theta(t)}-g_t(z))^2}\bigg)\bigg) \ud t ,
\\
\label{eqn::15}
\ud \Big( \Theta(t) + \sum_{j=2}^n V_{j}(t) \Big) = \;&\sqrt\kappa \, \ud W(t) .
\end{align}
Combining~\eqref{eqn::11}--\eqref{eqn::15}, we obtain
\begin{align}\label{eqn::16}
\ud H(t) = \bigg(\!\! -2\Re\bigg(\frac{\ee^{\ii \Theta(t)}}{\ee^{\ii \Theta(t)}-g_t(z)}\bigg) + 1\bigg)  \ud W(t) ,
\end{align}
which proves that the process $(H(t \wedge T) \colon t\ge 0)$ is a martingale.
To get the identity~\eqref{eqn::drift2}, we can use the explicit formula
\begin{align*}
\Gren_\D (g_t(z),g_t(w)) 
 = -\log\bigg|\frac{g_t(z)-g_t(w)}{1-\overline{g_t(z)}g_t(w)}\bigg| ,
\end{align*}
to obtain
\begin{align*}
\ud \Gren_\D (g_t(z),g_t(w)) 
= 1 - 2 \, \Re\bigg(\frac{\ee^{2\ii \Theta(t)}}{(\ee^{\ii \Theta(t)}-g_t(z))(\ee^{\ii \Theta(t)}-g_t(w))}+\frac{\overline{g_t(z)}g_t(w)}{(e^{-i\Theta(t)}-\overline{g_t(z)})(\ee^{\ii \Theta(t)}-g_t(w))}\bigg) .
\end{align*}
Writing $a:=\frac{\ee^{\ii \Theta(t)}}{(\ee^{\ii \Theta(t)}-g_t(z))}$ and $b:=\frac{\ee^{\ii \Theta(t)}}{(\ee^{\ii \Theta(t)}-g_t(w))}$, we see that
\begin{align*}
\ud \Gren_\D (g_t(z),g_t(w)) 
= 1 - 2 \, \Re \big(ab+\overline{(1-a)}(1-b) \big) 
= 1 - 4 \, (\Re \, a)(\Re \, b) + 2 \, \Re \, a + 2 \, \Re \, b ,
\end{align*}
which, combined with~\eqref{eqn::16}, implies~\eqref{eqn::drift2} and completes the proof.
\end{proof}

\appendix
\section{Brownian loop identity via martingale argument}
\label{app::scaling_constant}

The below Lemma~\ref{lem::scaling_constant} is a re-statement of Equation~\eqref{eqn::scaling_constant_poly} in Proposition~\ref{prop::BM_mix_Diri}, which we prove here using scaling limit results and a martingale argument.

\begin{lemma}\label{lem::scaling_constant}
Suppose $\simpleCurv$ is a simple curve on $\OmegaFillCl$ which only intersects $\BdryOut$ at its starting point and does not hit $0$. Then, we have 
\begin{align} \label{eqn::scaling_constant_poly2}
\lim_{\diam (\BdryIn)\to 0} \exp \Big( 2\mu_{\Omega}^{\textnormal{ND}} \big[\ell \in \LL_* \textnormal{ such that } \tfrac{1}{2\pi}\ph(\ell)\textnormal{ is odd and }\ell\cap\simpleCurv\neq\emptyset \big] \Big) 
= \bigg|\frac{\phi_\simpleCurv'(0)}{\phi'(0)}\bigg|^{1/4} ,
\end{align}
where $\phi \colon \OmegaFill \to \D$ and $\phi_\simpleCurv \colon \OmegaFill\setminus\simpleCurv \to \D$ are any conformal isomorphisms fixing the origin. 
\end{lemma}

\begin{proof}
By Equation~\eqref{eqn::conv_loop_odd_aux} and the conformal invariance of Brownian loop measure, it suffices to prove~\eqref{eqn::scaling_constant_poly2} when $\Omega=\A_r$. 
In this case, we can take $\phi$ to be the identity map and $c = c(\simpleCurv) := \phi_\simpleCurv'(0)>1$. 
Note that the right-hand side of~\eqref{eqn::conv_loop_odd_aux} only depends on the distortion factor $c$: 
\begin{align}\label{eq:Fc}
\; & \mu_{\D} \big[ \ell\in \LL_* \textnormal{ such that }  \tfrac{1}{2\pi}\ph(\ell)\textnormal{ is odd and }\ell\cap\simpleCurv\neq\emptyset \big]
\\
\nonumber
= \; & \lim_{r\to 0} \bigg( 
\mu_{\A_r}^{\textnormal{ND}} \big[ \ell\in \LL_* \textnormal{ such that }  \tfrac{1}{2\pi}\ph(\ell)\textnormal{ is odd } \big]
\, - \, 
\mu_{\A_r\setminus\simpleCurv}^{\textnormal{ND}} \big[ \ell\in \LL_* \textnormal{ such that }  \tfrac{1}{2\pi}\ph(\ell)\textnormal{ is odd } \big]
\bigg)
\\
\nonumber
= \; & \lim_{r\to 0} \bigg( 
\mu_{\A_r}^{\textnormal{ND}} \big[ \ell\in \LL_* \textnormal{ such that }  \tfrac{1}{2\pi}\ph(\ell)\textnormal{ is odd } \big]
\, - \, 
\mu_{\A_{cr}}^{\textnormal{ND}} \big[ \ell\in \LL_* \textnormal{ such that }  \tfrac{1}{2\pi}\ph(\ell)\textnormal{ is odd } \big]
\bigg)
&& \textnormal{[conf.~inv.]}
\\
\nonumber
= \; & \mu_{\A_{1/c}}^{\textnormal{ND}} \big[ \ell\in \LL_* \textnormal{ such that }  \tfrac{1}{2\pi}\ph(\ell)\textnormal{ is odd } \big]
\, =: \, F(c) .
&& \textnormal{[conf.~inv.]}
\end{align}

As in the proof of Proposition~\ref{prop::scaling_curve} when $n=1$, 
we will denote by $\Gamma_{w}^\delta$ the loop-erasure of a random walk $\LR\sim \Prob_{w}$ on $\A_r^\delta$ starting from $w^\delta\in\A_r^\delta$ and ending at $\AnnBdryOut_r^\delta$, reflected at $\AnnBdryIn_r^\delta$; 
and by $\cev{\Gamma}_{w}^\delta$ its time-reversal. 
Next, we will use a martingale technique to obtain the explicit formula~\eqref{eqn::scaling_constant_poly2}. 
Consider the case $n=2$ and take $\A_r^\delta:=\A_r\cap\delta\Z^2$. 
We take $\alpha_2^\delta$ to be an edge $e^\delta$ and assume that it converges to $e\in\partial\D$ as $\delta\to 0$; 
and we take $x_1=r$ and $x_2=-r$. 
Recall the definition of $E^\delta_{\bs{x}, \bs{\alpha}} = E(\A_r^\delta;x_1^\delta,x_2^\delta;\alpha_1^\delta,\alpha_2^\delta)$ in Equation~\eqref{eqn::defE3} for the UST branches. 
For every discrete simple curve $\simpleCurv^\delta$ which intersects $\alpha_1^\delta$, we have
\begin{align*}
\;&\PP^\delta \Big[ \cev{\gamma}_1^\delta[0,t\wedge \cev{\exitTT}_{\eps}^\delta]=\simpleCurv^\delta \Bigcond E(\A_r^\delta;x_1^\delta,x_2^\delta;\alpha_1^\delta,\alpha_2^\delta) \Big] 
\\
=\;&\frac{\PP^\delta \big[E(\A_r^\delta\cup\simpleCurv^\delta;x_1^\delta,x_2^\delta;\simpleCurv^\delta(t\wedge \exitTT_{\eps}^\delta),\alpha_2^\delta)\big]}{\PP^\delta \big[E(\A_r^\delta;x_1^\delta,x_2^\delta;\alpha_1^\delta,\alpha_2^\delta)\big]}
\; \frac{\#\{\textnormal{spanning trees $\LT^\delta$ on }\A_r^\delta \textnormal{ s.t. } \simpleCurv^\delta\subset\LT^\delta\}}{\#\{\textnormal{spanning trees $\LT^\delta$ on }\A_r^\delta\}}\\
=\;&\frac{\PP^\delta \big[E(\A_r^\delta\cup\simpleCurv^\delta;x_1^\delta,x_2^\delta;\simpleCurv^\delta(t\wedge \exitTT_{\eps}^\delta),\alpha_2^\delta)\big]}{\PP^\delta \big[E(\A_r^\delta;x_1^\delta,x_2^\delta;\alpha_1^\delta,\alpha_2^\delta)\big]}
\;\frac{\PP^\delta \big[ \cev{\Gamma}_{x_1}^\delta[0,t\wedge \cev{\exitTT}_{\eps}^\delta]=\simpleCurv^\delta \big]}{\Prob_{x_1}[\LR\textnormal{ hits }\AnnBdryOut_r^\delta\cup\simpleCurv^\delta\textnormal{ at }\simpleCurv^\delta(t\wedge \exitTT_{\eps}^\delta)]} ,
\end{align*}
where we denote by $\A_r^\delta\cup\simpleCurv^\delta$ the graph obtained by wiring $\simpleCurv^\delta$ and $\AnnBdryOut_r^\delta$ together.
Thus, the Radon--Nikodym derivative of the conditional law of the time-reversed branch $\cev{\gamma}_1^\delta[0,t\wedge \cev{\exitTT}_{\eps}^\delta]$ given $E(\A_r^\delta;x_1^\delta,x_2^\delta;\alpha_1^\delta,\alpha_2^\delta)$, with respect to the time-reversed LERW $\cev{\Gamma}_{x_1}^\delta[0,t\wedge \cev{\exitTT}_{\eps}^\delta]$, is given by 
\begin{align*}
M(t\wedge \cev{\exitTT}_{\eps}^\delta) 
:= \frac{\one{\{\cev{\Gamma}_{x_1}^\delta(0)\in\alpha_1^\delta\}}}{\PP^\delta \big[E(\A_r^\delta;x_1^\delta,x_2^\delta;\alpha_1^\delta,\alpha_2^\delta) \big]} 
\; \frac{\PP^\delta \big[E(\A_r^\delta\cup\cev{\Gamma}_{x_1}^\delta[0,t\wedge \cev{\exitTT}_{\eps}^\delta];x_1^\delta,x_2^\delta;\cev{\Gamma}_{x_1}^\delta(t\wedge \cev{\exitTT}_{\eps}^\delta),\alpha_2^\delta) \big]}{\Prob_{x_1}[\LR\textnormal{ hits }\AnnBdryOut_r^\delta\cup\cev{\Gamma}_{x_1}^\delta[0,t\wedge \cev{\exitTT}_{\eps}^\delta]\textnormal{ at }\cev{\Gamma}_{x_1}^\delta(t\wedge \cev{\exitTT}_{\eps}^\delta)]} .
\end{align*}
In particular, $(M(t\wedge \cev{\exitTT}_{\eps}^\delta) \colon t\ge 0)$ is a discrete martingale for $\cev{\Gamma}_{x_1}^\delta$. 
Define 
\begin{align*}
N(t\wedge \cev{\exitTT}_{\eps}^\delta)
:= \;& \frac{\PP^\delta \big[E(\A_r^\delta\cup\cev{\Gamma}_{x_1}^\delta[0,t\wedge \cev{\exitTT}_{\eps}^\delta];x_1^\delta,x_2^\delta;\cev{\Gamma}_{x_1}^\delta(t\wedge \cev{\exitTT}_{\eps}^\delta),\alpha_2^\delta) \big]}{\Prob_{x_1}[\LR\textnormal{ hits }\AnnBdryOut_r^\delta\cup\cev{\Gamma}_{x_1}^\delta[0,t\wedge \cev{\exitTT}_{\eps}^\delta]\textnormal{ at }\cev{\Gamma}_{x_1}^\delta(t\wedge \cev{\exitTT}_{\eps}^\delta)] 
\; \Prob_{x_2}[\LR\textnormal{ hits }\AnnBdryOut_r^\delta\cup\cev{\Gamma}_{x_1}^\delta[0,t\wedge \cev{\exitTT}_{\eps}^\delta]\textnormal{ through }\alpha_2^\delta]}\\
\;&\times \frac{\Prob_{x_2}[\LR\textnormal{ hits }\AnnBdryOut_r^\delta\cup\cev{\Gamma}_{x_1}^\delta[0,t\wedge \cev{\exitTT}_{\eps}^\delta]\textnormal{ through }\alpha_2^\delta]}{\Prob_{x_2}[\LR\textnormal{ hits }\AnnBdryOut_r^\delta\textnormal{ through }\alpha_2^\delta]}
\; \one{\{\cev{\Gamma}_{z}^\delta(0)\in\alpha_1^\delta\}}
\end{align*}
Then, $N = (N(t\wedge \cev{\exitTT}_{\eps}^\delta) \colon t\ge 0)$ is a discrete martingale for $\cev{\Gamma}_{x_1}^\delta$
--- and by a Harnack-type inequality, there exists a constant $C=C(\eps)>0$ that only depends on $\eps$, such that
\begin{align*}
N(t\wedge \cev{\exitTT}_{\eps}^\delta)
\; \le \; 1 \; + \; \frac{\Prob_{x_2}[\LR\textnormal{ hits }\AnnBdryOut_r^\delta\cup\cev{\Gamma}_{x_1}^\delta[0,t\wedge \cev{\exitTT}_{\eps}^\delta]\textnormal{ at }\cev{\Gamma}_{x_1}^\delta(t\wedge \cev{\exitTT}_{\eps}^\delta)] 
\; \Prob_{x_1}[\LR\textnormal{ hits }\AnnBdryOut_r^\delta\textnormal{ through }\alpha_2^\delta]}{\Prob_{x_1}[\LR\textnormal{ hits }\AnnBdryOut_r^\delta\cup\cev{\Gamma}_{x_1}^\delta[0,t\wedge \cev{\exitTT}_{\eps}^\delta]\textnormal{ at }\cev{\Gamma}_{x_1}^\delta(t\wedge \cev{\exitTT}_{\eps}^\delta)] 
\; \Prob_{x_2}[\LR\textnormal{ hits }\AnnBdryOut_r^\delta\textnormal{ through }\alpha_2^\delta]}
\; \le \; C.
\end{align*}
We have proved the convergence (both in $(X, \dist_X)$ and locally uniformly)  
of the time-reversed UST branch $\cev{\gamma}_1^\delta$ to the radial $\SLE_2$ curve $\slecurv = \cev{\gamma}_1$ 
in the proof of Proposition~\ref{prop::scaling_curve} (case $n=1$), when first $\delta \to 0$ and then $r \to 0$.
Similarly, the family $\{\cev{\Gamma}_{x_1}^\delta\}_{\delta>0}$ is tight in $(X, \dist_X)$
(cf.~\cite[Theorem~3.9]{LSW:Conformal_invariance_of_planar_LERW_and_UST} and~\cite[Lemma~3.2]{Berestycki-Liu:Piecewise_Temperleyan_dimers_and_multiple_SLE8}). 
Suppose $\cev{\Gamma}_r \in X$ is any subsequential limit, 
which we parameterise by capacity. It follows similarly as in the proof of~\cite[Theorem~3.9]{LSW:Conformal_invariance_of_planar_LERW_and_UST} 
that the family $\{\cev{\Gamma}_{x_1}^\delta\}_{\delta>0}$ is also tight in the uniform topology of curves parameterised by capacity. 
Thus, by Skorokhod's representation theorem, without loss of generality we may assume that along a subsequence still denoted $\delta \to 0$, 
the time-reversed LERW $\cev{\Gamma}_{x_1}^\delta$ converges almost surely to $\cev{\Gamma}_r$ in both $(X, \dist_X)$ and locally uniformly.

Consider the conformal isomorphism $\phi^{(r)}_{t,\eps} \colon \A_r\setminus\cev{\Gamma}_r[0,t\wedge \cev{\exitTT}_{\eps}] \to \A_{r_{t,\eps}}$, such that $\phi^{(r)}_{t,\eps}(r) =: r({t\wedge \cev{\exitTT}_{\eps}}) =: r_{t,\eps}$.
The arguments leading to Equation~\eqref{eqn::TV_conv_1} show that 
$\cev{\Gamma}_r[0,\cev{\exitTT}_{\eps}] \to \slecurv[0,\exitTT_{\eps}]$ as $r \to 0$ in total variation distance.
In particular, there exists a coupling such that almost surely, for every $t\ge 0$, the curve $\cev{\Gamma}_r[0,t]$ agrees with the radial $\SLE_2$ curve $\slecurv[0,t]$ when $r>0$ is small enough. 
Hence, $\phi^{(r)}_{t,\eps}$ converges to the radial Loewner flow $g_{t\wedge \exitTT_{\eps}}$ corresponding to $\slecurv$ in~\eqref{eqn::ODE} locally uniformly on $\D\setminus(\{0\}\cup\slecurv[0,t\wedge \exitTT_{\eps}])$.

Similarly to Lemma~\ref{lem::Poisson_limit}, one can prove that $N$ converges almost surely as $\delta\to 0$ to a limit $N_r$, which satisfies the following two properties.
\begin{itemize}[leftmargin=*]
\item
The process $N_r = (N_r(t\wedge \cev{\exitTT}_{\eps})\colon t\ge 0)$ is a martingale for $\cev{\Gamma}_r$, 
and we have $N_r(t\wedge \cev{\exitTT}_{\eps}) \le C$.

\item Let $\A_r(t,\eps) := \A_r\setminus\cev{\Gamma}_r[0,t\wedge \cev{\exitTT}_{\eps}]$. 
We can write $N$ in the form 
\begin{align}
\nonumber
\hspace*{-5mm}
N_r(t\wedge \cev{\exitTT}_{\eps}) 
=\;& \one{\{\cev{\Gamma}_r(0)\in\alpha_1\}}
\; r_{t,\eps} \; 
\exp\Big(\!-2\mu_{ \A_r(t,\eps) }^{\textnormal{ND}} \big[ \ell\in \LL_* \textnormal{ such that }  \tfrac{1}{2\pi}\ph(\ell)\textnormal{ is odd } \big] \Big)
\\
\nonumber
\;&\times \sin\left(\frac{\arg\phi^{(r)}_{t,\eps}(x_2)-\arg\phi^{(r)}_{t,\eps}(x_1)}{2}\right)
\; \sin\left(\frac{\arg\phi^{(r)}_{t,\eps}(e)-\arg\phi^{(r)}_{t,\eps}(\cev{\Gamma}_r(t\wedge \cev{\exitTT}_{\eps}))}{2}\right) 
\\
\label{eqn::mart_N}
\;&\times 
\frac{|\big(\phi^{(r)}_{t,\eps}\big)'(e)|}{\Pois_{\A_r}(x_2,e) \; \Pois_{\A_{r_{t,\eps}}}(\phi^{(r)}_{t,\eps}(x_1),\phi^{(r)}_{t,\eps}(\cev{\Gamma}_r(t\wedge \cev{\exitTT}_{\eps})))}
\; \times (1+O( r_{t,\eps})) ,
\end{align}
where $\Pois_{\smash{\A_{r}}}(\cdot,\cdot)$ is the Poisson kernel on $\smash{\A_{r}}$ with Dirichlet boundary condition on $\partial B(0,1)$ 
and Neumann boundary condition on $\partial B(0,r)$: the function $\Pois_{\smash{\A_{r}}}(\cdot,\cdot)$ is smooth on the set 
\begin{align*}
(\overline{\A}_{r} \times\overline{\A}_{r})\setminus\{(y,y) \;|\; y\in\partial B(0,1)\};
\end{align*}
for every $z\in\partial B(0,1)$, the function $\Pois_{\smash{\A_{r}}}(\cdot,y)$ is harmonic on $\smash{\A_{r}}$, and
\begin{align*}
\Pois_{\smash{\A_{r}}}(\cdot,y)|_{\partial B(0,1)\setminus\{y\}}=0
\qquad\textnormal{and}\qquad 
\partial_n \Pois_{\smash{\A_{r}}}(\cdot,z)|_{\partial B(0,r)}=0.
\end{align*}
\end{itemize}

Next, for $c>0$, define $\sigma_c$ to be the first time for the curves such that $r(\sigma)=cr$ 
(we can first take $r$ small enough and then take $\eps$ small enough such that $\sigma_c < \cev{\exitTT}_{\eps}$). 
By the optional stopping theorem, 
\begin{align}\label{eqn::mart_value}
\E[N_r(\sigma_c)] = \E[N_r(0)].
\end{align}
By a Harnack-type inequality, for every $\eps_0>0$, there exists $s_0>0$, such that for every 
$s<s_0$,
\begin{align}\label{eqn::kernel_estimate}
\big|\Pois_{\A_s}(x,e)-\tfrac{1}{2\pi}\big|<\eps_0 , \qquad x\in\partial B(0,s) , \; e\in\partial\D .
\end{align}
Plugging~\eqref{eqn::kernel_estimate} into~(\ref{eqn::mart_N},~\ref{eqn::mart_value}), we find that
\begin{align}\label{eqn::cons_aux_mart}
\;&\frac{1}{c}\;\sin\bigg(\frac{\arg x_2-\arg x_1}{2}\bigg)
\;\E \bigg[\sin\bigg(\frac{\arg e-\arg\cev{\Gamma}_r(0)}{2}\bigg) \; \one{\{\cev{\Gamma}_r(0)\in\alpha_1\}}\bigg]\notag\\
=\;&(1+o_r(1)) \;
\E\bigg[e^{-2F(\sigma_c,r)}\sin\bigg(\frac{\arg\phi^{(r)}_{\sigma_c,\eps}(x_2)-\arg\phi^{(r)}_{\sigma_c,\eps}(x_1)}{2}\bigg)\sin\left(\frac{\arg\phi^{(r)}_{\sigma_c,\eps}(e)-\arg\phi^{(r)}_{\sigma_c,\eps}(\cev{\Gamma}_r(\sigma_c))}{2}\right)
\notag\\
\;& \qquad\qquad\qquad\quad
\times \big| (\phi^{(r)}_{\sigma_c,\eps})'(e) \big| \; \one{\{\cev{\Gamma}_r(0)\in\alpha_1\}}\bigg],
\end{align}
where 
\begin{align*}
F(\sigma_c,r) := \mu_{\A_{r}}^{\textnormal{ND}} \big[ \ell\in \LL_* \textnormal{ such that }  \tfrac{1}{2\pi}\ph(\ell)\textnormal{ is odd and } \ell\cap\cev{\Gamma}_r[0,\sigma_c]\neq\emptyset \big] .
\end{align*}

Let us now gather the following observations. 
\begin{itemize}[leftmargin=*]
\item By the optional stopping theorem, for $z\in\partial B(0,r)$ and $1\geq R>r$, we have
\begin{align*}
(\phi^{(r)}_{\sigma_c,\eps})'(z) 
= (\phi^{(r)}_{\sigma_c,\eps})'(B^{\textnormal{ND}}(0)) 
= \ProbEBM_{z} \big[ (\phi^{(r)}_{\sigma_c,\eps})'(B^{\textnormal{ND}}(\exitT_R))
\big] ,
\end{align*}
where $B^{\textnormal{ND}} \sim \ProbBM_{z}$ is the Brownian motion on $\A_r$ started at $B^{\textnormal{ND}}(0) = z\in\partial B(0,r)$ 
with reflective inner boundary $\partial B(0,r)$ and stopped when hitting $\partial B(0,R)$ at the hitting time $\exitT_R$. 
Since $\phi_{\sigma_c,\eps}^{(r)}$ converges to the radial Loewner flow $g_{\log c}$ 
locally uniformly on $\D\setminus(\{0\}\cup\slecurv[0,\log c])$, we have
\begin{align*}
\lim_{r\to 0}\max_{z\in\partial B(0,r)}\big|(\phi^{(r)}_{\sigma_c,\eps})'(z) - \log c \big|=0.
\end{align*}
This implies that
\begin{align*}
\lim_{r\to 0}\sin\bigg(\frac{\arg\phi^{(r)}_{\sigma_c,\eps}(x_2)-\arg\phi^{(r)}_{\sigma_c,\eps}(x_1)}{2}\bigg) 
= \sin\bigg(\frac{\arg x_2-\arg x_1}{2}\bigg) .
\end{align*}

\item
By the uniform convergence, we also have
\begin{align*}
\lim_{r\to 0}\phi^{(r)}_{\sigma_c,\eps}( \slecurv(\sigma_c)) = g_{\log c}(\slecurv(c)) = \exp\big(\ii \Theta(\log c)\big) ,
\end{align*}
where $(\Theta(t) \colon t\ge 0)$ is the driving function of the curve $(\slecurv(t) \colon t\ge 0)$ as in~\eqref{eqn::ODE}. 

\item By It\^o's formula, the following process is a martingale for $\slecurv$:
\begin{align*}
S(t) := \ee^{\frac{3}{4}t} \, |g_t'(e)| \, \sin\bigg(\frac{\arg g_{t}(e)-\Theta(t)}{2}\bigg) .
\end{align*}

\end{itemize}
Plugging these observations into~\eqref{eqn::cons_aux_mart}, by letting $r\to 0$ and shrinking $\alpha_1$ to a point, 
and applying the optional stopping theorem to $S$, we find 
\begin{align*}
\frac{1}{c} \; \sin\bigg(\frac{\arg e-\Theta(0)}{2}\bigg)
= \; & \ee^{-2 F(c)} \; \E\bigg[\sin\bigg(\frac{\arg g_{\log c}(e)-\Theta(\log c)}{2}\bigg) \big| g'_{\log c}(e) \big| \bigg] 
\\
= \; & c^{-\frac{3}{4}} \, \ee^{-2 F(c)} \; \E\big[ S(\log c) \big] 
= c^{-\frac{3}{4}} \, \ee^{-2 F(c)} \; \sin\bigg(\frac{\arg e -\Theta(0)}{2}\bigg) ,
\end{align*}
which by Equation~\eqref{eq:Fc} implies the asserted exact identity~\eqref{eqn::scaling_constant_poly2}: 
\begin{align*}
F(c) = \mu_{\D} \big[ \ell\in \LL_* \textnormal{ such that }  \tfrac{1}{2\pi}\ph(\ell)\textnormal{ is odd and }\ell\cap\simpleCurv\neq\emptyset \big]
= \frac{1}{8} \log c .
\end{align*}
\end{proof}

{\small \bibliographystyle{alpha}

}
\end{document}